\documentclass{amsart}

\usepackage{amssymb}
\usepackage{amsrefs}
\usepackage{amsmath}
\usepackage{amsfonts}
\usepackage{amsthm}
\usepackage{array}
\usepackage{bm}
\usepackage{enumitem}
\usepackage{xcolor}

\usepackage{eucal}
\usepackage{upgreek}

\usepackage{hyperref}

\usepackage{tikz}
\usetikzlibrary{cd}

\hypersetup{
	colorlinks,
	linkcolor={red!50!black},
	citecolor={blue!50!black},
	urlcolor={blue!80!black}
}

\newtheorem{thm}{Theorem}[section]
\newtheorem{prop}[thm]{Proposition}
\newtheorem{lem}[thm]{Lemma}
\newtheorem{cor}[thm]{Corollary}
\newtheorem*{cor*}{Corollary}
\newtheorem*{thm*}{Theorem}

\theoremstyle{definition}
\newtheorem{definition}[thm]{Definition}
\newtheorem{example}[thm]{Example}

\theoremstyle{plain} 
\newcommand{\thistheoremname}{}
\newtheorem{genericthm}[thm]{\thistheoremname}

\newtheorem*{genericthm*}{\thistheoremname}
\newenvironment{namedthm*}[1]
{\renewcommand{\thistheoremname}{#1}%
	\begin{genericthm*}}
	{\end{genericthm*}}

\theoremstyle{remark}

\newtheorem{remark}[thm]{Remark}

\numberwithin{equation}{section}

\newcommand{\R}{\mathbb{R}}
\newcommand{\C}{\mathbb{C}}
\newcommand{\V}{\mathsf{V}}
\newcommand{\SO}{\mathrm{SO}}
\newcommand{\SU}{\mathrm{SU}}

\newcommand{\G}{\mathrm{G}}
\newcommand{\U}{\mathrm{U}}
\newcommand{\eps}{\varepsilon}

\newcommand{\LamW}{\operatorname{W}}

\renewcommand{\Re}{\operatorname{Re}}
\renewcommand{\Im}{\operatorname{Im}}

\begin{document}

\title{Quadratic closed $\G_2$-structures}

\author{Gavin Ball}
\address{\textsc{Universit\'{e} du Qu\'{e}bec \`{a} Montr\'{e}al}, 	\textsc{D\'{e}partement de math\'{e}matiques}, \textsc{Case postale 8888, succursale centre-ville, Montr\'{e}al (Qu\'{e}bec), H3C 3P8, Canada}}
\email{gavin.ball@cirget.ca}
\urladdr{https://www.gavincfball.com/} 

\begin{abstract}
This article studies closed $\G_2$-structures satisfying the quadratic condition, a second-order PDE system introduced by Bryant involving a parameter $\lambda.$ For certain special values of $\lambda$ the quadratic condition is equivalent to the Einstein condition for the metric induced by the closed $\G_2$-structure (for $\lambda = 1/2$), the extremally Ricci-pinched (ERP) condition (for $\lambda=1/6$), and the condition that the closed $\G_2$-structure be an eigenform for the Laplace operator (for $\lambda = 0$). Prior to the work in this article, solutions to the quadratic system were known only for $\lambda = 1/6,$ $-1/8,$ and $2/5,$ and for these values only a handful of solutions were known.

In this article, we produce infinitely many new examples of ERP $\G_2$-structures, including the first example of a complete inhomogeneous ERP $\G_2$-structure and a new example of a compact ERP $\G_2$-structure. We also give a classification of homogeneous ERP $\G_2$-structures. We provide the first examples of quadratic closed $\G_2$-structures for $\lambda = -1,$ $1/3,$ and $3/4,$ as well as infinitely many new examples for $\lambda = -1/8$ and $2/5.$ Our constructions involve the notion of special torsion for closed $\G_2$-structures, a new concept that is likely to have wider applicability.

In the final section of the article, we provide two large families of inhomogeneous complete steady gradient solitons for the Laplacian flow, the first known such examples.
\end{abstract}

\maketitle

\tableofcontents

\section{Introduction}

A $\G_2$-structure on a 7-manifold $M$ is given by a 3-form $\varphi \in \Omega^3 (M)$ satisfying a non-degeneracy condition. Such a 3-form gives rise to a metric $g_{\varphi},$ a volume form $\mathrm{vol}_\varphi,$ and a Hodge star operator ${*}_{\varphi}$ on $M.$ If the $\G_2$-structure satisfies the equations
\begin{align}\label{eq:IntroTorFree}
d \varphi = 0, \:\:\: d *_{\varphi} {\varphi} = 0,
\end{align}
then the metric $g_{\varphi}$ has holonomy contained in $\G_2$ \cite{BryExcept}, which implies that $g_{\varphi}$ is Ricci-flat and has a parallel spinor field. Conversely, if $\left( M, g \right)$ is a Riemannian 7-manifold with holonomy group contained in $\G_2,$ then $M$ carries a $g$-parallel $\G_2$-structure $\varphi$ satisfying equations (\ref{eq:IntroTorFree}) and such that $g=g_{\varphi}.$

A $\G_2$-structure satisfying the less restrictive condition
\begin{equation}
d \varphi = 0
\end{equation}
is called \emph{closed}. Closed $\G_2$-structures appear in a natural way in the study of manifolds with holonomy $\G_2.$ Currently the only known method for producing $\G_2$-structures satisfying equations (\ref{eq:IntroTorFree}) is to begin with a closed $\G_2$-structure $\varphi$ such that the 5-form $d \, {*}_{\varphi} {\varphi}$ is small in a suitable sense, then use a result of Joyce \cite{Joyce96} to perturb $\varphi$ to a nearby $\G_2$-structure satisfying equations (\ref{eq:IntroTorFree}).

Closed $\G_2$-structures also play a pivotal role in attempts to use a geometric flow to understand the solutions of (\ref{eq:IntroTorFree}). The \emph{Laplacian flow} \cite{Bry05} is the geometric flow for $\G_2$-structures defined by the equation
\begin{equation}\label{eq:IntroLapFlow}
\tfrac{d}{dt} \varphi = \Delta_{\varphi} \varphi,
\end{equation}
where $\Delta_{\varphi}$ is the Laplacian induced by $\varphi$.  The closed condition $d \varphi = 0$ is preserved along the Laplacian flow, and so closed $\G_2$-structures are thought to provide suitable initial conditions for the flow. Geometric and analytic properties of the Laplacian flow have been studied in several works \cites{LotWeiAna,LotWeiLapShi,LotWeiStab,FiYaHype18,HuWa18,BryXuLapFlow11}, but the long term behaviour of solutions remains largely mysterious.

In this article we study closed $\G_2$-structures satisfying a particular second-order equation, the quadratic condition. When $\varphi$ is closed, the exterior derivative of the 4-form ${*}_{\varphi} \varphi$ is given by $d {*}_{\varphi} \varphi = \tau \wedge \varphi,$ for a unique 2-form $\tau$ called the \emph{torsion form} of $\varphi.$ The torsion form $\tau$ is a fundamental invariant of a closed $\G_2$-structure, and the quadratic condition is formulated in terms of its exterior derivative $d \tau.$

\begin{definition}
	Let $\lambda$ be a constant. A closed $\G_2$-structure $\varphi$ is called \emph{$\lambda$-quadratic} if it satisfies
	\begin{equation}\label{eq:QuadClosed}
	d \tau = \tfrac{1}{7} | \tau |^2 \varphi + \lambda \left( \tfrac{1}{7} | \tau |^2 \varphi + {*}_{\varphi} \left( \tau \wedge \tau \right) \right) .
	\end{equation}
\end{definition}

The $\lambda$-quadratic condition was introduced by Bryant \cite{Bry05} as the most general way in which the 3-form $d \tau$ can be written quadratically in terms of the components of $\tau.$ There are several special values of $\lambda$ of particular interest:
\begin{itemize}
	\item The $1/2$-quadratic condition is equivalent to the Einstein condition for the induced metric $g_{\varphi}.$
	\item The induced metric $g_{\varphi}$ of a closed $\G_2$-structure on a compact manifold $M$ always satisfies the inequality 
	\begin{equation}
	\int_M | \, \mathrm{Ric}^{0} (g_{\varphi}) |^2 \, \mathrm{vol}_{\varphi} \geq \frac{4}{21} \int_M \, \mathrm{Scal}(g_{\varphi})^2 \, \mathrm{vol}_\varphi,
	\end{equation}
	due to Bryant \cite{Bry05}, and equality holds if and only if $\varphi$ is $1/6$-quadratic. Hence $1/6$-quadratic closed $\G_2$-structures are given the name \emph{extremally Ricci-pinched}, or \emph{ERP} for short.
	\item A closed $\G_2$-structure is $0$-quadratic if and only if $\varphi$ is an \emph{eigenform}, i.e. an eigenfunction for $\Delta_{\varphi}.$ These structures are to the Laplacian flow as Einstein metrics are to the Ricci flow.
	\item A closed $\G_2$-structures $\varphi$ inducing a conformally flat metric $g_{\varphi}$ is always $-1/8$-quadratic. Closed $\G_2$-structures satisfying this condition have been classified by the author \cite{BallConfFlat20}.
\end{itemize}

There are very few examples of $\lambda$-quadratic closed $\G_2$-structures appearing in the literature. Bryant \cite{Bry05} gives one example of a homogeneous ERP structure, Lauret \cite{LauretLap} another, and Lauret--Nicolini \cites{LauNicERPLI20,LauNicERPLICAG20} have classified left-invariant ERP structures on Lie groups, providing three other distinct examples. The only known inhomogeneous examples of $\lambda$-quadratic closed $\G_2$-structures in the literature are two examples each for $\lambda = -1/8$ and $\lambda = 2/5$ given by the author \cite{BallConfFlat20} in a previous article.

Bryant \cite{Bry05} has shown that the only possibility for $\lambda$ on a compact manifold $M$ is $\lambda = 1/6.$ Both Bryant's homogeneous ERP structure \cite{Bry05} and Lauret's first example \cites{LauretLap,KaLaERP20} admit compact quotients.

\subsection{Results and methods} We approach the study of $\lambda$-quadratic closed $\G_2$-structures in this article using the techniques of exterior differential systems and the moving frame.

After providing necessary background in \S\ref{sect:Prelim}, we begin in \S\ref{sect:QuadG2} by studying $\lambda$-quadratic closed $\G_2$-structures from the point of view of exterior differential systems. This is a natural approach to use, because the system (\ref{eq:QuadClosed}), viewed as a non-linear second order PDE system for a closed section $\varphi$ of $\Omega^3(M),$ is highly overdetermined. We show in \S\ref{sssect:CharVar} that this PDE system is elliptic modulo diffeomorphism. However, it turns out that the system is not involutive, even after a prolongation. In particular, there are several non-trivial conditions that the 2-jet of a quadratic closed $\G_2$-structure must satisfy that are not algebraic consequences of (\ref{eq:QuadClosed}), as we show in \S\ref{sssect:NonInvol}.

The non-involutivity of the system (\ref{eq:QuadClosed}) is of a sufficiently complicated nature that the existence of any examples becomes an interesting question. Motivated by the form of the conditions derived in \S\ref{sssect:NonInvol}, in \S\ref{sect:U2pl} and \S\ref{sect:U2mi} we study quadratic closed $\G_2$-structures satisfying an additional first order condition, which we describe now.

\subsubsection{Special torsion}
The torsion 2-form $\tau$ of a closed $\G_2$-structure $\varphi$ takes values in a 14-dimensional subbundle of $\Lambda^2 T^* M$ that is modeled on the Lie algebra $\mathfrak{g}_2.$ Under the adjoint action of $\G_2,$ the generic element of $\mathfrak{g}_2$ is stabilised by a maximal torus, but there are two exceptional $\G_2$-orbits in $\mathfrak{g}_2$ consisting of elements stabilised by subgroups of $\G_2$ isomorphic to $\U(2).$ The elements $\beta$ of one of these $\G_2$-orbits, thought of as 2-forms $\beta,$ satisfy $\beta^3 = 0$ and we call these elements \emph{positive}, while the elements of the other $\G_2$-orbit satisfy $|\beta^3|^2 = \tfrac{2}{3} | \beta |^6$ and we call these elements \emph{negative}. The stabiliser of a positive element is not conjugate in $\G_2$ to the stabiliser of a negative element, despite the fact that both groups are isomorphic to $\U(2).$

\begin{definition}[\S\ref{ssect:SpecTors}]
	A closed $\G_2$-structure $\varphi$ said to have \emph{special torsion} of \emph{positive (resp. negative) type} if the torsion 2-form $\tau$ of $\varphi$ is positive (resp. negative).
\end{definition}

In \S\ref{sect:U2pl} and \S\ref{sect:U2mi}, we study the $\lambda$-quadratic condition under the assumption that $\varphi$ has special torsion. This assumption simplifies considerably the conditions derived in \S\ref{sssect:NonInvol} and allows us to prove powerful existence and classification results under certain natural conditions.

\subsubsection{Positive type}

In \S\ref{sect:U2pl}, we study $\lambda$-quadratic closed $\G_2$-structures with special torsion of positive type. An identity due to Bryant (\ref{eq:Bry466}) implies that the only possibility for $\lambda$ in this case is $1/6,$ so every structure of this type is ERP. We also note that it follows from Bryant's \cite{Bry05} work that an ERP $\G_2$-structure on a compact manifold $M$ has special torsion of positive type.

When a closed $\G_2$-structure $\varphi$ on $M$ has special torsion of positive type, the torsion form $\tau$ defines a $\U(2)$-structure on $M$ canonically associated to $\varphi.$ The ERP condition places strong restrictions on the geometry of this $\U(2)$-structure, as we show in \S\ref{ssect:U2plInduced}. It turns out that when $\varphi$ is ERP the torsion of the induced $\U(2)$-structure is forced to take values in a vector space of real dimension 53, whereas the torsion of a generic $\U(2)$-structure in dimension 7 takes values in a vector space of real dimension 119. The $\U(2)$-irreducible decomposition of the torsion space of an ERP $\U(2)$-structure allows us to define three natural tensors on such a 7-manifold $M,$ which we denote by $A,$ $N,$ and $S,$ taking values in vector bundles of ranks 12, 4, and 10 respectively.

In \S\ref{ssect:TypeA}-\ref{ssect:TypeS}, we study the cases where only one of the tensors $A,$ $N,$ or $S$ are nonvanishing (when all three vanish the structure is locally equivalent to Bryant's example, as we show in \S\ref{sssect:BryEx}). We say that an ERP $\U(2)$-structure is of \emph{type $A$ (resp. $N$, $S$)} if $A$ (resp. $N,$ $S$) is the only nonvanishing tensor out of $A,$ $N,$ and $S.$

For structures of type $A$ we have the following result, the details of which are given in Theorem \ref{thm:TypeAWeier} and \S\ref{ssect:TypeA}.

\begin{thm*}\label{thm:ERPTypeA}
	There exists a Weierstrass-type formula for ERP $\U(2)$-structures of type $A$. Given a holomorphic function of one complex variable, one can construct an ERP closed $G_2$-structure of type $A$---and conversely, every such structure can be locally described in this way.
\end{thm*}

Structures of type $N$ are studied in \S\ref{ssect:TypeN}. It is shown that every such structure arises as the bundle of compatible metrics over a 4-manifold $X$ carrying an $\mathrm{S}^1 \cdot \mathrm{SL}(2, \C)$-structure of a special type. The fibres of this bundle are all isometric to hyperbolic 3-space. Unfortunately, we are not able to fully understand the 4-manifolds $X$ carrying these $\mathrm{S}^1 \cdot \mathrm{SL}(2, \C)$-structures. However, in \S\ref{sssect:TypeNwzero} we identify a subclass of these structures that admit a Weierstrass formula similar to the type $A$ case. This is recorded in Theorem \ref{thm:TypeNwzero}.

In \S\ref{ssect:TypeS}, we study structures of type $S.$ We show that these structures are all foliated by flat coassociative submanifolds. Baraglia \cite{BaragSemi} has studied general closed $\G_2$-structures $\varphi$ fibred by flat coassociative tori and shown that such structures are equivalent to 3-dimensional space-like submanifolds of the pseudo-Euclidean space $\R^{3,3}.$ The condition that $\varphi$ be torsion-free is equivalent to the 3-submanifold being maximal (i.e. having vanishing mean curvature vector). In the ERP case we prove the following analogous result, the details of which are given in Theorem \ref{thm:TypeSCons}.

\begin{thm*}
	An ERP $\U(2)$-structure of type $S$ is equivalent to a maximal space-like submanifold of the 5-dimensional quadric of vectors of norm $-1$ in $\R^{3,3}.$
\end{thm*}

Every previously known example of an ERP closed $\G_2$-structure turns out to be of type $S.$ In \S\ref{sssect:TypeSEgs} we study the known examples from this point of view and give new examples. In particular, in Example \ref{eg:M3} we demonstrate the existence of a complete inhomogeneous ERP closed $\G_2$-structure, the first such example known.

\begin{thm*}
	There exists a complete ERP closed $\G_2$-structure which is not homogeneous.
\end{thm*}

The structures of type $S$ are also interesting from the point of view of the Laplacian flow (\ref{eq:IntroLapFlow}). We prove the following result in \S\ref{sssect:LapFlow}, where it appears as Theorem \ref{thm:TypeSLapSol}.
\begin{thm*}
	The universal cover of an ERP $\U(2)$-structure of type $S$ is a steady soliton for the Laplace flow.
\end{thm*}
This result provides an explanation for the observation of Lauret--Nicolini that all of the left-invariant ERP closed $\G_2$-structures appearing in their classification (which are all of type $S$) are steady Laplace solitons.

In \S\ref{ssect:ERPHomog}, we give a classification result for homogeneous ERP closed $\G_2$-structures. 
\begin{thm*}\label{thm:ERPHomogClassIntro}
	An ERP closed $\G_2$-structure $\varphi$ admitting a transitive action of diffeomorphisms preserving $\varphi$ is equivalent, up to rescaling, to one of the five examples listed in Theorem \ref{thm:ERPHomogClass}.
\end{thm*}
Each of the examples presented in Theorem \ref{thm:ERPHomogClass} is equivalent to an example in the Lauret--Nicolini classification of left-invariant ERP closed $\G_2$-structures on Lie groups, so there are no essentially new examples appearing in the classification. On the other hand, for some of their examples our classification reveals the existence of additional automorphisms, showing that these $\G_2$-structures are better presented as homogeneous spaces with non-trivial isotropy rather than as Lie groups. One consequence of this point of view is the existence of a compact manifold admitting an ERP $\G_2$-structure locally equivalent to the Lauret--Nicolini example $\G_{M3},$ see \S\ref{sssect:Compact}. This yields a new example of a compact ERP $\G_2$-structure not locally equivalent to the examples of Bryant \cite{Bry05} and Kath--Lauret \cite{KaLaERP20}.

Taken together, the results of \S\ref{sect:U2pl} provide many new examples of ERP closed $\G_2$-structures and shed fresh light on the existing examples.

\subsubsection{Negative type}

In \S\ref{sect:U2mi}, we study $\lambda$-quadratic closed $\G_2$-structures with special torsion of negative type. We follow the same approach as in \S\ref{sect:U2pl}, studying the geometry of the $\U(2)$-structure on $M$ induced by the torsion form $\tau$ of $\varphi.$ The subgroup $\U(2)$ appearing in this section is not conjugate to the subgroup $\U(2)$ of \S\ref{sect:U2pl}, and this lends the constructions of this section a different flavour.

Our first results, Theorems \ref{thm:U2miLamVals} and \ref{thm:U2MiIncomp}, restrict the possible values of $\lambda$ in this situation to a finite set and show that complete examples of this type do not exist.

\begin{thm*}
	Let $\left(M, \varphi \right)$ be a $\lambda$-quadratic closed $\G_2$-structure with special torsion of negative type. Then $\lambda$ is equal to $-1,$ $-1/8,$ $2/5,$ or $3/4.$
\end{thm*}

\begin{thm*}\label{thm:U2MiIncompIntro}
	Let $\left(M, \varphi \right)$ be a $\lambda$-quadratic closed $\G_2$-structure with special torsion of negative type. Then the induced metric $g_{\varphi}$ is incomplete.
\end{thm*}

It turns out (see Proposition \ref{prop:U2miRedToSix}) that the geometry of $M$ is determined completely by the geometry of any of the 6-manifolds $N$ defined by $\left\lbrace | \tau |^2 = \mathrm{const} \right\rbrace.$ In \S\ref{ssect:LamMin1}-\ref{ssect:Lam3by4} we study the geometry of $N$ in each of the four possible cases $\lambda = -1,$ $-1/8,$ $2/5,$ or $3/4.$ Existence of $\lambda$-quadratic closed $\G_2$-structures for each of these $\lambda$-values follows from the results in these sections.

\begin{thm*}
	There exist $\lambda$-quadratic closed $\G_2$-structures with special torsion of negative type for each of the possible $\lambda$ values $-1,$ $-1/8,$ $2/5,$ or $3/4.$ Furthermore, there exist infinitely many local examples of these structures for each of the possible values of $\lambda$.
\end{thm*}

The results of \S\ref{sect:U2mi} expand considerably the set of known values of $\lambda$ for which $\lambda$-quadratic closed $\G_2$ structures can exist and show that for these $\lambda$ values there are relatively many local solutions to the system (\ref{eq:QuadClosed}). We note that in \S\ref{ssect:NonSpecTors} we construct a single example of a $1/3$-quadratic closed $\G_2$-structure that does not have special torsion. This is noteworthy because, by the results of \S\ref{sect:U2pl} and \S\ref{sect:U2mi}, the value $\lambda=1/3$ cannot occur in the special torsion cases.

\subsubsection{Laplace solitons}

A closed $\G_2$-structure $\varphi$ on a 7-manifold $M$ is said to be a \emph{Laplace soliton} if there exists a constant $c \in \R$ and a vector field $V \in \mathcal{X}(M)$ such that
\begin{equation}
\Delta_{\varphi} \, \varphi = c \, \varphi + \mathcal{L}_V \, \varphi.
\end{equation}
If $c$ is positive, zero, or negative, the soliton is called \emph{expanding}, \emph{steady}, or \emph{shrinking} respectively. When the vector field $V$ is the gradient of a function $f$ on $M$ the soliton is called a \emph{gradient soliton}.

Laplace solitons are exactly the solutions of the Laplacian flow equation (\ref{eq:IntroLapFlow}) that evolve under the flow by diffeomorphisms and scaling, and it is expected that they will play an important role in understanding the long-term behaviour of general solutions.

In \S\ref{sect:LapSols}, we give two families of examples of complete steady gradient Laplace solitons. These examples are inhomogeneous, unlike all of the other previously known examples of Laplace solitons \cites{FFMNilSols16,FiRa217,FiRa119,LauFirSol17,LauretLap,NicShrinkSol20}, and to our knowledge they are the first known examples of gradient solitons. Our examples are topologically products $\R \times N,$ where in the first family of examples $N$ is the twistor space over an anti-self-dual Ricci-flat 4-manifold, and in the second family $N$ is a certain $\mathrm{T}^2$-bundle over a suitable hyperk\"ahler 4-manifold. Both families of examples consist of infinitely many distinct structures, even when $N$ is assumed compact. The form of the $\G_2$-structure on $\R \times N$ for these examples is motivated by the results of \S\ref{sect:U2mi}.

The first family of examples is discussed in \S\ref{ssect:LapSolEg1} and the second family is discussed in \S\ref{ssect:LapSolEg2}. In both cases, the $\G_2$-structure is given explicitly in terms of the geometry of $N$ and functions of the coordinate in the $\R$-direction. We also give a description of the asymptotics of the examples as we move towards negative and positive infinity in the $\R$-direction.

\subsection{Acknowledgments} The majority of the work in this article appeared in my 2019 Duke University PhD thesis \cite{Ball19}. I thank my advisor Robert Bryant for his encouragement and for many helpful discussions. I would also like to thank the Simons Foundation for funding as a graduate student member of the Simons Collaboration on Special Holonomy in Geometry, Analysis and Physics during my graduate studies. I thank Jorge Lauret for helpful comments on a previous version of this article, and in particular for raising the possibility (fulfilled in \S\ref{sssect:Compact}) that the Lauret--Nicolini example on $\G_{M3}$ may admit a compact quotient.

\section{Preliminaries}\label{sect:Prelim}

This section provides background on the theory of $\G$-structures and the group $\G_2$ that will be used in the rest of the article.

\subsection{$\G$-structures}

We give here a brief description of the concept of a $\G$-structure, as most of the calculations and results in this article will be phrased in this language.

\begin{definition}
	Let $M$ be an $n$-dimensional manifold and $W$ a vector space of dimension $n.$ A \emph{coframe} at $p \in M$ is a linear isomorphism $u_p : T_p M \to W.$ The set of all coframes on $M,$ endowed with the group action $u \cdot g = g^{-1} \circ u,$ is a principal right $\mathrm{GL}(W)$-bundle over $M$, called the \emph{coframe bundle} of $M$ and denoted by $\mathcal{F} \left( M \right).$
\end{definition}

\begin{definition}\label{def:GStruct}
	Let $\G$ be a subgroup of $\mathrm{GL}(W).$ A \emph{$\G$-structure} on a manifold $M$ is a principal $\G$-subbundle of the coframe bundle $\mathcal{F} \left( M \right).$
\end{definition}

An important feature of the coframe bundle $\mathcal{F}(M)$ that is inherited by any $\G$-structure over $M$ is the existence of a canonical $V$-valued 1-form $\omega.$

\begin{definition}
	For any $\G$-structure $\mathcal{B} \subset \mathcal{F}(M),$ the \emph{tautological} 1-form $\omega$ is defined by
	\begin{align}
	\omega \left( v \right) = u \left( \pi_{*} \left( v \right) \right) \:\:\:\: \text{for all} \:\:\:\: v \in T_u \mathcal{B}, 
	\end{align}
	where $\pi : \mathcal{B} \to M$ is the bundle projection.
\end{definition}

The tautological form $\omega$ on $\mathcal{B}$ may be thought of as the pullback of the tautological form on the coframe bundle $\mathcal{F}(M)$ via the inclusion $\mathcal{B} \subset \mathcal{F}(M).$ The form $\omega$ is $\pi$-semibasic and has the reproducing property $\eta^{*} (\omega ) = \eta$ for any local section $\eta$ of $\mathcal{B}.$ The most important property of $\omega$ is that it detects diffeomorphisms of the base manifold $M$ respecting the $\G$-structure $\mathcal{B}.$

\begin{definition}
	If $M$ and $N$ are smooth $n$-manifolds and $f : M \to N$ is a local diffeomorphism, then a smooth bundle map $\tilde{f} : \mathcal{F} (M) \to \mathcal{F} (N)$ covering $f$ is defined by the rule
	\begin{align}
	\tilde{f} \left( u \right) = u \circ f^{-1}_*.
	\end{align}
	The map $f \mapsto \tilde{f}$ is functorial and covariant, and $\tilde{f}$ is a diffeomorphism.
\end{definition}

\begin{thm}\label{thm:MoE}
	If $f : M \to N$ is a diffeomorphism and $\mathcal{B}$ and $\mathcal{C}$ are $G$-structures over $M$ and $N$ respectively satisfying $\tilde{f} ( \mathcal{B} ) = \mathcal{C},$ then $\tilde{f}^{*} \omega_{\mathcal{B}} = \omega_{\mathcal{C}}.$
	
	Conversely, if $\mathcal{U} \subset \mathcal{B}$ is an open subset of a $G$-structure on $M$ with the property that its $\pi$-fibres are connected and $\phi : \mathcal{U} \to \mathcal{C}$ is any smooth mapping satisfying $\phi^{*} \omega_{\mathcal{C}} = \omega_{\mathcal{B}},$ then there exists a unique smooth mapping $f : \pi \left( \mathcal{U} \right) \to N$ that satisfies $f \circ \pi_{\mathcal{B}} = \pi_{\mathcal{C}} \circ g.$ Moreover, $f$ is a local diffeomorphism and $g$ is the restriction to $\mathcal{U}$ of $\tilde{f}.$
\end{thm}

Theorem \ref{thm:MoE} is the foundation of Cartan's method of equivalence. It is often used in conjunction with the following classical result, known as \textit{Cartan's Theorem on Maps into Lie Groups}:

\begin{thm} \label{thm:MaurerCartan} \cites{IvLaSecond, EDSBook} Let $\G$ be a Lie group with Lie algebra $\mathfrak{g}$, and let $\omega \in \Omega^1(\G; \mathfrak{g})$ be the Maurer-Cartan form on $\G$.  Let $P$ be a connected, simply-connected manifold admitting a $\mathfrak{g}$-valued $1$-form $\mu \in \Omega^1(P; \mathfrak{g})$ satisfying $d\mu = -\mu \wedge \mu$.  Then there exists a map $F \colon P \to \G$, unique up to composition with left-translation in $\G$, such that $F^*\omega = \mu$.
\end{thm}

\subsection{The group $\G_2$}

Let $V = \R^7,$ and let $e^1,...,e^7$ denote the canonical basis of $V^{*}$. Using the shorthand notation $e^{ijk}=e^i \wedge e^j \wedge e^k$ for wedge products, the element
$$\phi = e^{123}+e^{145}+e^{167}+e^{246}-e^{257}-e^{347}-e^{356}$$
will be called the \emph{standard 3-form on V}, and the group $\G_2$ is defined by
$$\G_2 = \{ A \in \mathrm{GL}(V) \mid A^* \phi = \phi \}.$$

The action of $\G_2$ on $V$ preserves the metric
$$g_{\phi}=\left(e^1\right)^2+\left(e^2\right)^2+\left(e^3\right)^2+\left(e^4\right)^2
+\left(e^5\right)^2+\left(e^6\right)^2+\left(e^7\right)^2$$
and volume form
$$\mathrm{vol}_{\phi} = e^1 \wedge e^2 \wedge e^3 \wedge e^4 \wedge e^5 \wedge e^6 \wedge e^7,$$
and it follows that $\G_2$ is a subgroup of $\SO(7)$. The Hodge star operator determined by $g_{\phi}$ and $\mathrm{vol}_\phi$ is denoted by $*_\phi.$ Note that $\G_2$ also fixes the 4-form
\begin{equation*}
*_\phi \phi = e^{4567} + e^{2367} + e^{2345} + e^{1357} - e^{1346} - e^{1256} - e^{1247}.
\end{equation*}

The $\mathrm{GL}(V)$-orbit of $\phi$ is open in $\Lambda^3 V^*,$ and will be denoted by $\Lambda^3_+ V^*.$

\subsubsection{Bryant's $\varepsilon$ symbol}

When working with the group $\G_2$ it is often very convenient to use an $\varepsilon$-notation introduced by Bryant \cite{Bry05}. Let $\varepsilon$ denote the unique symbol that is skew-symmetric in three or four indices and satisfies
\begin{subequations}
	\begin{align}
	\phi &= \tfrac{1}{6} \varepsilon_{ijk} e^{ijk}, \\
	{*}_{\phi}\phi &= \tfrac{1}{24} \varepsilon_{ijkl} e^{ijkl}.
	\end{align}
\end{subequations}
The $\varepsilon$ symbol satisfies various useful identities \cite{Bry05}.

\subsubsection{Representation theory of $\G_2$}\label{sect:G2reps}

The group $\G_2$ is a compact simple Lie group of rank two. Thus, each irreducible representation of $\G_2$ is indexed by a pair of integers $(p,q)$ corresponding to the highest weight of the representation with respect to a fixed maximal torus in $\G_2$ endowed with a fixed base for its root system. The irreducible representation associated to $(p,q)$ is denoted by $\mathsf{V}_{p,q}.$

The fundamental representation $\mathsf{V}_{1,0}$ is the standard representation $V = \R^7$ used to define the group $\G_2.$ In fact, the representation $\V_{p,0}$ for $p \geq 0$ is isomorphic to $\mathrm{Sym}^p_0(\V_{1,0}),$ so these irreducible $\SO(7)$-representations remain irreducible when restricted to $\G_2.$

The other fundamental representation of $\G_2$, $\V_{0,1}$, is isomorphic to the adjoint representation $\mathfrak{g}_2.$ The Lie algebra $\mathfrak{g}_2$ may be defined in components using the $\varepsilon$ symbol as
\begin{equation}
\mathfrak{g}_2 = \left\lbrace a_{ij} e_{i} \otimes e^j \mid a_{ij}=-a_{ji}, \:\: \varepsilon_{ijk}a_{jk}=0 \right\rbrace.
\end{equation}

The only other $\G_2$-representations that will appear in this work are $\V_{1,1}$ and $\V_{0,2}.$ We have
\begin{align}
\V_{0,2} &= \left\lbrace s_{ijkl} e_{i} \otimes e^j \otimes e_{k} \otimes e^l \: \vline\ \: \begin{aligned}
& s_{ijkl} = -s_{jikl}, \: s_{ijkl} = s_{klij}, \\
& \varepsilon_{ijk} s_{jklm} =0, \: s_{ijik} = 0, \: s_{ijjk}=0
\end{aligned} \right\rbrace, \\
\V_{1,1} &= \left\lbrace c_{ijk} e_i \otimes e_j \otimes e^k \mid \: c_{ijk} = - c_{ikj}, \: \eps_{mjk} c_{ijk} = 0, \: \eps_{mij} c_{ijk} = 0 \right\rbrace.
\end{align}

The exterior powers of the standard representation $V$ decompose as follows
\begin{subequations}\label{eqs:FormDecomp}
	\begin{align}
	\Lambda^2 (V^*) & = \Lambda^2_7  \oplus  \Lambda^2_{14},\\
	\Lambda^3 (V^*) & = \Lambda^3_1  \oplus  \Lambda^3_7 \oplus \Lambda^3_{27},\\
	\Lambda^4 (V^*) & = \Lambda^4_1  \oplus  \Lambda^4_7 \oplus \Lambda^4_{27},\\
	\Lambda^5 (V^*) & = \Lambda^5_7  \oplus  \Lambda^2_{14},
	\end{align}
\end{subequations}
where
\begin{subequations}
	\begin{align}
	\Lambda^2_7  &= \left\lbrace *_{\phi} \left( \alpha \wedge *_{\phi} \phi \right) \mid \alpha \in \Lambda^1 (V^*) \right\rbrace \cong V \cong \V_{1,0} \\
	\Lambda^2_{14}  &= \left\lbrace \beta \in \Lambda^1 (V^*) \mid \beta \wedge \phi = 2 *_{\phi} \beta \right\rbrace \cong \mathfrak{g}_2 \cong \V_{0,1}, \\
	\Lambda^3_1  &= \left\lbrace r \phi \mid r \in \mathbb{R} \right\rbrace \cong \mathbb{R} \cong \V_{0,0} \\
	\Lambda^3_7  &= \left\lbrace *_{\phi} \left( \alpha \wedge \phi \right) \mid \alpha \in \Lambda^1 (V^*) \right\rbrace \cong V \cong \V_{1,0}, \\
	\Lambda^3_{27}  &= \left\lbrace \gamma \in \Lambda^3 (V^*) \mid \gamma \wedge \phi = 0, \gamma \wedge *_{\phi} \phi = 0 \right\rbrace \cong \text{Sym}^2_0 (V) \cong \V_{2,0}, 
	\end{align}
\end{subequations}
and the Hodge star gives an isomorphism $\Lambda^p_j \cong \Lambda^{7-p}_j.$

\subsubsection{The adjoint representation}\label{ssect:G2adjoint}

One can easily check that
\begin{align}
\mathfrak{t} = \left\lbrace -2 f e_2 \otimes e^3 + \left( f+g \right) e_4 \otimes e^5 + \left(f-g \right) e_6 \otimes e^7 \mid f, g \in \R \right\rbrace
\end{align}
is an abelian subalgebra of $\mathfrak{g}_2,$ and is therefore a maximal torus.

Now, by Cartan's Theorem on maximal tori, every element $\beta$ of $\Lambda^2_{14} \cong \mathfrak{g}_2^{\flat}$ is conjugate to an element of $\mathfrak{t}^{\flat},$ so $\beta$ is conjugate to
\begin{align}
\tau = -2 f e_2 \wedge e_3 + \left( f+g \right) e_4 \wedge e_5 + \left(f-g \right) e_6 \wedge e_7,
\end{align}
for some constants $f, g \in \R.$ The stabiliser of an element of the maximal torus under the adjoint action of $\G_2$ depends on the value of $\left( f, g \right).$ 

When $f \left(f+g \right) \left( f-g \right) g \left( 3f +g \right) \left( 3 f -g \right) \neq 0,$ we are in the generic case, and the identity component of the stabiliser of $\tau$ is just the $\mathrm{T}^2$ subgroup of $\G_2$ obtained by exponentiating $\mathfrak{t}.$

When $f \left(f+g \right) \left( f-g \right) g \left( 3f +g \right) \left( 3 f -g \right) = 0,$ the stabiliser is isomorphic to $\U(2).$ There are two cases. We assume $\tau$ is non-vanishing.
\begin{itemize}
	\item If $f \left(f+g \right) \left( f-g \right) = 0,$ then $\tau$ is conjugate to a multiple of $e_4 \wedge e^5 - e_6 \wedge e^7.$ Denote the stabiliser of this element by $\U(2)^{+},$ and its $\G_2$-orbit by $\mathcal{O}^{-}.$
	\item If $g \left( 3 f + g \right) \left( 3 f -g \right) = 0,$ then $\tau$ is conjugate to a multiple of the 2-form $-2 \, e_2 \wedge e_3 + e_4 \wedge e_5 + e_6 \wedge e_7.$ Denote the stabiliser of this element by $\U(2)^{-},$ and its $\G_2$-orbit by $\mathcal{O}^{+}.$
\end{itemize}

The groups $\U(2)^+$ and $\U(2)^-$ are not conjugate in $\G_2.$ For example, $\U(2)^{-}$ preserves the vector $e_1,$ while the action of $\U(2)^{+}$ on $V \cong \R^7$ does not preserve any 1-dimensional subspace.

The maximal torus $\mathfrak{t}$ may be identified with the root space of $\mathfrak{g}_2.$ Figure \ref{G2Root} is a root diagram for $\mathfrak{g}_2,$ with stabilisers marked.

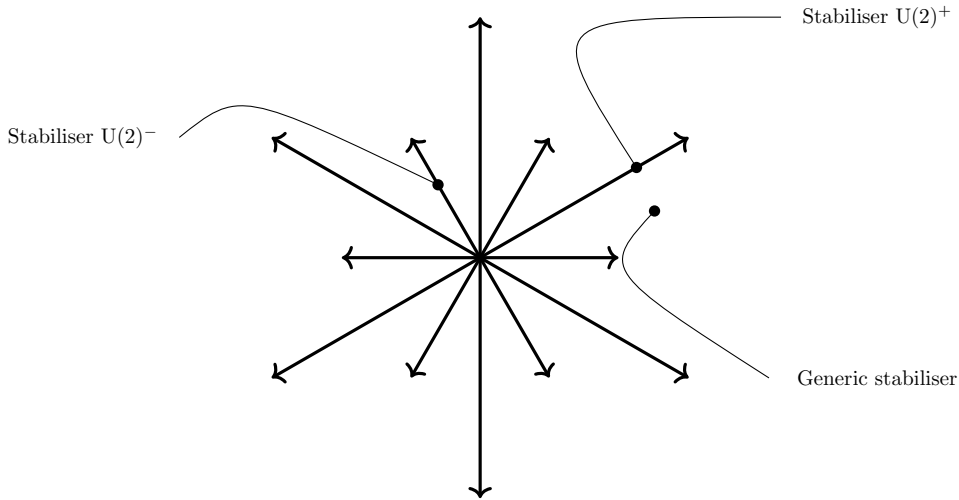
\begin{figure}
	\begin{tikzpicture}[scale=1.6]
	\draw[very thick,<->] (210:2) -- (30:2) node[anchor=north west,black]{};
	\draw[very thick,<-] (150:2) -- (0,0);
	\draw[very thick,->] (0,0) -- (0,2) node[black,anchor=south]{};
	
	\draw[very thick,->] (0,0) -- (-30:2);
	\draw[very thick,->] (0,0) -- (0,-2);
	\draw[very thick,<->] (-1.15,0) -- (1.15,0);
	\draw[very thick,<->] (-120:1.15) -- (60:1.15);
	\draw[very thick,<->] (120:1.15) -- (-60:1.15);
	\node[scale=.8] at (3.3,2)
	{Stabiliser $\U(2)^{+}$
	};
	\node[scale=.8] at (-3.3,1)
	{Stabiliser $\U(2)^-$
	};
	\node[scale=.8] at (3.3,-1)
	{Generic stabiliser
	};
	\draw[thin] (30:1.5)node[circle,fill,minimum size=0.15cm,inner sep=0cm]{} .. controls (0.5,2) .. (2.5,2);
	\draw[thin] (15:1.5)node[circle,fill,minimum size=0.15cm,inner sep=0cm]{} .. controls (1,-0.1) .. (2.4,-1);
	\draw[thin] (120:0.7)node[circle,fill,minimum size=0.15cm,inner sep=0cm]{} .. controls (-2,1.4) .. (-2.5,1);
	\end{tikzpicture}
	\caption{A root diagram for $\mathfrak{g}_2$}
	\label{G2Root}
\end{figure}

\section{Closed $\G_2$-structures and the quadratic condition}\label{sect:QuadG2}

\subsection{Structure equations} Let $\pi: \mathcal{B} \to M$ be a $\G_2$-structure on a 7-manifold $M.$ A $\G_2$-structure may equivalently be thought of as a 3-form $\varphi$ on $M$ that is everywhere linearly equivalent to the standard 3-form $\phi$ on $V \cong \R^7.$ We then have
\begin{equation}
\mathcal{B}=\left\lbrace u : T_xM \xrightarrow{\sim} V \mid x \in M, u^* \phi = \varphi_x \right\rbrace.
\end{equation}
In components, the tautological $V$-valued 1-form $\omega$ on $\mathcal{B}$ may be written $\omega = \omega_i e^i,$ where $e^i$ is the canonical basis of $V^* = \R^7.$ Then the pullback of $\varphi$ to $\mathcal{B}$ is given by
\begin{align}\label{eq:phionB}
\varphi = \tfrac{1}{6} \eps_{ijk} \omega_i \wedge \omega_j \wedge \omega_k.
\end{align}

Since $\G_2$ is a subgroup of $\SO(7),$ a $\G_2$-structure $\varphi$ induces on $M$ a metric $g_{\varphi},$ volume form $\text{vol}_{\varphi},$ and Hodge star operator $*_{\varphi}.$ We have
\begin{align}
{*}_{\varphi}\varphi &= \tfrac{1}{24} \varepsilon_{ijkl} \omega_{i} \wedge \omega_j \wedge \omega_k \wedge \omega_{l}, \\
g_{\varphi} &= \omega_1^2 + \omega_2^2 + \omega_3^2 + \omega_4^2 + \omega_5^2 + \omega_6^2 + \omega_7^2, \\
\textrm{vol}_{\varphi} &= \omega_1 \wedge \omega_2 \wedge \omega_3 \wedge \omega_4 \wedge \omega_5 \wedge \omega_6 \wedge \omega_7.
\end{align}

A $\G_2$-structure also induces decompositions of the exterior powers $\Lambda^p(T^*M)$, corresponding to the decompositions (\ref{eqs:FormDecomp}) given above.

Let $\mathcal{F}_{\SO(7)}$ denote the oriented orthonormal coframe bundle of the metric $g_{\varphi},$ and $\SO(7)$-structure on $M.$ There is an inclusion $\mathcal{B} \subset \mathcal{F}_{\SO(7)}$. By the Fundamental Lemma of Riemannian Geometry, there exists a unique $\mathfrak{so}(7)$-valued 1-form $\psi = \psi_{ij} e_i \otimes e^j,$ the Levi-Civita connection form of $g_{\varphi},$ so that the equation
\begin{equation}\label{eq:CISO7}
d \omega_i = - \psi_{ij} \wedge \omega_j
\end{equation}
holds on $\mathcal{F}_{\SO(7)}.$

Restricted to $\mathcal{B} \subset \mathcal{F}_{\text{SO}(7)}$, the Levi-Civita $1$-form $\psi$ is no longer a connection $1$-form in general.  According to the splitting $\mathfrak{so}(7) = \mathfrak{g}_2 \oplus V$, there is a decomposition
\begin{equation}
\psi_{ij} = \theta_{ij} + \eps_{ijk} \gamma_k.
\end{equation}
Here, $\theta = \theta_{ij} e_i \otimes e^j$ is a $\mathfrak{g}_2$-valued connection form on $\mathcal{B}$ (the \emph{natural connection} on $\mathcal{B}$ in the parlance of $\G$-structures), and $\gamma = \gamma_i e_i$ is a $V$-valued $\pi$-semibasic 1-form on $\mathcal{B}.$ Since $\gamma$ is $\pi$-semibasic,
\begin{equation}\label{eq:gamdecomp}
\gamma_i = T_{ij}\omega^j
\end{equation}
for some $\mathrm{End}(V)$-valued function $T = T_{ij} e_i \otimes e^j \colon \mathcal{B} \to \mathrm{End}(V)$.  The $1$-form $\gamma$, and hence the functions $T_{ij}$, encodes the \emph{torsion} of the $\text{G}_2$-structure $\varphi.$

\begin{definition}
	The $\G_2$-structure $\mathcal{B}$ is called \emph{closed} if the 3-form $\varphi$ is closed, i.e. if $d \varphi = 0.$
\end{definition}

From differentiating equation (\ref{eq:phionB}), closure of $\varphi,$ i.e. the equation $d \varphi = 0$, is equivalent to $T$ taking values in $\mathfrak{g}_2 \subset \mathrm{End}(V),$ i.e. the equation $\varepsilon_{ijk} T_{ij} = 0.$ The function $T$ is known as the \emph{torsion tensor} of the closed $\G_2$-structure $\varphi.$ It follows from the general theory of equivalence that the function $T$ is a complete first-order diffeomorphism invariant of $\varphi.$

\begin{definition}
	The 2-form $\tilde{\tau}$ on $\mathcal{B}$ defined by
	\begin{equation}
	\tilde{\tau} = 3 \, T_{ij} \omega_i \wedge \omega_j
	\end{equation}
	is invariant under the $\G_2$-action on $\mathcal{B}$ and is thus the pullback to $\mathcal{B}$ of a well-defined 2-form on $M,$ which will be denoted by $\tau.$ This 2-form $\tau$ is an element of $\Omega^2_{14}(M),$ and is called the \emph{torsion 2-form} of the closed $\G_2$-structure $\varphi.$
\end{definition}

The torsion 2-form $\tau$ of a closed $\G_2$-structure satisfies
\begin{equation}\label{eq:dstrphitau}
d \, {*}_{\varphi} \varphi = \tau \wedge \varphi.
\end{equation}
In fact, equation (\ref{eq:dstrphitau}) is often taken to be the definition of $\tau.$

Substitution of equation (\ref{eq:gamdecomp}) into equation (\ref{eq:CISO7}) gives \emph{Cartan's first structure equation} for closed $\G_2$-structures,
\begin{align}\label{eq:CartanIG2}
d\omega_i = -\theta_{ij} \wedge \omega + \eps_{ijk}T_{kl} \omega_j \wedge \omega_l.
\end{align}

The structure equations for closed $\G_2$-structures can be developed further  \cite{BallConfFlat20}. There exist functions $H: \mathcal{B} \to \V_{2,0},$ $C: \mathcal{B} \to \V_{1,1},$ and $S: \mathcal{B} \to \V_{0,2}$ such that \emph{Cartan's second structure equations} for closed $\G_2$-structures are satisfied:
\begin{align}
dT_{jk} &= C_{ijk}\omega_i + \left(\varepsilon_{jkl}H_{li}+ 3 \varepsilon_{kil}H_{lj}+3\varepsilon_{lij}H_{lk}\right)\omega_i + T_{jl}\theta_{lk}-T_{kl}\theta_{lj}, \label{eq:dTcomp} \\
d\theta_{ij} &= -\theta_{ik} \wedge \theta_{kj} + \left(S_{ijkl} + J(H,T)_{ijkl} +r(T)_{ijkl} + \eps_{mkl} C_{ijm} + \eps_{mkl} L(H,T)_{ijm} \right) \omega_k \wedge \omega_l,
\end{align} 
where $J(H,T)_{ijkl},$ $r(T)_{ijkl},$ and $L(H,T)_{ijm}$ are explicit functions linear in the components $H$ and quadratic in the components of $T$ given by formulas (3.23) and (3.24) of \cite{BallConfFlat20}.

The tensors $H$ and $C$ are irreducible constituents of the covariant derivative of $T$ with respect to the natural connection defined by $\theta,$ while $S$ is a part of the Weyl tensor of $g_{\varphi}.$ 

The tensors $H,$ $C,$ and $S$ form a complete set of second-order invariants for a closed $\G_2$-structure. In particular, it is possible to express the full Riemann curvature tensor of $g_{\varphi}$ in terms of $H, C, S,$ and $T$v\cite{BallConfFlat20}. The Ricci tensor of $g_{\varphi}$ depends only on $T$ and $H,$ and it is thus possible to express it in terms of $\tau$ and its exterior derivative $d \tau$ alone. This has been done by Bryant \cite{Bry05}, who gives the following formulas,
\begin{equation}\label{eq:G2ScalRic}
\begin{aligned}
\text{Scal}(g_{\varphi}) &= -\tfrac{1}{2} | \tau |^2, \\
\text{Ric}(g_{\varphi}) &= \tfrac{1}{4} | \tau |^2 - \tfrac{1}{4} \mathsf{j}_{\varphi} \left( d \tau - \tfrac{1}{2} *_{\varphi} \left( \tau \wedge \tau \right) \right),
\end{aligned}
\end{equation}
where $\mathsf{j}_{\varphi}$ is a certain linear map $\mathsf{j}_{\varphi} : \Omega^3 \left(M \right) \to \Gamma \left( \mathrm{Sym}^2 T^*M \right)$.

\subsection{The Laplacian flow}

The Laplacian flow is the nonlinear flow equation for a $\G_2$-structure $\varphi$ given by
\begin{equation}\label{eq:LapFlow}
\tfrac{d}{dt} \varphi = \Delta_{\varphi} \varphi,
\end{equation}
where $\Delta_{\varphi}$ is the Laplacian induced by $\varphi$. When $\varphi$ is closed we have $\Delta_{\varphi} \varphi = d \tau,$ and thus the closed condition is preserved along the flow. On a compact manifold, the fixed points of the Laplacian flow are exactly the torsion-free $\G_2$-structures, and one might na\"ively expect to be able to use the flow to deform an initial closed $\G_2$-structure into a torsion-free one. Of course, the situation is likely to be more complicated than this, and while progress is being made (see \cites{LotWeiLapShi,LotWeiStab,LotWeiAna} for example), much work remains to be done.

A \emph{Laplace soliton} is a closed $\G_2$-structure $\varphi$ satisfying the equation
\begin{align}\label{eq:Ch2LapSol}
\Delta_{\varphi} \varphi = c \, \varphi + \mathcal{L}_X \, \varphi
\end{align}
for some constant $c \in \R$ and vector field $X \in \mathcal{X}(M).$ Laplace solitons are the solutions to (\ref{eq:LapFlow}) that evolve along the flow by scalings and diffeomorphisms, and it is expected that they will play an important role in modeling the possible singularities of the flow. In \S\ref{sect:U2pl} and \S\ref{sect:LapSols} we will construct several examples of Laplace solitons.

\subsection{The quadratic condition}\label{ssect:QuadCond}

In \cite{Bry05}, Bryant introduced a 1-parameter family of equations for closed $\G_2$-structures,
\begin{equation}\label{eq:LamQuad}
d\tau= \textstyle\frac{1}{7}|\tau|^2\varphi + \lambda \left( \textstyle\frac{1}{7}|\tau|^2 \varphi + *_{\varphi} \left(\tau \wedge \tau \right) \right),
\end{equation}
where $\lambda$ is a real constant.
\begin{definition}
	We shall say a closed $\G_2$-structure $(M, \varphi)$ is \emph{$\lambda$-quadratic} if it satisfies (\ref{eq:LamQuad}).
\end{definition}
\begin{remark}
	Bryant explains \cite{Bry05} that (\ref{eq:LamQuad}) is the most general way in which $d\tau$ can be prescribed by a natural expression quadratic in $\tau.$
\end{remark}

The 3-form $d \tau$ pulls back to $\mathcal{B}$ to be given by
\begin{equation}\label{eq:dtauH}
\pi^* d \tau = \eps_{ikl} \left( 21 H_{ij} - \tfrac{1}{7} T_{im}T_{mj} \right) \omega_{j} \wedge \omega_k \wedge \omega_l,
\end{equation}
and using this expression, the $\lambda$-quadratic equation (\ref{eq:LamQuad}) is seen to be equivalent to the following equation on the tensor $H,$
\begin{equation}\label{eq:LamQuadH}
H_{ij}=\frac{1-6\lambda}{7}\left(T_{ik}T_{kj}+\frac{1}{7}\delta_{ij}T_{kl}T_{kl} \right).
\end{equation}

Much of the interest in $\lambda$-quadratic closed $\G_2$-structures arises from the fact that several other interesting equations are equivalent to the $\lambda$-quadratic equation for special values of $\lambda$:
\begin{enumerate}
	\item The metric $g_{\varphi}$ is Einstein if and only if $\varphi$ is $1/2$-quadratic.
	\item The 3-form $\varphi$ is an eigenform for the Laplace operator induced by $g_{\varphi}$ if and only if $\varphi$ is 0-quadratic. These eigenform closed $\G_2$-structures are a special type of soliton for the Laplace flow, analogous to Einstein metrics for the Ricci flow.
	\item The $1/6$-quadratic case is the so-called \emph{extremally Ricci-pinched} case, discussed below.
	\item Closed $\G_2$-structures with $g_{\varphi}$ conformally flat are $-1/8$-quadratic. These structures have been classified by the author \cite{BallConfFlat20}.
\end{enumerate}

\subsubsection{Extremally Ricci-pinched closed $\G_2$-structures}

In \cite{Bry05}, Bryant proves the following two theorems, which single out the case $\lambda={1}/{6}$ for special attention.

\begin{thm}[\cite{Bry05}]
	Let $(M, \varphi)$ be a closed $\G_2$-structure on a compact manifold. If $M$ is $\lambda$-quadratic and $\varphi$ is not torsion-free, then $\lambda=\frac{1}{6}$.
\end{thm}

\begin{thm}[\cite{Bry05}]\label{thm:PinchedResult}
	Suppose $(M,\varphi)$ is a closed $\G_2$-structure on a compact manifold that satisfies
	\begin{equation}\label{eq:PinchedEqns}
	\int_M \, |\mathrm{Ric}^0\left(g_\varphi\right)|^2 \, \mathrm{vol}_{\varphi} \leq \frac{4k}{21} \int_M \, |\mathrm{Scal}\left(g_{\varphi}\right)|^2 \, \mathrm{vol}_M
	\end{equation}
	for some constant $k \leq 1.$
	\begin{itemize}
		\item If $k < 1$ then $\tau = 0$, so $\varphi$ is torsion-free.
		\item If $k=1$ then equality holds in (\ref{eq:PinchedEqns}) everywhere on $M$, and moreover $\varphi$ is $\frac{1}{6}$-quadratic.
	\end{itemize}
\end{thm}

Theorem \ref{thm:PinchedResult} justifies the name \emph{extremally Ricci-pinched} (or \emph{ERP} for short) for the $1/6$-quadratic case. Bryant also proves the following theorem which shows that the compact ERP closed $\G_2$-structures have a rich geometric structure.

\begin{thm}\label{thm:BryERPCons}
	Let $(M,\varphi)$ be an ERP closed $\G_2$-structure on a compact manifold $M,$ and suppose that $\varphi$ is not coclosed. Then
	\begin{itemize}
		\item $|\tau|^2$ is constant,
		\item $\tau \wedge \tau \wedge \tau =0,$
		\item The simple 4-form $\tau \wedge \tau$ is closed and coclosed,
		\item The tangent bundle $TM$ splits as an orthogonal direct sum of two integrable subbundles $P$ and $Q$, where $P = \left\lbrace v \in TM \mid v \lrcorner \left( \tau \wedge \tau \right) = 0 \right\rbrace$ has rank 3 and $Q = \left\lbrace v \in TM \mid v \lrcorner *_{\varphi} \left( \tau \wedge \tau \right) = 0 \right\rbrace$ is of rank 4.
	\end{itemize}
\end{thm}

The key ingredients to Bryant's proofs of these theorems are the identities
\begin{align}
d \tau^3 &= \frac{3\left(6 \lambda -1 \right)}{7} |\tau|^4 \mathrm{vol}_{\varphi}, \label{eq:Bry466} \\
\left(3 \lambda - 4 \right) d \left(|\tau|^2 \right) &= 7 \lambda \left(2 \lambda -1 \right) {*}_{\varphi} \tau^3, \label{eq:Bry469}
\end{align}
valid for any $\lambda$-quadratic closed $\G_2$-structure.

\subsubsection{The characteristic variety}\label{sssect:CharVar}

We now use the structure equations for closed $\G_2$-structures to further investigate equation (\ref{eq:LamQuad}) from the point of view of exterior differential systems. Our first result concerns the characteristic variety associated to the system.

\begin{prop}\label{prop:RCharVar}
	The real characteristic variety associated to the system (\ref{eq:LamQuad}) is empty.
\end{prop}
\begin{proof}
	The system (\ref{eq:LamQuad}) is equivalent to the condition (\ref{eq:LamQuadH}) on the tensor $H$. Inspecting (\ref{eq:dTcomp}), the tableau $\V_{1,1} \subset \text{Hom}\left(\V_{0,1}, \V_{1,0}\right)$ is seen to be a part of the tableau of free derivatives for this system. Let $\Xi_{\R}$ denote the real characteristic variety associated to $\V_{1,1} \subset \text{Hom}\left(\V_{1,0}, \V_{0,1}\right).$ Then $\Xi_{\R}$ provides an upper bound for the characteristic variety of (\ref{eq:LamQuad}). We have
	\begin{equation}
	\Xi_{\R} = \left\lbrace [\xi] \in \mathbb{P}\,\V_{1,0}^* \mid \exists 0 \neq \beta \in \V_{0,1} \:\: \text{s.t.} \:\: \beta \otimes \xi \in \V_{1,1} \subset \V_{0,1} \otimes \V_{1,0}^* \right\rbrace.
	\end{equation}
	The decomposition of $\V_{0,1} \otimes \V_{1,0}^*$ into $\G_2$ irreducibles is $\V_{0,1} \otimes \V_{1,0}^* \cong \V_{1,1} \oplus \V_{2,0} \oplus \V_{1,0}$. The map $\Lambda^1_7 \otimes \Lambda^2_{14} \to \Lambda^3_1 \oplus \Lambda^3_7 \oplus \Lambda^3_{27}$ induced by the wedge product has image $\Lambda^3_7 \oplus \Lambda^3_{27}$ and gives a projection of $\V_{0,1} \otimes \V_{1,0}^*$ onto $\V_{2,0} \oplus \V_{1,0}.$ It follows that $\beta \otimes \xi \in \V_{1,1} \subset \V_{0,1} \otimes \V_{1,0}^*$ if and only if $\beta \wedge \xi = 0,$ where we think of $\beta$ as a 2-form and $\xi$ as a 1-form. This immediately gives that $\Xi_{\R}$ is empty, since there are no elements $\beta$ of $\V_{0,1} \cong \mathfrak{g}_2$ that are decomposable as 2-forms (this can been seen by conjugating so that $\beta$ lies in a maximal torus).
\end{proof}

Proposition \ref{prop:RCharVar} essentially states that the system (\ref{eq:LamQuad}) is elliptic (modulo diffeomorphisms).

\begin{cor}\label{cor:RealAnal}
	Closed $\G_2$-structures $\varphi$ satisfying the $\lambda$-quadratic condition (\ref{eq:LamQuad}) are real analytic in $g_{\varphi}$-harmonic coordinates.
\end{cor}

\begin{remark}
	The proof of Proposition \ref{prop:RCharVar} implies that any system of equations for a closed $\G_2$-structure that amounts to prescribing the tensor $H$ has empty real characteristic variety.
\end{remark}

\subsubsection{Non-involutivity}\label{sssect:NonInvol}

To continue the study of the $\lambda$-quadratic condition (\ref{eq:LamQuad}) from the point of view of exterior differential systems, we need to develop the structure equations to third order. Define a $\V_{2,0}$-valued function $\widetilde{H}$ on $\mathcal{B}$ by
\begin{equation}
\widetilde{H}_{ij} = \frac{1-6\lambda}{7}\left(T_{ik}T_{kj}+\frac{1}{7}\delta_{ij}T_{kl}T_{kl} \right) - H_{ij},
\end{equation}
so that $\widetilde{H}=0$ precisely when $\varphi$ is $\lambda$-quadratic. Let $\nabla$ denote the covariant derivative associated to the natural connection $\theta.$ Now,
\begin{equation}
\begin{split}
\nabla\widetilde{H} \:\: \text{takes values in} \:\: \V_{2,0} \otimes \V_{1,0} & \cong  \V_{3,0} \oplus \V_{1,1} \oplus \V_{2,0} \oplus \V_{0,1} \oplus \V_{1,0}, \\
\nabla C \:\: \text{takes values in} \:\: \V_{1,1} \otimes \V_{1,0} & \cong  \V_{2,1} \oplus \V_{0,2} \oplus \V_{3,0} \oplus \V_{1,1} \oplus \V_{2,0} \\
& \oplus \V_{0,1}, \\
\nabla {S} \:\: \text{takes values in} \:\: \V_{0,2} \otimes \V_{1,0} & \cong  \V_{1,2} \oplus \V_{2,1} \oplus \V_{1,1}.
\end{split}
\end{equation}

Solving the equations $d^2 T=0, d^2 \theta = 0$ yields various relations between the components of these tensors. When $\widetilde{H}=0,$ we have $\nabla \widetilde{H} = 0$ as well, and the equations for $[\nabla \widetilde{H}]^{1,0}$ and $[\nabla \widetilde{H}]^{2,0}$ yield the following relations,
\begin{subequations}\label{eq:2JetCond}
	\begin{align}
	\left( 3 \lambda -4 \right) C_{ijk} T_{jk} &+ \tfrac{12}{7} \left( 8 \lambda +1 \right) \left( 5 \lambda -2 \right) \eps_{ijk} T_{jl} T_{lm} T_{mk} = 0, \\
	\lambda & \left( \eps_{ilm} C_{ljr} T_{mr} + \eps_{jlm} C_{lir} T_{mr} \right) = 0.
	\end{align}
\end{subequations}
Equations equivalent to (\ref{eq:2JetCond}) can also be obtained by applying $d^2 = 0$ to (\ref{eq:LamQuad}), and they were obtained by Bryant \cite{Bry05} in this fashion. The advantage of our approach is that it shows that (\ref{eq:2JetCond}) are the \emph{only} second-order conditions obtained by taking three derivatives of the $\G_2$-structure.

The equations (\ref{eq:2JetCond}) show that the system (\ref{eq:LamQuadH}) is not involutive, as explained by Bryant \cite{Bry05}. One can ask if the system obtained by adjoining the relation (\ref{eq:2JetCond}) to (\ref{eq:LamQuad}) is involutive. Unfortunately, this is not the case: differentiating (\ref{eq:2JetCond}) yields other second order conditions which are not algebraic consequences of (\ref{eq:LamQuad}) or (\ref{eq:2JetCond}). For instance, if we let $Z$ denote the $\V_{1,0}$ component of $\nabla \widetilde{H}$, then $\nabla Z$ must vanish for a solution of (\ref{eq:LamQuad}). Now, $\nabla Z$ takes values in $\V_{1,0} \otimes \V_{1,0} \cong \V_{2,0} \oplus \V_{0,1} \oplus \V_{1,0} \oplus \V_{0,0}.$ Taking the $\V_{0,0}$ part of $\nabla Z$ gives the following second-order relation:
\begin{align}\label{eq:nZ00}
\left( 3 \lambda -4 \right) S_{ijkl} T_{ij}T_{kl} &= -\tfrac{3}{392} \left( 27648 \lambda^3 - 2484 \lambda^2 -5151 \lambda + 628 \right) \left( T_{ij} T_{ij}\right)^2 \\ &+ \left( 3 \lambda -4 \right) C_{ijk} C_{ijk}. \nonumber
\end{align}
If we let $M$ denote the $\V_{2,0}$ component of $\nabla \widetilde{H},$ then a particular combination of the $\V_{2,0}$ components of $\nabla Z$ and $\nabla M$ yields another second-order relation, with terms of the form $C^2,$ $T^4,$ $T^2 \cdot C,$ and $T^2 \cdot S.$ Other components of $\nabla Z$ and $\nabla M$ place certain restrictions on $\left[\nabla C \right]^{3,0},$ but we will not list them here. Further differentiation may yield more restrictions, but at this level the algebra involved becomes too cumbersome to continue.

\subsection{Special torsion}\label{ssect:SpecTors}

The non-involutivity of the $\lambda$-quadratic condition makes the study of these structures difficult, and the local existence of any examples becomes an interesting question. Bryant \cite{Bry05} gives one example of a homogeneous ERP structure, Lauret \cite{LauretLap} another, and Lauret--Nicolini \cites{LauNicERPLI20,LauNicERPLICAG20} have classified left-invariant ERP structures on Lie groups. Examples with $\lambda = -1/8$ and $\lambda=2/5$ have been given by the author \cite{BallConfFlat20}. In \S\ref{sect:U2pl} and \S\ref{sect:U2mi} we will give new examples of $\lambda$-quadratic closed $\G_2$-structures for $\lambda = -1,$ $-1/8,$ $1/6,$ $2/5,$ and $3/4.$

We now describe the idea that allows us to find the examples of \S\ref{sect:U2pl} and \S\ref{sect:U2mi}. As mentioned before, the torsion tensor $T$ of a closed $\G_2$-structure takes values in the Lie algebra $\mathfrak{g}_2$. Recall that the stabilisers of elements of $\mathfrak{g}_2$ were classified in \S\ref{ssect:G2adjoint}, and there were two exceptional $\G_2$-orbits in $\mathfrak{g}_2$: $\mathcal{O}^-$ and $\mathcal{O}^{+},$ with stabilisers $\U(2)^+$ and $\U(2)^-$ respectively.

\begin{definition}
	A closed $\G_2$-structure $\varphi$ said to have \emph{special torsion} of \emph{positive (resp. negative) type} if the torsion 2-form $\tau$ of $\varphi$ has pointwise stabiliser everywhere conjugate to $\U(2)^{+}$ (resp. $\U(2)^{-}$).
\end{definition}

The special torsion conditions are nonlinear first order conditions on the $\G_2$-structure. We have
\begin{equation*}
\begin{aligned}
&\varphi \:\: \text{has special torsion of positive type} \iff |\tau^3|^2 = 0, \\
&\varphi \:\: \text{has special torsion of negative type} \iff |\tau^3|^2 = \tfrac{2}{3} | \tau|^6.
\end{aligned}
\end{equation*}

The special torsion assumption simplifies considerably the identities (\ref{eq:2JetCond}) and (\ref{eq:nZ00}) which allows for the structure equations to be developed further in this setting and enables the classification theorems and examples of \S\ref{sect:U2pl} and \S\ref{sect:U2mi}.

More generally, any question that can be asked of closed $\G_2$-structures can be restricted to the setting of special torsion, and may be easier to answer there.

\subsection{An example with non-special torsion}\label{ssect:NonSpecTors}

In this section we construct an example of a $1/3$-quadratic closed $\G_2$-structure. This example is notable because it will be shown in \S\ref{sect:U2pl} and \S\ref{sect:U2mi} that $\lambda = 1/3$ cannot occur in the special torsion setting.

Let $\G$ be the group $SL(2, \R ) \ltimes \R^4,$ where $SL(2, \R)$ acts on $\R^4$ via the irreducible 4-dimensional representation $\mathrm{Sym}^3 \R^2$. Write the left-invariant Maurer-Cartan form on $\G$ as
\begin{small}
	\begin{align}
	\left( \begin {array}{ccccc} 0&0&0&0&0\\
	\nu_{{0}}&3\,\alpha_{{1}}&\alpha_{{2}}&0&0\\
	\nu_{{1}}&3\,\alpha_{{3}}&\alpha_{{1}}&2\,\alpha_{{2}}&0\\
	\nu_{{2}}&0&2\,\alpha_{{3}}&-\alpha_{{1}}&3\,\alpha_{{2}}\\
	\nu_{{3}}&0&0&\alpha_{{3}}&-3\,\alpha_{{1}}
	\end {array} \right),
	\end{align}
\end{small}
and define forms $\beta_2, ..., \beta_7,$ and $\kappa$ by
\begin{align}
\beta_2 &= - \nu_0 + \nu_2, \:\:\: \beta_3 = \nu_1 - \nu_3, \nonumber \\
\beta_4 &= 2 \alpha_1, \:\:\: \beta_5 = \alpha_2 + \alpha_3, \nonumber \\
\beta_6 &= -\tfrac{1}{15} \nu_1 - \tfrac{1}{5} \nu_3, \:\:\: \beta_7 = \tfrac{1}{5} \nu_0 + \tfrac{1}{15} \nu_2, \nonumber \\
\kappa &= \alpha_2 - \alpha_3. \nonumber
\end{align}
The vector dual to the form $\kappa$ generates an $S^1$-action on $\G.$ On $\R_+ \times {\G},$ with coordinate $r$ in the $\R$ direction, define a 3-form $\varphi$ by
\begin{equation}
\begin{aligned}
\varphi =& d r \wedge \left( \frac{1}{r^3} \beta_2 \wedge \beta_3 + 6 r^2 \beta_4 \wedge \beta_5 + \frac{15}{r^2} \beta_6 \wedge \beta_7 \right) \\
& + \frac{3}{r^2} \left( \beta_2 \wedge \beta_4 \wedge \beta_6 - \beta_2 \wedge \beta_5 \wedge \beta_7 - \beta_3 \wedge \beta_5 \wedge \beta_6 -\beta_3 \wedge \beta_4 \wedge \beta_7 \right).
\end{aligned}
\end{equation}
Then $\varphi$ descends to the space $M = \R_+ \times {\G}/{\mathrm{S}^1}$ to define a $1/3$-quadratic closed $\G_2$-structure on $M.$ A simple calculation using the structure equations of $\G$ gives that $\tau$ has pointwise stabiliser $\mathrm{T}^2$.

\section{Quadratic closed $\G_2$-structures with special torsion of positive type}\label{sect:U2pl}

This section is devoted to the study of $\lambda$-quadratic closed $\G_2$-structures with special torsion of positive type. The elements $\beta \in \Lambda^2_{14}$ with this stabiliser type are characterised by the condition $\beta^3 = 0.$ Thus, by equation (\ref{eq:Bry466}), $\lambda={1}/{6}$ and this section deals exclusively with ERP closed $\G_2$-structures. We also remark that by Bryant's Theorem \ref{thm:BryERPCons}, an ERP closed $\G_2$-structure on a compact 7-manifold has special torsion of positive type. 

\subsection{The induced $\U(2)^+$-structure}\label{ssect:U2plInduced}

Let $(M,\varphi)$ be an ERP closed $\G_2$-structure such that $\tau$ has stabiliser everywhere conjugate to $\U(2)^+.$ By equation (\ref{eq:Bry469}), $| \tau|^2$ is constant. In particular, if the $\G_2$-structure is not torsion-free then the torsion 2-form is non-zero everywhere. We may then adapt frames and define a $\U(2)^+$-subbundle $\mathcal{Q} \subset \mathcal{B}$ by
\begin{equation}
\mathcal{Q} = \left\lbrace u : T_x M \to V \mid u \in \mathcal{B}_x,  u^* k \left( e_4 \wedge e_5 - e_6 \wedge e_7 \right) = \tau \right\rbrace,
\end{equation}
where $k$ is a non-zero constant. We emphasise that the construction of $\mathcal{Q}$ depends only on the $\G_2$-structure $\varphi$, and so Theorem \ref{thm:MoE} can be used to test for equivalence of the $\G_2$-structures.

The action of $\U(2)^{+}$ on the standard representation $V \cong \R^7$ of $\G_2$ is reducible: we have $V \cong \langle e_1, e_2, e_3 \rangle \oplus \langle e_4, e_5, e_6, e_7 \rangle \cong \R^3 \oplus \C^2.$ We split the $V$-valued tautological 1-form $\omega$ accordingly as $\omega = \nu + \eta$. Explicitly, define complex-valued 1-forms $\eta_a$ and $\nu_{a \bar{b}} = - \overline{\nu_{b \bar{a}}}$ on $\mathcal{Q}$ by
\begin{equation*}
\begin{aligned}
\nu_{1\bar{1}} &= - i \omega_1, & \eta_1 &= \omega_4 + i \omega_5,  \\
\nu_{1\bar{2}} &= - \omega_2 + i \omega_3, & \eta_2 &= -\omega_6 + i \omega_7, \\
\nu_{2\bar{2}} &= i \omega_1.
\end{aligned}
\end{equation*}

Throughout the rest of this section, $1 \leq a, b, c, \ldots \leq 2$ will be indices obeying the complex Einstein summation convention, meaning that any pair of barred and unbarred indices is implicitly summed over. We will also make use of an $\eps$-symbol with two such indices, defined to be the unique skew-symmetric symbol satisfying $\eps_{12} = 1/2.$

The natural $\mathfrak{g}_2$-valued connection form $\theta$ also splits. There is a $\U(2)^+$-invariant decomposition
\begin{equation}
\mathfrak{g}_2 = \mathfrak{u}(2)^+ \oplus \C \oplus \textrm{Sym}^3 \C^2,
\end{equation}
and $\theta$ decomposes accordingly as $\theta = \kappa + \xi + \sigma.$ Explicitly, define complex-valued 1-forms $\kappa_{a \overline{b}} = - \overline{\kappa_{b \bar{a}}},$ $\xi,$ and $\sigma_{abc} = \sigma_{bac} = \sigma_{acb}$ by
\begin{equation*}
\begin{aligned}
\kappa_{1\bar{1}} &= -i \theta_{45} , & \sigma_{111} & = \tfrac{1}{2} \left(\theta_{35}-\theta_{24} \right) - \tfrac{i}{2} \left(\theta_{25} - \theta_{34} \right), \\
\kappa_{1\bar{2}} &= \tfrac{1}{2} \left(\theta_{57} - \theta_{46} \right) - \tfrac{i}{2} \left(\theta_{47} + \theta_{56} \right) , & \sigma_{112} & = \tfrac{1}{2} \left(\theta_{26} - \theta_{37} \right) - \tfrac{i}{2} \left(\theta_{27} + \theta_{36} \right), \\
\kappa_{2\bar{2}} &= i \theta_{67} , & \sigma_{122} & = \tfrac{1}{2} \left(\theta_{35} -\theta_{24} \right) - \tfrac{i}{2} \left(\theta_{25} + \theta_{34} \right), \\
\xi & =  \tfrac{1}{2} \left(\theta_{46} + \theta_{57} \right) - \tfrac{i}{2} \left(\theta_{47} - \theta_{56} \right) , & \sigma_{222} & = \tfrac{1}{2} \left(\theta_{26} + \theta_{37} \right) - \tfrac{i}{2} \left(\theta_{27} - \theta_{36} \right).
\end{aligned}
\end{equation*}
The $\mathfrak{u}(2)$-valued 1-form $\kappa$ is the connection form associated to the natural connection on the $\U(2)^+$-structure $\mathcal{Q},$ while the 1-forms $\xi$ and $\sigma$ are are semibasic for the projection $\mathcal{Q} \to M.$

\subsubsection{Structure equations}

The first structure equation on $\mathcal{B}$ (\ref{eq:CartanIG2}) restricted to $\mathcal{Q}$ reads
\begin{equation}\label{eq:CIU2plPre}
\begin{aligned}
d \eta_a &= - \kappa_{a \bar{b}} \wedge \eta_b + i k \nu_{a \bar{b}} \wedge \eta_b + \eps_{ab} \xi \wedge \overline{\eta_b} - 2 \overline{\eps_{be}} \sigma_{abc} \wedge \nu_{e \bar{c}}, \\
d \nu_{a \bar{b}} &= - \kappa_{a\bar{c}} \wedge \nu_{c\bar{b}} + \kappa_{c\bar{b}} \wedge \nu_{a\bar{c}} - 2 \overline{\eps_{cb}} \sigma_{acd} \wedge \overline{\eta_d} + 2 \eps_{ca} \overline{\sigma_{bcd}} \wedge \eta_d.
\end{aligned}
\end{equation}

On $\mathcal{Q},$ the $\G_2$ 3-form $\varphi,$ 4-form $*_{\varphi} \varphi,$ and torsion 2-form $\tau$ are given by
\begin{equation}
\begin{aligned}
\varphi &= -\tfrac{1}{12} \nu_{a\bar{b}} \wedge \nu_{b\bar{c}} \wedge \nu_{c\bar{a}} + \tfrac{1}{2} \nu_{a\bar{b}} \wedge \overline{\eta_a} \wedge \eta_b, \\
*_{\varphi} \varphi &= -\tfrac{1}{4} \overline{\eps_{ab}} \eps_{cd} \eta_a \wedge \eta_b \wedge \overline{\eta_c} \wedge \overline{\eta_d} + \tfrac{1}{4} \nu_{a\bar{b}} \wedge \nu_{b\bar{c}} \wedge \overline{\eta_a} \wedge \eta_c, \\
\tau &= 3 i k \, \eta_a \wedge \overline{\eta_a}.
\end{aligned}
\end{equation}
The ERP condition (\ref{eq:LamQuad}) implies that
\begin{equation}\label{eq:ERPU2plcond}
d \left( i \, \eta_a \wedge \overline{\eta_a} \right) = 2 k \, \nu_{a\bar{b}} \wedge \overline{\eta_a} \wedge \eta_b.
\end{equation}
Equation (\ref{eq:ERPU2plcond}) together with its exterior derivative
\begin{equation}
d \left(  \nu_{a\bar{b}} \wedge \overline{\eta_a} \wedge \eta_b \right) = 0
\end{equation}
implies that there exist complex-valued functions $A_{abcde}$, $S_{abcd},$ and $N_a$ on $\mathcal{Q},$ each fully symmetric in their indices, such that
\begin{equation}
\begin{aligned}
\sigma_{abc} &= S_{abcd} \overline{\eta_d} - 2 \overline{\eps_{df}}A_{abcde} \nu_{f \bar{e}},\\
 \xi &= N_a \overline{\eta_a}.
\end{aligned}
\end{equation}
Substituting these equations into (\ref{eq:CIU2plPre}) yields Cartan's first structure equation for $\mathcal{Q}:$
\begin{equation}\label{eq:StructEqU2pl}
\begin{aligned}
d \eta_a &= - \kappa_{a \overline{b}} \wedge \eta_b + i k \nu_{a \overline{b}} \wedge \eta_b + \eps_{ab} N_c \overline{\eta_c \wedge \eta_b} - 2 \overline{\eps_{be}} S_{abcd} \overline{\eta_d} \wedge \nu_{e \overline{c}}, \\
d \nu_{a \bar{b}} &= - \kappa_{a\bar{c}} \wedge \nu_{c\bar{b}} + \kappa_{c\bar{b}} \wedge \nu_{a\bar{c}} + 4 \overline{\eps_{cb}} \overline{\eps_{eg}} A_{acdef} \nu_{g\bar{f}} \wedge \overline{\eta_d} \\
& +4 \eps_{ca} \eps_{eg} \overline{A_{bcdef}} \nu_{f\bar{g}} \wedge \eta_d.
\end{aligned}
\end{equation}

An immediate consequence of (\ref{eq:StructEqU2pl}) is that the orthogonal distributions defined by $\nu = 0$ and $\eta=0$ are integrable. In fact, these distributions are exactly the integrable distributions $P$ and $Q$ defined in Theorem \ref{thm:BryERPCons}. The torsion 2-form $\tau$ restricts to a symplectic 2-form on each 4-dimensional integral manifold of $\nu=0,$ and together with the restriction of the metric $g_{\varphi}$ this defines an almost K\"ahler structure on these integral manifolds.

\begin{remark}
	The torsion of a general $\U(2)^{+}$-structure on a 7-manifold is a section of a vector bundle of dimension
	\begin{equation}
	\operatorname{dim} \left(\frac{\mathfrak{so}(7)}{\mathfrak{u}(2)}  \otimes \R^7 \right) =119.
	\end{equation}
	The calculations above show that the torsion of the $\U(2)^{+}$-structure $\mathcal{Q}$ associated to an ERP $\G_2$-structure takes values in a 53-dimensional subspace. Conversely, any $\mathrm{U}(2)^+$-structure with torsion taking values in this 53-dimensional subspace gives rise to a unique ERP closed $\G_2$-structure.
\end{remark}

\begin{remark}\label{rmk:rescale}
	The ERP condition is invariant under rescaling the $\G_2$-structure $\varphi$ by a non-zero constant. Under the condition $\tau^3 = 0,$ we see that the ERP condition is also invariant under the two-parameter family of rescalings $\tilde{\nu} = a \, \nu,$ $\tilde{\eta} = b \, \eta$ for non-zero constants $a$ and $b.$ Under such a rescaling, the torsion functions $A,$ $S,$ and $N,$ and the constant $k$ scale as
	\begin{equation}
	\begin{aligned}
	\tilde{A} =& \frac{1}{b} A, & \tilde{S} =& \frac{1}{a} S, & \tilde{N} =& \frac{1}{b} N, & \tilde{k} =& \frac{1}{a} k.
	\end{aligned}
	\end{equation}
\end{remark}

The group $\U(2)^{+}$ does not appear in the list of possible holonomy groups of a 7-dimensional manifold, so a torsion-free $\U(2)^{+}$-structure is flat. It follows that the Bianchi identities $d^2 \eta = d^2 \nu = 0$ express the curvature of the natural connection $\kappa$ as a function of the torsion terms $A, S, N,$ their derivatives, and $k.$ We do not record the result of this calculation here as we will not use it in this generality.

At this point, the natural way forward is to apply the techniques of \cite{BryEDSNotes} to try to determine the existence and generality of these structures. However, a calculation shows that the Jacobi manifold defined by (\ref{eq:StructEqU2pl}) and the equation for $d \kappa$ is not involutive. Prolonging once does not yield any new integrability conditions, but the tableau of free derivatives is not involutive. Prolonging once more does give new integrability conditions, so the exterior differential system associated to this type of structure is not involutive. Unfortunately, the integrability conditions are of an algebraically complicated nature and this precludes further investigation of the general ERP structure with special torsion. Instead, we will restrict to special cases defined by first order conditions.

\subsubsection{Invariants on the base}

Each of the functions $A,$ $S$, and $N$ take values in an irreducible representation of $\U(2),$ and the following tensors are well-defined on the manifold $M$:
\begin{equation}\label{eq:U2plBaseTens}
\begin{aligned}
\mathsf{A} = \overline{A_{abcde}} \eta_a\eta_b\eta_c\eta_d\eta_e \otimes \left(\eta_1 \wedge \eta_2 \right)^{1/2} &\in \mathrm{Sym}^5_{\C} \C^2 \otimes  \left( \Lambda^2_{\C} \C^2 \right)^{1/2}, \\
\mathsf{S} = \overline{S_{abcd}} \eta_a\eta_b\eta_c\eta_d \otimes \left(\eta_1 \wedge \eta_2 \right) &\in \mathrm{Sym}^4_{\C} \C^2 \otimes  \Lambda^2_{\C} \C^2 , \\
\mathsf{N} = \overline{N_a} \eta_a \otimes \left(\eta_1 \wedge \eta_2 \right)^{3/2} &\in \C^2 \otimes  \left( \Lambda^2_{\C} \C^2 \right)^{3/2}.
\end{aligned}
\end{equation}

The tensor $A$ at a point $p \in M$ can be identified with the second fundamental form at $p$ of the leaf of the distribution $\omega=0$ passing through $p,$ while the tensor $S$ at $p$ can be identified with the second fundamental form at $p$ of the leaf of $\nu=0$ through $p.$ The tensor $N$ can be identified with the Nijenhuis tensor of the almost complex structure induced on the leaves of $\nu=0.$

\subsubsection{Bryant's example}\label{sssect:BryEx}

The first example of an ERP closed $\G_2$-structure was given by Bryant \cite{Bry05}. Let $\G$ be the group of affine transformations of $\mathbb{C}^2$ that preserve the complex volume form up to phase. Bryant's example is an ERP closed $\G_2$-structure on the homogeneous space ${\G}/{\U(2)}$. We can give an interpretation of this example in the context of $\U(2)^{+}$-structures.

\begin{thm}
	A $\U(2)^{+}$-structure is locally equivalent to Bryant's example (up to rescaling) if and only if the torsion functions $A, S,$ and $N$ vanish identically.
\end{thm}

\begin{proof}
	Suppose that $\mathcal{Q} \to M$ is a $\U(2)^{+}$-structure with $A=P=Z=0$. Then the first structure equation (\ref{eq:StructEqU2pl}) becomes
	\begin{equation}\label{eq:StructEqU2plBry}
	\begin{aligned}
	d \eta_a &= - \kappa_{a \overline{b}} \wedge \eta_b + i k \, \nu_{a \overline{b}} \wedge \eta_b, \\
	d \nu_{a \bar{b}} &= - \kappa_{a\bar{c}} \wedge \nu_{c\bar{b}} + \kappa_{c\bar{b}} \wedge \nu_{a\bar{c}}.
	\end{aligned}
	\end{equation}
	The identities $d^2 \nu = d^2 \eta = 0$ imply
	\begin{equation}\label{eq:StructEq2U2plBry}
	d \kappa_{a \bar{b}} = - \kappa_{a \bar{c}} \wedge \kappa_{c \bar{b}} + k^2 \, \nu_{a \bar{c}} \wedge \nu_{c \bar{b}}.
	\end{equation}
	
	Equations (\ref{eq:StructEqU2plBry}) and (\ref{eq:StructEq2U2plBry}) are the structure equations for an 11-dimensional Lie group. To identify this group as the group $G$ of affine transformations of $\mathbb{C}^2$ that preserve the standard complex volume form up the phase, we write $G$ as the group of 3-by-3 complex matrices of the form
	\begin{align}
	\left( \begin{array}{ccc}
	a & b & x \\
	c & d & y \\
	0 & 0 & 0
	\end{array} \right)
	\end{align}
	with $|ad-bc|=1.$ The left-invariant Maurer-Cartan form may be written as
	\begin{align}
	\mu = \left( \begin{array}{ccc}
	\kappa_{1\bar{1}} - ik \, \nu_{1\bar{1}} & \kappa_{1\bar{2}} - ik \, \nu_{1\bar{2}} & \eta_1 \\
	\kappa_{2\bar{1}} - ik \, \nu_{2\bar{1}} & \kappa_{2\bar{2}} - ik \, \nu_{2\bar{2}} & \eta_2 \\
	0 & 0 & 0
	\end{array} \right),
	\end{align}
	and Maurer-Cartan equation $d\mu + \mu \wedge \mu =0$ matches up with the structure equations (\ref{eq:StructEqU2plBry}) and (\ref{eq:StructEq2U2plBry}), showing that the $\U(2)^{+}$-structure $\mathcal{Q} \to M$ is locally equivalent to $\G \to {\G}/{\U(2)}.$
	
	The converse statement follows by reversing the above steps.
\end{proof}

\begin{remark}
	This theorem can be compared to Theorem 1.2 in \cite{CleyIv08}, which says that if $(M, \varphi)$ is a manifold with closed $\G_2$-structure such that the torsion $\tau$ is parallel under the natural $\G_2$ connection, then $M$ is locally equivalent to Bryant's example. If $\tau$ is invariant under $\nabla$ then, using the notation of the previous section, ${H}=0$ and $C=0.$ The conditions $A= S =N = 0$ are equivalent to the vanishing of $C,$ and ${H}=0$ is simply the ERP condition.
\end{remark}

\subsubsection{Evolution under Laplacian flow}\label{sssect:LapFlow}

The evolution of ERP $\U(2)^{+}$-structures under the Laplacian flow (\ref{eq:LapFlow}) is now simple to compute. From the form (\ref{eq:ERPU2plcond}) of $d\tau,$ the subbundles defined by the vanishing of $\nu$ and $\eta$ are preserved. In fact, we have $\nu(t) = \nu(0)$, $\eta(t) = \exp(6ct) \eta(0)$, and the 1-parameter family of closed $\G_2$-structures
\begin{equation}
\begin{aligned}\label{eq:U2plLapFlow}
\varphi(t) &= -\tfrac{1}{12} \nu_{a\bar{b}} \wedge \nu_{b\bar{c}} \wedge \nu_{c\bar{a}} + \tfrac{1}{2} \nu_{a\bar{b}} \wedge \overline{\eta_a} \wedge \eta_b, \\
&= \omega_{123} + e^{12kt} \left( \omega_{145} + \omega_{167} + \omega_{246} - \omega_{257} - \omega_{347} - \omega_{356} \right)
\end{aligned}
\end{equation}
solves the Laplacian flow equation. Thus, the Laplacian flow simply moves the $\U(2)^+$-structure along the rescaling of Remark \ref{rmk:rescale}. In particular, the ERP condition is preserved. This result has been found independently by Fino--Raffero \cite{FiRafLap18}.

\subsection{Type $A$}\label{ssect:TypeA}

We shall say that the $\U(2)^+$-structure $\mathcal{Q}$ is \emph{type $A$} if $A$ is the only non-vanishing torsion function, i.e. if $S=N=0$ on $\mathcal{Q}.$

If $\mathcal{Q}$ is type $A$, then the structure equations (\ref{eq:StructEqU2pl}) become
\begin{equation}\label{eq:StructEqtypeA}
\begin{aligned}
d \theta_a &= - \kappa_{a \overline{b}} \wedge \theta_b + i k \nu_{a \overline{b}} \wedge \theta_b, \\
d \nu_{a \bar{b}} &= - \kappa_{a\bar{c}} \wedge \nu_{c\bar{b}} + \kappa_{c\bar{b}} \wedge \nu_{a\bar{c}} + 4 \overline{\eps_{cb}} \overline{\eps_{eg}} A_{acdef} \nu_{g\bar{f}} \wedge \overline{\eta_d} \\
& +4 \eps_{ca} \eps_{eg} \overline{A_{bcdef}} \nu_{f\bar{g}} \wedge \eta_d.
\end{aligned}
\end{equation}

The equations $d^2 \theta = d^2 \nu = 0$ then imply that the tensor $A$ must satisfy
\begin{equation}\label{eq:TypeARestrict}
\overline{\eps_{bg}\eps_{ch}\eps_{di}\eps_{ej}}A_{abcde}A_{fghij} = 0, \:\:\: a,f = 1, 2.
\end{equation}
This equation forces $A$ to take a special algebraic form.
\begin{lem}
	Equations (\ref{eq:TypeARestrict}) hold if and only if the tensor $\mathsf{A}$ (\ref{eq:U2plBaseTens}) has a quadruple linear factor at each point of $M$.	
\end{lem}
\begin{proof}
	Equations (\ref{eq:TypeARestrict}) are invariant under the action of $\mathrm{SL}_2 \C$ on $\mathrm{Sym}^5 \C^2,$ not just the group $\U(2).$ If $\mathsf{A}$ has three distinct linear factors at some point $p \in M,$ then we may act by an element of $\mathrm{SL}_2 \C$ so that they are $\eta_1=0,$ $\eta_2 = 0,$ and $\eta_1 - \eta_2 = 0,$ and $\mathsf{A}$ is of the form $\eta_1 \eta_2 \left(\eta_1 - \eta_2\right) \mathsf{Q}$ for some quadratic polynomial in $\eta_1$ and $\eta_2.$ A simple calculation shows that equations (\ref{eq:TypeARestrict}) imply that $\mathsf{Q} = 0.$
	
	Conversely, if $\mathsf{A}$ has a quadruple linear factor, we may act by an element of $\mathrm{SL}_2 \C$ so that $\mathsf{A} = \eta_1^4 \mathsf{L}$ for some linear factor $\mathsf{L}.$ A simple calculation shows that equations (\ref{eq:TypeARestrict}) are satisfied.
\end{proof}

We now suppose that $A$ is non-vanishing on $\mathcal{Q},$ and restrict to the $\mathrm{T}^2$-bundle $\mathcal{Q}_1 \subset \mathcal{Q}$ where $\mathsf{A} = \eta_1^4 \mathsf{L}$ for some linear factor $\mathsf{L}.$ On $\mathcal{Q}_1,$
\begin{equation}
A_{11122} = A_{11222} = A_{12222} = A_{22222} = 0.
\end{equation}
Let us write $A_{11111} = A_1, A_{11112} = A_2$ on $\mathcal{Q}_1.$ The equations $d^2 \nu = 0,$ $d^2 \eta = 0$ imply that
\begin{equation}
d \kappa_{i \bar{i}} = \left(|A_1|^2 + 4 |A_2|^2 \right) \theta_1 \wedge \overline{\theta_1} + \overline{A_1}A_2 \theta_1 \wedge \overline{\theta_2} + A_1 \overline{A_2} \theta_2 \wedge \overline{\theta_1} + |A_2|^2 \theta_2 \wedge \overline{\theta_2},
\end{equation}
and the equation $d^2 \kappa_{i\bar{i}} = 0$ implies that $A_1 = 0$ or $A_2 = 0$.

If $A_1 = 0,$ then the equations $d^2 \nu = 0,$ $d^2 \eta = 0, d^2 A_2 = 0$ imply that $A_2 =0$ as well, so we always have $A_2=0.$ Thus, $\mathsf{A}$ has a linear factor of multiplicity 5.

Now, solving the equations $d^2 \nu = 0,$ $d^2 \eta = 0,$ $d^2 \kappa = 0,$ we find that 
\begin{equation}
\kappa_{1\bar{2}} = i k \, \nu_{1\bar{2}},
\end{equation}
and that there exists a complex valued function $B_1$ on $\mathcal{Q}_1$ so that the following equations hold
\begin{equation}\label{eq:TypeAStructPre}
\begin{aligned}
d \nu_{1\bar{1}} &= 0, \\
d \nu_{1\bar{2}} &= \left(\kappa_{2\bar{2}} - \kappa_{1\bar{1}}\right) \wedge \nu_{1\bar{2}} - 2 i k \, \nu_{1\bar{1}} \wedge \nu_{1\bar{2}} - A_1 \, \overline{\nu_{1\bar{2}} \wedge \eta_1}, \\
d \eta_1 &= -\kappa_{1\bar{1}} \wedge \eta_1 + i k \, \nu_{1\bar{1}} \wedge \eta_1, \\
d \eta_2 &= - \kappa_{2\bar{2}} \wedge \eta_2 - 2 i k \overline{\nu_{1\bar{2}}} \wedge \eta_1 - i k \, \nu_{1\bar{1}} \wedge \eta_2, \\
d \kappa_{1\bar{1}} &= |A_1|^2 \eta_1 \wedge \overline{\eta_1}, \\
d \kappa_{2\bar{2}} &= 0, \\
d A_1 &= -3 A_1 \, \kappa_{1\bar{1}} + 2 A_1 \, \kappa_{2\bar{2}} - ik A_1 \, \nu_{1\bar{1}} + B_1 \overline{\eta_1}.
\end{aligned}
\end{equation}

\begin{prop}\label{prop:typeAexist}
	The $\U(2)^+$-structures of type $A$ exist locally and depend on 2 functions of 1 variable.
\end{prop}

\begin{proof}
	This is a simple application of Cartan's work on prescribed coframing problems (see \cite{BryEDSNotes} for a modern treatment), so we omit the details.
\end{proof}

\subsubsection{Integrating the structure equations}

The characteristic variety of the prescribed coframing problem associated to the structure equations (\ref{eq:TypeAStructPre}) is given by the pair of complex conjugate points $\eta_1,$ $\overline{\eta_1}.$ This suggests that is may be possible to find an description of the structures of type $A$ in terms of holomorphic data.

The conditions $d \nu_{1\bar{1}} = d \kappa_{2\bar{2}} = 0$ imply that locally there exist real-valued functions $r$ and $s$ on $M$ such that $\nu_{1\bar{1}} = i r,$ $\kappa_{2\bar{2}} = i s.$ Defining complex-valued 1-forms $\nu_1, \theta_1, \theta_2$ and a complex-valued function $W$ on $\mathcal{Q}_1$ by
\begin{equation}
\begin{aligned}
\nu_1 &= e^{-2kr-is} \nu_{1\bar{2}}, & \theta_1 &= e^{kr} \overline{\eta_1}, \\
W &= e^{-kr-2is} A_1, & \theta_2 & = e^{-kr+is} \eta_2, \\
X &= e^{-2kr-2is} B_1, & &
\end{aligned}
\end{equation}
equations (\ref{eq:TypeAStructPre}) imply
\begin{equation}\label{eq:TypeAcplx}
\begin{aligned}
d \nu_1 &= - \kappa_{1\bar{1}} \wedge \nu_1 - W \overline{\nu_1} \wedge \theta_1, \\
d \theta_1 &= \kappa_{1\bar{1}} \wedge \theta_1, \\
d \theta_2 & = 2ik \, \overline{\theta_1 \wedge \nu_1}, \\
d W &= -3W \, \kappa_{1\bar{1}} + X \theta_1.
\end{aligned}
\end{equation}
These are the structure equations for an $\mathrm{S}^1$-structure on the six-dimensional level sets of the function $r$ on $M.$ The forms $\nu_1,$ $\theta_1,$ and $\theta_2$ are semi-basic and are the components of the tautological 1-form for this structure, and the form $\kappa_{1\bar{1}}$ is a connection form. Let $N_c$ denote the level set $\left\lbrace r = c \right\rbrace \subset M.$

On $N_c,$ the distribution $\nu_1 = \theta_2 = 0$ is integrable. Each leaf of the resulting foliation has a metric given by the restriction of $|W|^2 \theta_1 \circ \overline{\theta_1},$ and the equations
\begin{equation}
\begin{aligned}
d \left( W \theta_1 \right) &= - 2 \kappa_{1\bar{1}} \wedge \left( W \theta_1 \right), \\
d \kappa_{1\bar{1}} &= - \left(W \theta_1 \right) \wedge \left(\overline{W \theta_1} \right),
\end{aligned}
\end{equation}
imply that this metric has constant curvature $-4$. This implies that locally there exists a complex coordinate $z_1$ on $N_c$ and a distinguished section of the $\mathrm{S}^1$-structure such that
\begin{equation}
W \theta_1 = \frac{dz_1}{1-|z_1|^2}, \:\:\:\: \kappa_{1\bar{1}} = \frac{1}{2} \frac{\overline{z_1} d z_1 - z_1 d \overline{z_1}}{1-|z_1|^2}.
\end{equation}
The equation
\begin{equation}
d \left( W dz_1 \right) = \frac{3}{2} \frac{Wz_1}{1-|z_1|^2} d \overline{z_1} \wedge dz_1
\end{equation}
implies that
\begin{equation}
W = \frac{h(z_1)}{\left( 1 - |z_1|^2 \right)^{3/2}},
\end{equation}
for some holomorphic function $h(z_1)$. Thus,
\begin{equation}\label{eq:TypeAthet1form}
\theta_1 = \frac{1}{h(z_1)} \left(1 - |z_1|^2 \right)^{1/2} d z_1.
\end{equation}

The $d\nu_1$ equation in (\ref{eq:TypeAcplx}) implies
\begin{equation}
d \left( \frac{i}{\sqrt{1-|z_1|^2}} \left( \nu_1 - z_1 \overline{\nu_1} \right) \right) =0,
\end{equation}
so we may introduce a local complex coordinate $z_2$ with
\begin{equation}\label{eq:TypeAsigform}
\nu_1 = \frac{i}{\sqrt{1-|z_1|^2}} \left( d z_2 - z_1 d \overline{z_2} \right).
\end{equation}

Finally, if we let $g(z_1)$ be a locally defined holomorphic function satisfying $g''(z_1)={1}/{h(z_1)},$ the $d\theta_2$ equation in (\ref{eq:TypeAcplx}) implies
\begin{equation}
d \left( \overline{\theta_2} -2k \left( g'(z_1)dz_2 - \left( z_1 g'(z_1)-g(z_1) \right) d \overline{z_2} \right) \right) = 0,
\end{equation}
so we may introduce a complex coordinate $z_3$ with
\begin{equation}\label{eq:TypeAthet2form}
\overline{\chi} = d z_3 + 2k \left( g'(z_1)dz_2 + \left( z_1 g'(z_1)-g(z_1) \right) d \overline{z_2} \right).
\end{equation}

We have now proven the first part of the following theorem. The second part follows by reversing the steps above.

\begin{thm}\label{thm:TypeAWeier}
	Let $M$ be a 7-manifold with a $\U(2)^+$-structure of type $A$. Then locally there exist complex coordinates $z_1, z_2, z_3,$ a real coordinate $r$, a constant $k$, and a holomorphic function $g(z_1)$ such that the $\G_2$ 3-form $\varphi$ is given by 
	\begin{equation}\label{eq:TypeAG2struct}
	\begin{split}
	\varphi =& \tfrac{i}{2} dr \wedge \left( e^{4k r} \nu_1 \wedge \overline{\nu_1} +  e^{-2kr} \theta_1 \wedge \overline{\theta_1} +  e^{2kr} \theta_2 \wedge \overline{\theta_2} \right)  \\
	&+ \tfrac{1}{2} e^{2kr} \left( \nu_1 \wedge \theta_1 \wedge \theta_2 + \overline{\nu_1 \wedge \theta_1 \wedge \theta_2} \right),
	\end{split}
	\end{equation}
	where $\nu_1, \theta_1,$ and $\theta_2$ are defined in terms of $z_1, z_2, z_3$ and $g(z_1)$ by equations (\ref{eq:TypeAthet1form}), (\ref{eq:TypeAsigform}), and (\ref{eq:TypeAthet2form}) above, where $h(z_1)={1}/{g''(z_1)}.$
	
	Conversely, on  $\R \times \mathbb{C}^{3}$ with coordinates $r, z_1, z_2, z_3$, let $k$ be a constant and $g(z_1)$ be a meromorphic function. Define $\nu_1, \theta_1,$ and $\theta_2$ using the formulas above. Let $\Sigma \subset \mathbb{C}$ be the subset of the unit disc where $g''(z_1)$ has no zeroes or poles.  Then (\ref{eq:TypeAG2struct}) defines the $\G_2$-structure associated to a $\U(2)^+$-structure of type $A$ on $\R \times \Sigma \times \mathbb{C}^2$.
\end{thm}

Theorem \ref{thm:TypeAWeier} may be thought of as a type of Weierstrass representation for the structures of type $A.$

\subsection{Type $N$}\label{ssect:TypeN}

Suppose $\mathcal{Q} \to M$ is a $\U(2)^{+}$-structure of type $N$, i.e. $A$ and $S$ vanish identically on $\mathcal{Q}$. Define complex-valued 1-forms $\pi^{a}_{b}$ and a real-valued 1-form $\rho$ by
\begin{equation}\label{eq:TypeNForms}
\pi^a_b = \kappa_{a\bar{b}} -\tfrac{1}{2} \delta_{a\bar{b}} \kappa_{c\bar{c}} - i k \nu_{a\bar{b}}, \:\:\: i \rho = \tfrac{1}{2} \kappa_{a\bar{a}},
\end{equation}
so that $\pi^a_a = 0$ (summation implied). Then, writing $\theta^a = \theta_a$ and $N^a = N_a,$ the $d\theta$ part of the structure equations (\ref{eq:StructEqU2pl}) becomes
\begin{equation}\label{eq:CartanITypeN}
d \theta^a = - \pi^a_b \wedge \theta^b - i \rho \wedge \theta^a + N^a \overline{\eps_{bc} \theta^b \wedge \theta^c}. 
\end{equation}
The $d \nu$ part of the structure equations (\ref{eq:StructEqU2pl}) together with some of the conditions arising from $d^2 \theta = 0$ imply
\begin{equation}\label{eq:CartanIITypeN}
d \pi^a_b = - \pi^a_c \wedge \pi^c_b
\end{equation}

Equations (\ref{eq:CartanITypeN}) and (\ref{eq:CartanIITypeN}) allow us to reduce the study $\U(2)^+$ structures of type $N$ to a 4-dimensional problem. Consider the action on the total space $\mathcal{Q}$ generated by the vector fields dual to the $\kappa$ and $\nu$ forms. Equation (\ref{eq:CartanIITypeN}) implies that this is a free action of the group $SL(2, \mathbb{C} ) \cdot S^1.$ Denote the 4-dimensional quotient manifold by $X.$ Then (\ref{eq:CartanITypeN}) implies that $\mathcal{Q}$ is the total space of an $\mathrm{SL}(2, \mathbb{C} ) \cdot \mathrm{S}^1$-structure over $X$. An $\mathrm{SL}(2, \mathbb{C} ) \cdot \mathrm{S}^1$-structure on a 4-manifold is equivalent to an almost complex structure and a holomorphic volume form defined up to phase, and the manifold $M$ is then the bundle of compatible metrics over $X,$ with fibers isometric to ${(\mathrm{SL}(2, \mathbb{C} ) \cdot S^1)}/{\mathrm{U}(2)} \cong \mathbb{H}^3.$ Conversely, suppose $\mathcal{Q} \to X$ is an $\mathrm{SL}(2, \mathbb{C} ) \cdot \mathrm{S}^1$-structure over a 4-manifold $X$ together with a connection $\pi, \rho$ satisfying equations (\ref{eq:CartanITypeN}) and (\ref{eq:CartanIITypeN}). Then we may define $M = {\mathcal{Q}}/{\U(2)}$ and consider $\mathcal{Q} \to M$ as a $\mathrm{U}(2)^{+}$-structure of type $N$ over $M$ by reversing the equations in (\ref{eq:TypeNForms}).

We now investigate the consequences of equations (\ref{eq:CartanITypeN}) and (\ref{eq:CartanIITypeN}). The equation $d^2 \theta=0$ implies
\begin{equation}\label{eq:TypeNdNdrho}
\begin{aligned}
d N^a &= -3iN^a \rho - N^b \pi^a_b + w \theta^a + u^a_{\bar{b}} \overline{\theta^b}, \\
d \rho &= i w \overline{\theta^1 \wedge \theta^2} - i \overline{w} \theta^1 \wedge \theta^2 - 4 i \eps_{ab} \overline{\eps_{cd}} N^a \overline{N^c} \theta^b \wedge \overline{\theta^d},
\end{aligned}
\end{equation}
for some complex-valued functions $w$ and $u^a_{\bar{b}}$ on $\mathcal{Q}.$ Next, the equations $d^2 N = 0$ and $d^2 \rho = 0$ imply
\begin{subequations}
\begin{align}
d w &= -2iw \rho + 2 \eps_{ab} u^a_{\bar{c}} \overline{N^c} \theta^b + 6 \eps_{ab} \overline{w} N^a \theta^b + \overline{v_a \theta^a}, \label{eq:TypeNdw} \\
d u^a_{\bar{b}} &= - u^c_{\bar{b}} \pi^a_c  + u^a_{\bar{c}} \overline{\pi^c_b} - 4 i u^a_{\bar{b}} \rho + \overline{v_b} \theta^a + 12 \overline{\eps_{bc}} \eps_{de} N^a N^e \overline{N^c} \theta^d \label{eq:TypeNdu} \\
& + x^{a}_{\bar{b}\bar{c}} \overline{\theta^c} - 2 w N^a \overline{\eps_{bc}} \overline{\theta^c}, \nonumber
\end{align}
\end{subequations}
for some complex-valued functions $v_a$ and $x^a_{\bar{b}\bar{c}} = x^a_{\bar{c}\bar{b}}$ on $\mathcal{Q}$.

Next, the identity $d^2 w \wedge \theta^1 \wedge \theta^2$ implies
\begin{equation}
8 |w|^2 + 3 N^av_a + 3 \overline{N^av_a} + u^a_{\bar{b}} \overline{u^b_{\bar{a}}} = 0.
\end{equation}
This is an algebraic restriction on the 3-jet of a type $N$ structure only revealed after taking four derivatives. Thus, the differential system associated to structures of type $N$ is not involutive.

Due to algebraic difficulties encountered in the prolongation process, we will not investigate the general type $N$ structure further in this article. Instead, we will consider the special case with $w=0$, where we find that the structure equations can be integrated explicitly.

\subsubsection{The case $w=0$}\label{sssect:TypeNwzero}

We now assume that $\mathcal{Q}$ is a $\mathrm{U}(2)^{+}$-structure of type $N$ satisfying the additional condition that $w=0$. Equation (\ref{eq:TypeNdw}) implies that
\begin{equation}
u^a_{\bar{b}} \overline{N^b} = 0, \:\:\:\: a = 1, 2,
\end{equation}
and it follows that there exist complex-valued functions $u^a$ on $\mathcal{Q}$ such that
\begin{equation}\label{eq:TypeNuone}
u^a_{\bar{b}} = 2 \, \overline{\eps_{bc}} u^a \overline{N^c}.
\end{equation}
Substituting equation (\ref{eq:TypeNuone}) into equation (\ref{eq:TypeNdu}) we find
\begin{equation}
\eps_{ab} N^a u^b = 0,
\end{equation}
so there exists a complex-valued function $u$ on $\mathcal{Q}$ with
\begin{equation}
u^a = u \, N^a.
\end{equation}

Equations (\ref{eq:TypeNdNdrho}) now read
\begin{equation}\label{eq:TypeNdNdrhowzero}
\begin{aligned}
d N^a &= -3iN^a \rho - N^b \pi^a_b + 2 u \overline{\eps_{bc}} N^a \overline{N^c} \overline{\theta^b}, \\
d \rho &= - 4 i \eps_{ab} \overline{\eps_{cd}} N^a \overline{N^c} \theta^b \wedge \overline{\theta^d},
\end{aligned}
\end{equation}
and the identity $d^2 N^a = 0$ implies
\begin{equation}\label{eq:TypeNduwzero}
d u = -4iu \rho -2 \left(3 - |u|^2 \right) \eps_{ab} N^a \theta^b - 2 v \overline{\eps_{ab} N^a \theta^b}.
\end{equation} 

\begin{prop}\label{prop:TypeNwzeroexistgen}
	The $\mathrm{U}(2)^+$-structures of type $N$ with $w=0$ exist locally and depend on 2 functions of 1 variable.
\end{prop}
\begin{proof}
	This is a simple application of Cartan's work on prescribed coframing problems (see \cite{BryEDSNotes} for a modern treatment), so we omit the details.
\end{proof}

The characteristic variety of the system associated to structures of type $N$ with $w=0$ consists two conjugate points, $\eps_{ab} N^a \theta^b$ and $\overline{\eps_{ab} N^a \theta^b}.$ This suggests that it may be possible to describe these structures in terms of holomorphic data.

Motivated by the form of the characteristic variety, define a complex-valued 1-form $\sigma$ and a real-valued 1-form $\alpha$ by
\begin{equation}\label{eq:TypeNsigrhodef}
\begin{aligned}
\sigma &= -2 \eps_{ab} N^a \theta^b , \\
\alpha &= \rho + \tfrac{i}{4} \left( u \overline{\sigma} - \overline{u} \sigma \right).
\end{aligned}
\end{equation}

The structure equations (\ref{eq:CartanITypeN}) and (\ref{eq:TypeNdNdrhowzero}) give that
\begin{align}
d \sigma &= -4 i \alpha \wedge \sigma, \\
d \alpha &= \tfrac{i}{2} \sigma \wedge \overline{\sigma}.
\end{align}
Thus, the distribution defined by the vanishing of $\sigma$ is integrable, and the metric defined by $\sigma \circ \overline{\sigma}$ on the leaf space has constant curvature $-4$. It follows that there exists a complex-valued function $z_1$ and a real-valued function $s$ on $\mathcal{Q}$ such that
\begin{equation}
\sigma = \frac{e^{is} dz_1}{1-|z_1|^2}, \:\:\:\: \alpha = \frac{i}{4} \frac{z_1 d \overline{z_1} - \overline{z_1} d z_1}{1-|z_1|^2} + ds.
\end{equation}
We restrict to the locus where $s=0.$ This corresponds to reducing the $SL(2,\mathbb{C}) \cdot S^1$-structure $\mathcal{Q}$ over $X$ to an $SL(2,\mathbb{C})$-structure $\mathcal{Q}' \subset \mathcal{Q}$. The function $z_1$ is then a complex coordinate on $X.$

Equation (\ref{eq:TypeNduwzero}) implies
\begin{align}
d \left( \overline{u} \sigma \right) = 3 \overline{\sigma} \wedge \sigma,
\end{align}
so if we let $p = {\overline{u}}/({1-|z_1|^2}),$ we find
\begin{align}
\frac{\partial p}{\partial \overline{z_1}} = \frac{3}{\left( 1-|z_1|^2 \right)^2},
\end{align}
which implies
\begin{align}
p = \frac{3}{z_1 \left( 1 - |z_1|^2 \right)} + 4 \, h(z_1),
\end{align}
where $h(z_1)$ is a locally defined holomorphic function (the factor of 4 is included for later convenience). Thus
\begin{align}
\overline{u} = \tfrac{3}{z_1} + 4 h(z_1) \left( 1 - |z_1|^2 \right).
\end{align}
We then have
\begin{align}
\rho =  \frac{i}{4} \frac{3-|z_1|^2}{1-|z_1|^2} \left( \frac{dz_1}{z_1} - \frac{d \overline{z_1}}{\overline{z_1}} \right) + i\left( h(z_1) dz_1 - \overline{h(z_1) dz_1} \right).
\end{align}
Let $f(z_1) = \int h(z_1) dz_1$, and define $g(z_1) = z_1^{\frac{3}{4}} e^{f(z_1)},$ so that $h(z_1) = \frac{g'(z_1)}{g(z_1)} - \frac{3}{4z_1}.$ We now observe that the function $G(z_1)$ defined by
\begin{align}
G(z_1) = \log \frac{g(z_1)^2 |g(z_1)|^2}{\left(1-|z_1|^2 \right)^{\frac{3}{2}}}
\end{align}
satisfies $d(G(z_1)) = u \overline{\sigma} - 3 i \rho$. Note also that the connection $\kappa$ on $\mathcal{Q}'$ is flat. Let $\psi : \mathcal{Q}' \to \mathrm{SL}(2, \mathbb{C})$ integrate $\kappa,$ i.e. $\psi^{-1} d \psi = \kappa$. The map $\psi$ is unique up to left multiplication by an element of $\mathrm{SL}(2, \mathbb{C}).$ By acting by $\psi$ appropriately we may assume that $\kappa=0$ in our structure equations.

Thus
\begin{align}\label{eq:TypeZdzmod}
d \left( \begin{array}{c}
e^{-G(z_1)} N_1 \\
e^{-G(z_1)} N_2
\end{array} \right) = 0,
\end{align}
so there exist constants $k_1,$ and $k_2$ such that $ \left( e^{-G(z_1)} N_1 , e^{-G(z_1)} N_2 \right) = \left( k_1, k_2 \right).$ From the left-multiplication ambiguity in the definition of $\psi$ we get an $\mathrm{SL}(2, \mathbb{C})$ action on $\left( k_1, k_2 \right)$, and we may assume that $\left( k_1, k_2 \right) = \left(0,1 \right)$ (the case where $N=0$ identically having been covered in \S\ref{sssect:BryEx}).

By the definition of $\sigma$ (\ref{eq:TypeNsigrhodef}) we have
\begin{align}
\theta_1 = \frac{\sigma}{N_2} = \frac{e^{G(z_1)}}{1-|z_1|^2} dz_1= \frac{\left( 1 - |z_1|^2 \right)}{g(z_1)^2 |g(z_1)|^2} ds.
\end{align}
The first structure equation (\ref{eq:CartanITypeN}) implies
\begin{align*}
d \left( \begin{array}{c}
\theta_2 \\
\overline{\theta_2}
\end{array} \right) =& - \left( \begin{array}{cc}
-i \, \mathrm{Im} \left( \left( \frac{1}{1-|z_1|^2} + \frac{2g(z_1)}{g'(z_1)} \right) dz_1 \right) & - \frac{g(z_1)^2}{\overline{g(z_1)}^2} \frac{1}{1-|z_1|^2}d\overline{z_1} \\
- \frac{\overline{g(z_1)}^2}{g(z_1)^2} \frac{1}{1-|z_1|^2}d{z_1} & i \, \mathrm{Im} \left( \left( \frac{1}{1-|z_1|^2} + \frac{2g(z_1)}{g'(z_1)} \right) d z_1 \right)
\end{array} \right) \wedge \left( \begin{array}{c}
\theta_2 \\
\overline{\theta_2}
\end{array} \right). \nonumber
\end{align*}
Then
\begin{align}
d \left( \frac{i}{\left( 1 - |z_1|^2 \right)^{\frac{1}{2}}} \left( \frac{\overline{g(z_1)}}{g(z_1)} \theta_2 - \frac{\overline{z_1}g(z_1)}{\overline{g(z_1)}}  \overline{\theta_2} \right) \right) = 0,
\end{align}
and we may introduce a complex-coordinate $z_2$ on $X$ with
\begin{align}
\theta_2 = \left( \frac{i}{\left( 1 - |z_1|^2 \right)^{\frac{1}{2}}} \frac{g(z_1)}{\overline{g(z_1)}} \left( d z_2 - \overline{z_1 dz_2} \right) \right).
\end{align}

We have now proven the first part of the following theorem. The second part follows by reversing the steps above.

\begin{thm}\label{thm:TypeNwzero}
	Let $X$ be a 4-manifold with an $\mathrm{SL}(2, \mathbb{C} ) \cdot \mathrm{S}^1$-structure satisfying equations (\ref{eq:CartanITypeN}) and (\ref{eq:CartanIITypeN}) and the condition $w=0$. Then locally there exist complex coordinates $z_1,z_2$ and a holomorphic function $g(z_1)$ such that
	\begin{align}\label{eq:TypeNthetdefn}
	\theta_1 &= \frac{\left( 1 - |z_1|^2 \right)}{g(z_1)^2 |g(z_1)|^2} dz_1, \\
	\theta_2 &=  \left( \frac{i}{\left( 1 - |z_1|^2 \right)^{\frac{1}{2}}} \frac{g(z_1)}{\overline{g(z_1)}} \left( d z_2 - \overline{z_1 dz_2} \right) \right) \nonumber
	\end{align}
	span the space of $(1,0)$-forms, and the complex volume form is given by $\theta_1 \wedge \theta_2$.
	
	Conversely, on  $\mathbb{C}^{2}$ with coordinates $z_1, z_2$, let $g(z_1)$ be a holomorphic function. Define $\theta_1,$ and $\theta_2$ by (\ref{eq:TypeNthetdefn}), and let $\Sigma$ denote the subset of the unit disc in $\mathbb{C}$ where $g(s)$ has no zeroes or poles.  Then the forms $\theta_1$ and  $\theta_2$ define an $SL(2, \mathbb{C} ) \cdot S^1$-structure satisfying equations (\ref{eq:CartanITypeN}) and (\ref{eq:CartanIITypeN}) and $w=0$ on $\Sigma \times \mathbb{C}.$
\end{thm}

Theorem \ref{thm:TypeNwzero} may be thought of as a type of Weierstrass representation for these structures. The 7-manifold $M$ is recovered as the total space of the bundle of compatible metrics on $X,$ as in the discussion following equation (\ref{eq:CartanIITypeN}).

\subsection{Type $S$}\label{ssect:TypeS}

Suppose $\mathcal{Q} \to M$ is a $\U(2)^+$-structure of type $S$, i.e. $A$ and $N$ vanish identically on $\mathcal{Q}.$ By Remark \ref{rmk:rescale}, we may rescale the structure so that $k = 1/2,$ and we fix this scale for the remainder of this section.

It will be useful to use real notation in this section. Let $1 \leq a, b \ldots \leq 2$ and $1 \leq i,j, \ldots \leq 3$ be indices with the specified ranges, and define functions $S_{aij} = S_{aji}$ by
\begin{small}
\begin{equation*}
\begin{aligned}
S_{111} &= -4 \Re S_{1122}, & S_{211} &= -4 \Im S_{1122}, \\
S_{112} &= -2 \Im S_{1112} -2 \Im S_{1222}, & S_{212} &= 2 \Re S_{1112} + 2 \Re S_{1222}, \\
S_{113} &= - 2 \Re S_{1112} + 2 \Re S_{1222}, & S_{213} &= - 2 \Im S_{1112} + 2 \Im S_{1222}, \\
S_{122} &= \Re S_{1111} + 2 \Re S_{1122} + \Re S_{2222}, & S_{222} &= \Im S_{1111} + 2 \Im S_{1122} + \Im S_{2222}, \\
S_{123} &= - \Im S_{1111} + \Im S_{2222}, & S_{223} &= \Re S_{1111} - \Re S_{2222}, \\
S_{133} &= - \Re S_{1111} + 2 \Re S_{1122} - \Re S_{2222}, & S_{233} &= - \Im S_{1111} + 2 \Im S_{1122} - \Im S_{2222}.
\end{aligned}
\end{equation*}
\end{small}
Note that $S_{aii} = 0$ (summation implied). Next, define 1-forms $\sigma_{ai},$ $\psi_{ij} = -\psi_{ji},$ and $\rho_{ab} = - \rho_{ba}$ by 
\begin{equation}
\begin{aligned}
\psi_{23} &=  - i \left( \kappa_{1\bar{1}} - \kappa_{2\bar{2}} \right), & \sigma_{ai} &= S_{aij} \omega_j, \\
\psi_{31} &= 2 \Re \kappa_{1\bar{2}}, & \rho_{12} &= i \left( \kappa_{1\bar{1}} + \kappa_{2\bar{2}} \right), \\
\psi_{12} &= -2 \Im \kappa_{1\bar{2}}. & &
\end{aligned}
\end{equation}

In this notation, the first structure equations (\ref{eq:StructEqU2pl}) become
\begin{subequations}
	\begin{align}
	d \omega_i &= -\psi_{ij} \wedge \omega_j, \label{eq:TypeSdomo} \\
	d \left( \begin{array}{c}
	\omega_4 \\
	\omega_5 \\
	\omega_6 \\
	\omega_7
	\end{array}
	\right) &= - \mu \wedge \left( \begin{array}{c}
	\omega_4 \\
	\omega_5 \\
	\omega_6 \\
	\omega_7
	\end{array}
	\right), \label{eq:TypeSdomfo}
	\end{align}
\end{subequations}
where $\mu$ is the $4 \times 4$ traceless matrix
\begin{small}
	\begin{equation*}
	\frac{1}{2} \left( \begin {array}{cccc} -\sigma_{{12}}-\sigma_{{23}}-\omega_{{1}} & \sigma_{{13}} - \sigma_{{22}} - \psi_{{23}} + \rho_{{12}} & \sigma_{{21}} + \psi_{{31}} - \omega_{{3}} & \sigma_{{11}} + \psi_{{12}} - \omega_{{2}} \\
	\sigma_{{13}} - \sigma_{{22}} + \psi_{{23}} - \rho_{{12}} & \sigma_{{12}} + \sigma_{{23}} - \omega_{{1}} & -\sigma_{{11}} + \psi_{{12}} - \omega_{{2}} & \sigma_{{21}} - \psi_{{31}} + \omega_{{3}} \\
	\sigma_{{21}} - \psi_{{31}} - \omega_{{3}} & -\sigma_{{11}} -\psi_{{12}} - \omega_{{2}} & \sigma_{{23}} - \sigma_{{12}} + \omega_{{1}} & \sigma_{{22}} + 
	\sigma_{{13}} - \psi_{{23}} - \rho_{{12}} \\
	\sigma_{{11}} - \psi_{{12}} - \omega_{{2}} & \sigma_{{21}} + \psi_{{31}} + \omega_{{3}} & \sigma_{{22}} + \sigma_{{13}} + \psi_{{23}} + \rho_{{12}} & -\sigma_{{23}} + \sigma_{{12}} + \omega_{{1}} \end{array}
	\right).
	\end{equation*}
\end{small}

The vanishing of $d^2 \omega,$ $d^2 \psi,$ and $d^2 \rho$ imply
\begin{equation}\label{eq:TypeScII}
\begin{aligned}
d \rho_{ab} &= - \sigma_{ai} \wedge \sigma_{bi}, \\
d \psi_{ij} &= - \psi_{ik} \wedge \psi_{kj} - \omega_i \wedge \omega_j - \sigma_{ai} \wedge \sigma_{aj}, \\
d \sigma_{ai} &= - \rho_{ab} \wedge \sigma_{bi} - \psi_{ij} \wedge \sigma_{aj}, \\
d S_{aij} &= - S_{bij} \rho_{ab} - S_{aik} \psi_{jk} - S_{ajk} \psi_{ik} + C_{aijk} \omega_{k},
\end{aligned}
\end{equation}
where $C_{aijk}$ are functions on $\mathcal{Q}$ satisfying $C_{aijk} = C_{ajik} = C_{aikj}$ and $C_{aijj} = 0.$

We are now in a position to give an existence and generality result for these structures.

\begin{prop}\label{prop:TypeSexistgen}
	The $\U(2)^+$-structures of type $S$ exist locally and depend on 4 functions of 2 variables.
\end{prop}

\begin{proof}
	This is an application of Cartan's work on prescribed coframing problems \cite{BryEDSNotes}. The `primary invariants' are the functions $S_{aij}$, while the `free derivatives' are the functions $C_{aijk}$. The tableau of free derivatives is involutive with $s_1 =10, s_2=4,$ and $s_k = 0$ for $k \geq 3.$
\end{proof}

\subsubsection{Maximal submanifolds}

Recall that the $\U(2)^+$-structure $\mathcal{Q} \to M$ always has the property that $M$ is foliated by 4-dimensional leaves of the distribution defined by $\omega_1 = \omega_2 = \omega_3 = 0,$ and that these leaves are coassociative submanifolds of $M.$ A simple computation using equations (\ref{eq:TypeSdomfo}) and (\ref{eq:TypeScII}) shows that the metric $g_{\varphi}$ restricted to these leaves is flat when the $\U(2)^+$-structure is type $S$. Thus, the $\U(2)^+$-structures of type $S$ all carry canonically defined semi-flat coassociative fibrations. Baraglia \cite{BaragSemi} has studied semi-flat coassociative fibrations in general closed $\G_2$-structures, and proven the following result, which provides a link with the theory of space-like submanifolds in pseudo-Riemannian manifolds with indefinite signature (see also recent work of Donaldson \cite{DonaldsonAdiabat17}).

\begin{thm}[\cite{BaragSemi}]\label{thm:Barag}
	Fix a volume form on $\R^4,$ so that $\Lambda^2 \R^4$ is identified with the pseudo-Euclidean space $\R^{3,3}.$ Let $B$ be an oriented 3-manifold and $u : B \to \R^{3,3} = \Lambda^2 \R^4$ be a space-like immersion, and let $c$ be a positive constant. Let $h$ be the pullback metric on $B$ with volume form $\mathrm{vol}_h.$ Let $M = B \times \R^4,$ and define a $\G_2$-structure $\varphi$ on $M$ by
	\begin{equation}
	\varphi = c \, \mathrm{vol}_h + d u,
	\end{equation}
	where $u$ is considered as a 2-form on $M.$ Then $(M, \varphi)$ is a closed $\G_2$-structure with a semi-flat coassociative fibration. Conversely, any closed $\G_2$-structure with semi-flat coassociative fibration is locally of this form.
	
	Furthermore, the $\G_2$-structure $\varphi$ constructed in this way is torsion-free if and only if the immersion $u$ is maximal (meaning that its mean curvature vector vanishes).
\end{thm}

\begin{remark}
	It is not difficult to show that any closed $\G_2$-structure with a semi-flat coassociative fibration has special torsion of positive type (see \S\ref{ssect:SpecTors}). This is interesting in light of the fact that semi-flat coassociative fibrations are preserved by the Laplacian flow \cite{LotLamLap19}.
\end{remark}

In light of Theorem \ref{thm:Barag}, it is natural to ask under what conditions does a space-like immersion $u : B \to \R^{3,3}$ give rise to an ERP closed $\G_2$-structure (or, equivalently, a $\U(2)^+$-structure of type $S$). This is answered by the following theorem, the proof of which will contain a proof of the relevant case of Theorem \ref{thm:Barag}.

\begin{thm}\label{thm:TypeSCons}
	Let $\mathcal{Q} \to M$ be a $\mathrm{U}(2)^+$-structure of type $S.$ Then, up to rescaling, the associated space-like maximal submanifold of $\R^{3,3}$ is a maximal submanifold of the quadric $Q \subset \R^{3,3}$ consisting of vectors of norm $-1$ in $\R^{3,3}.$ Conversely, any maximal submanifold of $Q$ endowed with its pseudo-Riemannian metric of signature $(2,3)$ and constant curvature $-1$ gives rise to a $\mathrm{U}(2)^+$-structure of type $S$ by following the construction of Theorem \ref{thm:Barag}.
\end{thm}

\begin{proof}
	Let
	\begin{equation*}
	-(e^0)^2 + (e^1)^2 + (e^2)^2 + (e^3)^2 - (e^4)^2 - (e^5)^2
	\end{equation*}
	 define the pseudo-Euclidean metric on $\R^{3,3}$ and let $\mathrm{ASO}(3,3)$ denote the corresponding group of rigid motions. Then $\R^{3,3}$ is the homogeneous space $\mathrm{ASO}(3,3) / \mathrm{SO}(3,3),$ and the quadric $Q$ is the orbit of the vector $e^0$ under the action of $\mathrm{SO}(3,3).$ The stabiliser of $e^0$ is the group $\mathrm{SO}(3,2),$ and $Q$ is thus the symmetric space $\SO(3,3)/\SO(3,2).$
	 
	 Let $\mathcal{Q} \to M$ be a $\mathrm{U}(2)^+$-structure of type $S,$ let the manifold $B$ be a leaf of the foliation defined by $\omega_4 = \ldots = \omega_7 = 0$ on $M,$ and let $\mathcal{Q}(B)$ denote the restriction of $\mathcal{Q}$ to $B.$ Now, equations (\ref{eq:TypeSdomo}) and (\ref{eq:TypeScII}) imply that the $\mathfrak{so}(3,3)$-valued matrix
	 \begin{equation}\label{eq:TypeSgam}
	 \gamma = \left(\begin{array}{ccc}
	 0 & \omega_j & 0 \\
	 \omega_i & \psi_{ij} & \sigma_{bi} \\
	 0 & \sigma_{ai} & \rho_{ab}
	 \end{array}\right)
	 \end{equation}
	 restricted to $\mathcal{Q}(B)$ satisfies the Maurer-Cartan equation $d \gamma = -\gamma \wedge \gamma.$ Thus, by Cartan's Theorem on Maps into Lie Groups (\ref{thm:MaurerCartan}), locally there exists a map $f : \mathcal{Q}(B) \to \SO(3,3)$ such that $f^{-1} df = \gamma.$ The map $f$ identifies $\mathcal{Q}(B)$ with the adapted coframe bundle of a space-like immersion $u: B \to \SO(3,3)/\SO(3,2).$ The tensor $S_{aij}$ is identified with the second fundamental form of $u(B),$ and the condition $S_{aii} = 0$ is equivalent to the vanishing of the mean curvature vector of $u(B).$
	 
	 There is an exceptional isogeny $\mathrm{SL}(4, \R) \to \mathrm{SO}(3,3)$ coming from the action of $\mathrm{SL}(4,\R)$ on $\Lambda^2 \R^4.$ Let $\tilde{f}$ be a (possibly only locally defined) lift of $f$ to $\mathrm{SL}(4, \R).$ This can be arranged so that $\widetilde{f}^{-1} d \tilde{f} = \mu,$ and it follows from equation (\ref{eq:TypeSdomfo}) that
	 	\begin{equation}
	 	d \left( \tilde{f} \left( \begin{array}{c}
	 	\omega_4 \\
	 	\omega_5 \\
	 	\omega_6 \\
	 	\omega_7
	 	\end{array}
	 	\right) \right) = 0,
		\end{equation}
		and we may introduce coordinates $x_1, \ldots, x_4$ so that
		\begin{equation}
		\left( \begin{array}{c}
		\omega_4 \\
		\omega_5 \\
		\omega_6 \\
		\omega_7
		\end{array}
		\right) = \widetilde{f}^{-1} \left( \begin{array}{c}
		d x_1 \\
		d x_2 \\
		d x_3 \\
		d x_4
		\end{array}
		\right).
		\end{equation}
		It follows that the $\G_2$-structure $\varphi$ has the form claimed. This proves the first part of the theorem.
		
		The proof of the converse statement is a straightforward reversal of the above steps. The key point is to adapt frames to the maximal space-like immersion so that the Maurer-Cartan form of $\SO(3,3)$ takes the form (\ref{eq:TypeSgam}).
\end{proof}

\begin{remark}
	Two $\mathrm{U}(2)^+$-structures of type $S$ are locally equivalent if and only if the corresponding space-like immersions are locally equivalent up to rigid motion in $Q$.
\end{remark}

\subsubsection{Examples}\label{sssect:TypeSEgs}

We will now use Theorem \ref{thm:TypeSCons} to give examples of ERP closed $\G_2$-structures, and reinterpret the known examples in this context. The computations in this section will also prove useful in \S\ref{ssect:ERPHomog}.

\begin{example}
	As proven in \S\ref{sssect:BryEx}, Bryant's first example \cite{Bry05} of an ERP closed $\G_2$-structure is the unique local model for $\U(2)^{+}$-structures of type $S$ with $S$ identically zero. Since $S$ represents the second fundamental form of the associated maximal immersion $B \to {\SO(3,3)}/{\SO(3,2)},$ for this example the submanifold $B$ is totally geodesic and thus isometric to hyperbolic 3-space.
\end{example}

\begin{example}\label{eg:GJ}
	After Bryant's example, the next discovered example of an ERP closed $\G_2$-structure was given by Lauret \cite{LauretLap}. It is homogeneous under the action of a solvable Lie group. A calculation shows that Lauret's example is of type $S$ and that it has the interesting property that the roots of the tensor $\mathsf{S}$ (see (\ref{eq:U2plBaseTens})), viewed as a homogeneous polynomial in $\eta_1$ and $\eta_2,$ are arranged in a regular tetrahedron inscribed in the Riemann sphere. In fact, this property uniquely characterises Lauret's example.
	
	\begin{prop}
		Suppose that $\mathcal{Q} \to M$ is a $\U(2)^{+}$-structure of type $S$ such that for all points on $M$ the roots of the polynomial $\mathsf{S}$ (see (\ref{eq:U2plBaseTens})) are arranged on a regular tetrahedron inscribed in the Riemann sphere. Then $M$ is locally equivalent to Lauret's example.
	\end{prop}

\begin{proof}
	Let $\mathcal{Q} \to M$ be a $\U(2)^{+}$-structure of type $S$ such that for all points on $M$ the roots of the polynomial $\mathsf{S}$ (\ref{eq:U2plBaseTens}) are arranged on a regular tetrahedron inscribed in the Riemann sphere. The elements of the unit sphere in $\mathrm{Sym}^4 \C^2$ satisfying this condition comprise a single $\U(2)$ orbit, and it follows that at each point of $M$ there is a coframe $(\tilde{\nu}, \tilde{\eta})$ for which
	\begin{equation}
	\mathsf{S} = r \left( \tilde{\eta}_1^4 - 2 \sqrt{2} \tilde{\eta}_1 \tilde{\eta}_2^3 \right)
	\end{equation}
	for some positive $r$ on $M$. The collection of all such coframes forms a principal $\mathrm{A}_4$-subbundle $\mathcal{Q}_1$ of $\mathcal{Q},$ and we shall work on this subbundle.
	
	On $\mathcal{Q}_1,$ 
	\begin{equation}
	\left( \begin{array}{ccc}
	\sigma_{11} & \sigma_{12} & \sigma_{13} \\
	\sigma_{21} & \sigma_{22} & \sigma_{23}
	\end{array} \right) = \left( \begin{array}{ccc}
	-\sqrt {2}r\omega_{{3}}&r\omega_{{2}}&-\sqrt {2}r\omega_{{1}}-r\omega_{{3}}\\ -\sqrt {2}r\omega_{{2}}&-\sqrt {2}r\omega_{{1}}+r\omega_{{3}}&r\omega_{{2}}
	\end{array}
	\right),
	\end{equation}
	and the equations (\ref{eq:TypeScII}) imply
	\begin{equation}
	\psi_{ij} = 0, \:\:\:\: \rho_{ab} = 0, \:\:\:\: r = \tfrac{1}{2} 
	\end{equation}
	We are now left with structure equations
	\begin{small}
		\begin{align}
		d \omega _1 = d& \omega_2 = d \omega_3 =0, \label{eq:LauretFlat} \\
		d  \left( \begin{array}{c}
		\omega_4 \\
		\omega_5 \\
		\omega_6 \\
		\omega_7 
		\end{array} \right) =& - \frac{1}{2} \left( \begin{array}{cccc}
		- \omega_{{1}} - \omega_{{2}} & - \omega_{{3}} & - \tfrac{1}{\sqrt{2}} \omega_{{2}} - \omega_{{3}} & - \omega_{{2}} - \tfrac{1}{\sqrt{2}} \omega_{{3}} \\
		- \omega_{{3}} & - \omega_{{1}} + \omega_{{2}} & - \omega_{{2}} + \tfrac{1}{\sqrt{2}} \omega_{{3}} & - \tfrac{1}{\sqrt{2}} \omega_{{2}} + \omega_{{3}} \\
		- \tfrac{1}{\sqrt{2}} \omega_{{2}} - \omega_{{3}} & - \omega_{{2}} + \tfrac{1}{\sqrt{2}} \omega_{{3}} & \omega_{{1}} & - \sqrt{2} \omega_{{1}} \\
		- \omega_{{2}} - \tfrac{1}{\sqrt{2}} \omega_{{3}} & - \tfrac{1}{\sqrt{2}} \omega_{{2}} + \omega_{{3}} & - \sqrt{2} \omega_{{1}} & \omega_{{1}}
		 \end{array} \right) \wedge \left( \begin{array}{c}
		\omega_4 \\
		\omega_5 \\
		\omega_6 \\
		\omega_7 
		\end{array} \right).
		\end{align}
	\end{small}
	These are the structure equations of a 7-dimensional solvable Lie algebra $\mathfrak{g}$. It then straightforward to show that the resulting ERP closed $\G_2$-structure on $\G$ is equivalent to Lauret's example on $\R^3 \ltimes \R^4$. \end{proof}

An integration of the Maurer-Cartan form in this case gives that the maximal submanifold $B \to Q$ is given by
\begin{equation}
\begin{aligned}
\mathbb{R}^3 \to Q , \:\:\: (x,y,z) \mapsto \tfrac{1}{\sqrt{3}} \left( \cosh x, \sinh x, \sinh y, \sinh z, \cosh y, \cosh z \right), \nonumber
\end{aligned}
\end{equation}
which can be thought of as an analogue of the Clifford torus in this setting. The induced metric is flat, as is clear from (\ref{eq:LauretFlat}).
\end{example}

Lauret's example suggests that looking for $\U(2)^{+}$-structures of type $S$ where the second fundamental form $S$ has a non-trivial $\U(2)$-stabiliser may be fruitful. The following examples are of this kind.

\begin{example}(\emph{Quadruple root}) We now suppose that $\mathcal{Q} \to M$ is a $\U(2)^{+}$-structure of type $S$ such that for all points on $M$ the polynomial $\mathsf{S}$ (\ref{eq:U2plBaseTens}) has a quadruple root. Similarly to the previous example, at each point on $M$ we may find a coframe $\left( \tilde{\nu}, \tilde{\eta} \right)$ so that
	\begin{equation}
	\mathsf{S} = (r_1 + i r_2) \tilde{\eta_1}^4,
	\end{equation}
	for some real $r_1, r_2$ on $M.$ The collection of all such coframes forms a principal $\mathrm{T}^2$-subbundle $\mathcal{Q}_1$ of $\mathcal{Q},$ and we shall work on this subbundle.
	
	On $\mathcal{Q}_1,$
	\begin{equation}
	\left( \begin{array}{ccc}
	\sigma_{11} & \sigma_{12} & \sigma_{13} \\
	\sigma_{21} & \sigma_{22} & \sigma_{23}
	\end{array} \right) = \left( \begin{array}{ccc}
	0 & r_{{1}} \omega_{{2}} - r_{{2}} \omega_{{3}} & -r_{{2}} \omega_{{2}} -r_{{1}} \omega_{{3}} \\ 0 & r_{{2}} \omega_{{2}} + r_{{1}} \omega_{{3}} & r_{{1}} \omega_{{2}} - r_{{2}} \omega_{{3}}	
	\end{array}
	\right),
	\end{equation}
	and the equations (\ref{eq:TypeScII}) imply
	\begin{equation}
	\psi_{12} = p_1 \omega_2 - p_2 \omega_3, \:\:\:\: \psi_{31} = -p_2 \omega_2 - p_1 \omega_3,
	\end{equation}
	for some functions $p_1, p_2$ on $\mathcal{Q}_1.$ Using equations (\ref{eq:TypeScII}) again, we have
	\begin{equation}\label{eq:TypeSQuaddrdp}
	\begin{aligned}
	d r_1 &= \left( p_{{1}}r_{{1}}-p_{{2}}r_{{2}} \right) \omega_{{1}}+q_{{1}} \omega_{{2}}-q_{{2}} \omega_{{3}}+2\,r_{{2}}\psi_{{23}}-r_{{2}}\rho_{{12}}, \\
	d r_2 & = \left( p_{{1}}r_{{2}}+p_{{2}}r_{{1}} \right) \omega_{{1}}+q_{{2}}
	\omega_{{2}}+q_{{1}}\omega_{{3}}}-2\,r_{{1}}\psi_{{23}}+r_{{1}}\rho_{{1,2}, \\
	d p _1 &=  \left( {p_{{1}}}^{2}-{p_{{2}}}^{2}-1 \right) \omega_{{1}} +u_{{1}}
	\omega_{{2}}+u_{{2}}\omega_{{3}}, \\
	d p_2 &= 2p_{{1}}p_{{2}} \, \omega_{{1}}-u_2 \omega_{{2}}+ u_1 \omega_{{3}},
	\end{aligned}
	\end{equation}
	for some functions $q_1, q_2, u_1, u_2$ on $\mathcal{Q}_1.$ Using these structure equations, it is possible to show that the maximal space-like submanifolds of $Q$ of this type exist locally and depend on 4 functions of 1 variable. We will not study these examples further, except to make the following observation, which will be used in \S\ref{ssect:ERPHomog}.
	
	\begin{prop}
		Any homogeneous 3-dimensional maximal space-like submanifold of $Q$ such that $S$ has a quadruple root is totally umbilic (i.e. $S=0$).
	\end{prop}

\begin{proof}
	If $B$ is a homogeneous 3-dimensional maximal submanifold of $Q$ such that $P$ has a quadruple root is totally umbilic, then the functions $r_1^2+r_2^2$ and $p_1^2+p_2^2$ (which are well-defined on $B$) must be constant. Then equations (\ref{eq:TypeSQuaddrdp}) imply that $r_1=r_2=0,$ so that $B$ is totally umbilic.
\end{proof}
\end{example}

\begin{example}\label{eg:M2} (\emph{Triple root antipodal to single root}) We now suppose that $\mathcal{Q} \to M$ is a $\U(2)^{+}$-structure of type $S$ such that for all points on $M$ the polynomial $\mathsf{S}$ (\ref{eq:U2plBaseTens}) has a triple root and an antipodal single root. Similarly to the previous examples, at each point on $M$ we may find a coframe $\left( \tilde{\nu}, \tilde{\eta} \right)$ so that
	\begin{equation}
	\mathsf{S} = r \, \tilde{\eta}_1^3 \tilde{\eta}_2,
	\end{equation}
	for some positive real $r$ on $M.$ The collection of all such coframes forms a principal $\mathrm{S}^1$-subbundle $\mathcal{Q}_1$ of $\mathcal{Q},$ and we shall work on this subbundle.
	
	On $\mathcal{Q}_1,$
	\begin{equation}
	\left( \begin{array}{ccc}
	\sigma_{11} & \sigma_{12} & \sigma_{13} \\
	\sigma_{21} & \sigma_{22} & \sigma_{23}
	\end{array} \right) = \left( \begin{array}{ccc}
	2\,r\omega_{{3}}&0&2\,r\omega_{{1}} \\
	-2\,r\omega_{{2}}&-2\,r\omega_{{1}}&0
	\end{array}
	\right),
	\end{equation}
	and the equations (\ref{eq:TypeScII}) imply
	\begin{equation}
	\psi_{12} = \psi_{{31}} = 0, \:\:\:\: \rho_{12} = \psi_{23}, \:\:\: r= \tfrac{1}{2}.
	\end{equation}
	
	We are left with structure equations
	\begin{small}
	\begin{equation}
	\begin{aligned}
	d \omega_1 & = 0, \\
	d \omega_2 & = - \psi_{23} \wedge \omega_3, \\
	d \omega_3 & = \psi_{23} \wedge \omega_2, \\
	d \psi_{23} & = - \omega_2 \wedge \omega_3, \\
	d  \left( \begin{array}{c}
	\omega_4 \\
	\omega_5 \\
	\omega_6 \\
	\omega_7 
	\end{array} \right) &= - \frac{1}{2} \left( \begin{array}{cccc}
     -\omega_{{1}} & 2\,\omega_{{1}} & -\omega_{{3}}-\omega_{{2}} &-\omega_{{2}}+\omega_{{3}} \\ 
     2\,\omega_{{1}} & -\omega_{{1}} & -\omega_{{3}}-\omega_{{2}} & -\omega_{{2}}+
\omega_{{3}} \\
-\omega_{{3}}-\omega_{{2}} & -\omega_{{3}} -\omega_{{2}} & \omega_{{1}} & -2\,\psi_{{23}} \\
-\omega_{{2}}+\omega_{{3}} & -\omega_{{2}}+\omega_{{3}} & 2\,\psi_{{23}} & \omega_{{1}}
	\end{array} \right) \wedge \left( \begin{array}{c}
	\omega_4 \\
	\omega_5 \\
	\omega_6 \\
	\omega_7 
	\end{array} \right).
	\end{aligned}
	\end{equation}
\end{small}
	These are the structure equations of an 8-dimensional Lie algebra $\mathfrak{g}$ isomorphic to the semi-direct product $\left(\R \times \mathfrak{so}(2,1) \right) \ltimes \left( \R \oplus \R^3 \right),$ where the $\R$-factor acts on $\R \oplus \R^3$ by  $a \cdot v = (-3 a v_1, a v_2, av_3, av_4 ),$ and the $\mathfrak{so}(2,1)$ factor acts trivially on $\R$ and via the standard representation on $\R^3.$  Thus, this ansatz gives rise to a unique (up to scaling) solution which is homogeneous with 8-dimensional symmetry group. The 7-manifold $M$ is given by the quotient $\G / \mathrm{S^1}$ by the $\psi_{23}$-action.
	
	An integration of the Maurer-Cartan form in this case gives that the corresponding maximal submanifold $B \to Q$ is given by
	\begin{equation}
	\begin{aligned}
	\R^2 \times \mathrm{S}^1 &\to Q, \\
	\left(x, y, \theta \right) & \mapsto
	 \left( \cosh \left( x \right) \cosh \left( y
	\right), \cosh \left( x \right) \sinh \left( y \right), -\sinh
	\left( x \right) \sin \left( \theta \right) \cosh \left( y \right), \right. \\
	  & \left. \sinh \left( x \right) \cos \left( \theta \right) \cosh \left( y
	\right), \sinh \left( x \right) \cos \left( \theta \right) \sinh
	\left( y \right), \sinh \left( x \right) \sin \left( \theta \right) 
	\sinh \left( y \right) \right).
	\end{aligned}
	\end{equation}
	The Riemannian metric on $B$ is isometric to $\R \times \mathbb{H}^2.$
\end{example}

\begin{example}\label{eg:M3} (\emph{Antipodal double roots}) We now suppose that $\mathcal{Q} \to M$ is a $\U(2)^{+}$-structure of type $S$ such that for all points on $M$ the polynomial $\mathsf{S}$ (\ref{eq:U2plBaseTens}) has antipodal double roots. Similarly to the previous examples, at each point on $M$ we may find a coframe $\left( \tilde{\nu}, \tilde{\eta} \right)$ so that
	\begin{equation}
	\mathsf{S} = 3 r^3 \, \tilde{\eta}_1^2 \tilde{\eta}^2_2,
	\end{equation}
	for some positive real $r$ on $M$ (we use a cube here to simplify some later equations). The collection of all such coframes forms a principal $\mathrm{S}^1$-subbundle $\mathcal{Q}_1$ of $\mathcal{Q},$ and we shall work on this subbundle.
	
	On $\mathcal{Q}_1,$
	\begin{equation}
	\left( \begin{array}{ccc}
	\sigma_{11} & \sigma_{12} & \sigma_{13} \\
	\sigma_{21} & \sigma_{22} & \sigma_{23}
	\end{array} \right) = \left( \begin{array}{ccc}
	-2\,{r}^{3}\omega_{{1}}&{r}^{3}\omega_{{2}} & {r}^{3}\omega_{{3}} \\
    0 & 0 & 0
	\end{array}
	\right),
	\end{equation}
	and the equations (\ref{eq:TypeScII}) imply
	\begin{equation}\label{eq:DoubRootdr}
	\psi_{12} = s \, \omega_2, \:\:\:\: \psi_{{31}} = -s \, \omega_{{3}}, \:\:\:\: \rho_{12} = 0, \:\:\: d r= r s \, \omega_1,
	\end{equation}
	for some function $s.$ In particular, we see that the associated maximal submanifold $B \to Q$ actually lies in a totally geodesic hypersurface $\SO(3,2)/\SO(3,1) \subset Q.$ Using equations  (\ref{eq:TypeScII}) again, we have
	\begin{equation}\label{eq:DoubRootds}
	d s = \left( 2 r^6 + s^2 - 1 \right) \omega_1. 
	\end{equation}
	The structure equations of the associated maximal submanifold $B \to Q$ now read
	\begin{equation}
	\begin{aligned}
	d \omega_1 &= 0, \\
	d \omega_2 &= -\psi_{23} \wedge \omega_3 - s \omega_1 \wedge \omega_2, \\
	d \omega_3 &= \psi_{23} \wedge \omega_2 - s \omega_1 \wedge \omega_3, \\
	d \psi_{23} &= - \left( r^6 - s^2 + 1 \right) \omega_2 \wedge \omega_3,
	\end{aligned}
	\end{equation}
	and the exterior derivatives of these equations are identities.
	
	Equations (\ref{eq:DoubRootdr}) and (\ref{eq:DoubRootds}) imply that
	\begin{equation}
	d \left( \frac{r^6 - s^2 +1}{r^2} \right) = 0,
	\end{equation}
	so there exists a constant $c$ on $B$ such that $s^2 = r^6 - c r^2 +1.$ Thus, from equations (\ref{eq:DoubRootdr}),
	\begin{equation}
	\omega_1 = \frac{dr}{r\sqrt{r^6 - c r^2 +1}}.
	\end{equation}
	Moreover, setting $\eta_i = r \omega_i$ for $i = 2,3,$ we find
	\begin{equation}
	d \eta_2 = -\psi_{23} \wedge \eta_3, \:\:\: d \eta_3  = \psi_{23} \wedge \eta_3, \:\:\: d \psi_{23} = - c \, \eta_2 \wedge \eta_3,
	\end{equation}
	which are the structure equations for a metric of constant curvature $k.$ Conversely, let $\Sigma$ be one of $\mathbb{S}^2,$ $\R^2,$ or $\mathbb{H}^2,$ with a metric of constant curvature $-c.$ Reversing the steps above and constructing the Maurer-Cartan form $\gamma$ (\ref{eq:TypeSgam}) (which will have all zeroes in its final row and column), gives a maximal spacelike immersion of $\Sigma \times I$ into $\mathrm{SO}(3,2)/\mathrm{SO}(3,1),$ where $I$ is any open interval containing no zeroes of $r^6 - c r^2 +1.$ It can be checked that the resulting maximal submanifold $B$ has a cohomogeneity-one action of one of the groups $\SO(2,1), \mathrm{ASO}(2),$ or $\SO(3),$ depending on the sign of $c.$
	
	There are two special cases of particular interest. If $c = 2^{-2/3} \cdot 3,$ then ${r^6 - c r^2 +1}$ factors as $\left( r^2 - 2^{-1/3} \right)^2 \left(r^2 + 2^{2/3} \right),$ and we see that the immersion of $\mathbb{H}^2 \times \left(0, 2^{-1/6} \right)$ constructed in the above paragraph is complete, because
	\begin{equation}
	\int_{0}^{2^{-1/{5}}} \omega_1 = - \infty, \:\:\:\:\: \int_{2^{-1/{5}}}^{2^{-1/6}} \omega_1 = \infty.
	\end{equation}
	Thus, this example gives rise to a complete  and inhomogeneous ERP closed $\G_2$-structure on $M$.
	
	Next, from equations (\ref{eq:DoubRootdr}) the only possibility for $(r,s)$ to be constant on $B$ is if $r=2^{-1/6}$ and $s=0.$ We then have structure equations
	\begin{small}
	\begin{equation}
	\begin{aligned}
	d \omega_1 &= 0, \\
	d \omega_2 &= -\psi_{23} \wedge \omega_3, \\
	d \omega_3 &= \psi_{23} \wedge \omega_2, \\
	d \psi_{23} &= - \tfrac{3}{2} \omega_2 \wedge \omega_3, \\
	d  \left( \begin{array}{c}
	\omega_4 \\
	\omega_5 \\
	\omega_6 \\
	\omega_7 
	\end{array} \right) &= - \frac{1}{2} \left( \begin{array}{cccc}
	-\tfrac{1}{\sqrt {2}} \omega_{{2}} - \omega_{{1}} & 
	\tfrac{1}{\sqrt {2}} \omega_{{3}} -\psi_{{23}} & -\omega_{{3}} & -\sqrt {2} \omega	_{{1}} - \omega_{{2}} \\
	\tfrac{1}{\sqrt {2}} \omega_{{3}} + \psi_{{23}} & \tfrac{1}{\sqrt {2}} \omega_{{2}} - \omega_{{1}} & \sqrt {2} \omega_{{1}} - \omega_{{2}} & \omega_{{3}} \\
	-\omega_{{3}} &	\sqrt {2} \omega_{{1}} - \omega_{{2}} & \tfrac{1}{\sqrt {2}} \omega_{{2}} + \omega_{{1}} & \tfrac{1}{\sqrt {2}} \omega_{{3}} - \psi_{{23}} \\ 
	-\sqrt {2} \omega_{{1}} - \omega_{{2}} & \omega_{{3}} & \tfrac{1}{\sqrt {2}} \omega_{{3}} + \psi_{{23}} & \tfrac{1}{\sqrt {2}} \omega_{{2}} + \omega_{{1}}
	\end{array} \right) \wedge \left( \begin{array}{c}
	\omega_4 \\
	\omega_5 \\
	\omega_6 \\
	\omega_7 
	\end{array} \right).
	\end{aligned}
	\end{equation}
\end{small}
These are the structure equations of an 8-dimensional Lie algebra $\mathfrak{g}$ isomorphic to the semidirect product  $\left(\R \times \mathfrak{sl}(2,\R) \right) \ltimes \left( \R^2 \oplus \R^2 \right),$ where the $\R$-factor acts on $\R^2 \oplus \R^2$ by  $a \cdot v = (a v_1, a v_2, -a v_3, -a v_4 ),$ and the $\mathfrak{sl}(2,\R)$ factor acts on $\R^2 \oplus \R^2$ as two copies of the standard representation. Thus, this ansatz gives rise to a unique (up to scaling) homogeneous solution, with 8-dimensional symmetry group. The 7-manifold $M$ is given by the quotient $\G / \mathrm{S^1}$ by the $\psi_{23}$-action. An integration of the Maurer-Cartan form in this case gives that the corresponding maximal submanifold $B \to Q$ is given by
	\begin{equation}
	\begin{aligned}
	\R^2 \times \mathrm{S}^1 &\to Q, \\
	\left(x, y, \theta \right) & \mapsto
	\tfrac{1}{6} \left( 2\,\cosh x +4\,\cosh y , 2\sqrt {3} \, \sinh x , 2\sqrt {6} \, \sinh y \sin \theta , \right. \\
	  & \left. 2\sqrt {6} \, \sinh y \cos \theta, -2 \sqrt {2} \left( \cosh x - \cosh y \right), 0 \right).
	\end{aligned}
	\end{equation}
	The Riemannian metric on $B$ is isometric to $\R \times \mathbb{H}^2.$
\end{example}

\subsubsection{Laplacian flow}

Recall from \S\ref{sssect:LapFlow} that the Laplacian flow of an ERP $\U(2)^+$-structure simply rescales the leaves of the foliation determined by $\nu = 0.$

\begin{thm}\label{thm:TypeSLapSol}
	Any $\U(2)^+$-structure of type $S$ is locally a steady Laplace soliton. The universal cover of a $\U(2)^+$-structure of type $S$ is a steady Laplace soliton.
\end{thm}

\begin{proof}
	By Theorem \ref{thm:TypeSCons}, any structure $\mathcal{Q} \to M$ of type $S$ is locally equivalent to an $\R^4$-bundle over a 3-manifold $B,$ where the $\R^4$-fibres are flat and leaves of the foliation $\nu = 0.$ Restricting to an open set in $M$ consisting of the $\R^4$-bundle over contractible open set in $B,$ let $X$ be the vector field whose flow has the property that it preserves the $\R^4$-fibres and scales each fiber at speed $\exp(6kt)$. Then from the Laplacian evolution equation (\ref{eq:U2plLapFlow}), $d \tau = \mathcal{L}_X \varphi,$ and $\varphi$ is locally a steady Laplacian soliton.
	
	To see that the universal cover of $M$ is a \emph{bona fide} Laplace soliton, note that the pullback of the $\R^4$-bundle over $B$ to the universal cover of $B$ will be trivial (because $\mu$ defines a flat $\mathrm{SL}(4, \R)$-connection), so the vector field $X$ exists globally.
\end{proof}

\subsection{Classification of homogeneous ERP closed $\G_2$-structures}\label{ssect:ERPHomog}

In this section, we shall use the work of the previous sections to classify the homogeneous ERP closed $\G_2$-structures.

We begin by recalling the classification of Lie groups admitting left-invariant ERP closed $\G_2$-structures due to Lauret--Nicolini \cite{LauNicERPLI20}. They find, up to equivalence, five examples on different completely solvable Lie groups, and we shall follow their notation in labeling them $\G_B,$ $\G_{M1},$ $\G_{M2},$ $\G_{M3},$ and $\G_J.$ Lauret--Nicolini note that $\G_B$ is equivalent to Bryant's example (\S\ref{sssect:BryEx}).

\begin{prop}\label{prop:HomogEgs}
	A left-invariant ERP closed $\G_2$-structure is of type $S.$ Furthermore: \begin{itemize}
		\item $(\G_J, \varphi_J)$ is equivalent to the $\G_2$-structure constructed in Example \ref{eg:GJ}. The $\G_2$-automorphism group of this example has dimension 7.
		\item $(\G_{M2}, \varphi_{M2})$ is equivalent to the homogeneous $\G_2$-structure of Example \ref{eg:M2} constructed on $\left( \left(\R \times \mathrm{SO}(2,1) \right) \ltimes \R^4 \right) / \mathrm{S}^1.$ The $\G_2$-automorphism group of this example has dimension 8.
		\item $(\G_{M3}, \varphi_{M3})$ is equivalent to the homogeneous $\G_2$-structure of Example \ref{eg:M3} constructed on $\left( \left(\R \times \mathrm{SL}(2,\R) \right) \ltimes \R^4 \right) / \mathrm{S}^1.$ The $\G_2$-automorphism group of this example has dimension 8.
		\item The $\G_2$-automorphism group of $\left(\G_{M1}, \varphi_{M1} \right)$ has dimension 7.
	\end{itemize}
\end{prop}

\begin{proof}
	A simple calculation using the structure equations provided by Lauret--Nicolini shows that all of the examples in their classification are of type $S.$ Next, we note that any $\G_2$-automorphism fixing a point $p$ must preserve $\mathsf{S}_p.$ For cases $\G_J$ and $\G_{M1},$ the subgroup fixing $\mathsf{S}_p$ is discrete, so these examples have 7-dimensional symmetry groups. For cases $\G_{M2}$ and $\G_{M3},$ the tensor $\mathsf{S}$ has a triple root antipodal to a single root and antipodal double roots respectively, in both cases $\mathsf{S}_p$ has a 1-dimensional stabiliser. The homogeneous examples of these type are classified in Examples \ref{eg:M2} and \ref{eg:M3}, and the statement about equivalence follows.
\end{proof}

\begin{remark}
	Lauret--Nicolini \cite{LauNicERPLI20} have computed the $\G_2$-automorphism group of $(\G_J, \varphi_J),$ but this result is new in the other cases. In particular, the existence of an additional symmetry in the $\G_{M2}$ and $\G_{M3}$ cases is noteworthy.
\end{remark}

The homogeneous classification does not reveal any new examples.

\pagebreak[3]

\begin{thm}\label{thm:ERPHomogClass}
	An ERP closed $\G_2$-structure $\varphi$ admitting a transitive action of diffeomorphisms preserving $\varphi$ is equivalent to one of the following examples:
	\begin{itemize}
		\item Bryant's example with 11-dimensional symmetry group $\left( \mathrm{S}^1 \cdot \mathrm{SL}(2, \mathbb{C}) \right) \ltimes \C^2$ (see \S\ref{sssect:BryEx}).
		\item Lauret's example on $\G_{J}$ (see Example \ref{eg:GJ}).
		\item Lauret--Nicolini's example on $\G_{M1}.$
		\item The homogeneous $\G_2$-structure of Example \ref{eg:M2} constructed on $\left( \left(\R \times \mathrm{SO}(2,1) \right) \ltimes \R^4 \right) / \mathrm{S}^1.$
		\item The homogeneous $\G_2$-structure of Example \ref{eg:M3} constructed on $\left( \left(\R \times \mathrm{SL}(2,\R) \right) \ltimes \R^4 \right) / \mathrm{S}^1.$
	\end{itemize}
\end{thm}

\begin{proof}
	We begin by noting that the equation
	\begin{equation}
	d \left( |\tau|^2  \right)= \tfrac{2}{9} *_{\varphi} \tau^3
	\end{equation}
	implies that a homogeneous ERP closed $\G_2$-structure has special torsion of positive type.
	
	Next, we remark that, due to Lauret--Nicolini's results, we need only consider the case where the $\G_2$-automorphism group has dimension greater than 7. Let $\mathrm{H}$ denote the isotropy group of the $\G_2$-automorphism group. Since any $\G_2$-automorphism must fix the $\U(2)^+$-structure $\mathcal{Q},$ $\mathrm{H}$ is a subgroup of $\U(2)^+.$ Furthermore, fixing a point $p \in M,$ we have
	\begin{equation}
	\mathrm{H} \leq \mathrm{Stab}_{\U(2)^+} \left(\mathsf{A}_p, \mathsf{S}_p, \mathsf{Z}_p \right).
	\end{equation}
	Now, $\left(\mathsf{A}_p, \mathsf{S}_p, \mathsf{Z}_p \right)$ is an element of the $\U(2)$-module
	\begin{equation}
	W = \left( \mathrm{Sym}^5_{\C} \C^2 \otimes  \left( \Lambda^2_{\C} \C^2 \right)^{1/2} \right) \oplus \left( \mathrm{Sym}^4_{\C} \C^2 \otimes  \Lambda^2_{\C} \C^2 \right) \oplus \left( \C^2 \otimes  \left( \Lambda^2_{\C} \C^2 \right)^{3/2} \right).
	\end{equation}
	The only element of $W$ with stabiliser $\U(2)^+$ is $(0,0,0),$ and this corresponds to Bryant's example \S\ref{sssect:BryEx}. There are no elements with stabiliser $\mathrm{SU}(2)$ or $\mathrm{T}^2,$ so there are no ERP $\G_2$-structures with symmetry groups of dimension $10$ or $9.$
	
	Next, it is easily checked that any element of $W$ stabilised by an $\mathrm{S}^1$-subgroup of $\U(2)^+$ takes values in exactly one of the three irreducible constituents of  $W.$ Thus, a homogeneous ERP structure is of type $A,$ $N,$ or $S.$
	
	The structure equations (\ref{eq:TypeAStructPre}) imply that a homogeneous structure of type $A$ has $A=0,$ because otherwise the norm of $\mathsf{A}$ will change along the $\nu_{1\bar{1}}$-direction. A homogeneous structure of type $N$ has $N=0$ because otherwise the $\mathrm{SL}(2,\C)$ action defined in \S\ref{ssect:TypeN} does not fix the norm of $\mathsf{N}.$
	
	Next, the elements of $\mathrm{Sym}^4_{\C} \C^2 \otimes  \Lambda^2_{\C} \C^2$ fixed by an $\mathrm{S}^1$-subgroup are exactly those with a quadruple root, a triple root antipodal to a single root, or antipodal double roots. The homogeneous examples satisfying these conditions have been classified in \S\ref{sssect:TypeSEgs}.
\end{proof}

\subsubsection{A new compact example}\label{sssect:Compact}

The Lie group $\G_{M3}$ in the Lauret--Nicolini classification is not unimodular, so it does not admit a lattice. However, by Proposition \ref{prop:HomogEgs}, the group $\left( \R \times \mathrm{SL}(2, \R) \right) \ltimes \R^4$ acts on this example by $\G_2$-automorphisms and this group does admit cocompact lattices, as we now explain.
	
By the standard theory of arithmetic subgroups of $\mathrm{SL}(2,\R)$ \cite[\S 6.2]{MorrisArith}, the subgroup $\mathrm{SL}(1, \mathbb{H}_{\mathbb{Z}}^{3,2})$ of $\mathrm{SL}(2, \R)$ preserving the lattice in $\R^4 \cong \R^2 \oplus \R^2$ spanned by the vectors
\begin{align}
v_1 = \left(\begin{array}{c}
1 \\ 0 \\ 0 \\1 
\end{array}\right), \:\: v_2 = \left(\begin{array}{c}
\sqrt{3} \\ 0 \\ 0 \\ -\sqrt{3},  
\end{array}\right), \:\: v_3 = \left(\begin{array}{c}
0 \\ 2 \\ 1 \\ 0  
\end{array}\right), \:\: v_4 = \left(\begin{array}{c}
0 \\ -2\sqrt{3} \\ \sqrt{3} \\ 0
\end{array}\right),
\end{align}
is a cocompact lattice in $\mathrm{SL}(2,\R).$ Furthermore, the matrix
\begin{equation}
\begin{aligned}
A &= \operatorname{diag} \left( 2 + \sqrt{3}, 2 + \sqrt{3}, 2 - \sqrt{3}, 2 - \sqrt{3} \right) \\
&= \exp \operatorname{diag} \left( \log (2 + \sqrt{3}), \log (2 + \sqrt{3}), -\log (2 + \sqrt{3}), - \log (2 + \sqrt{3}) \right)
\end{aligned}
\end{equation}
and its inverse $A^{-1}$ preserve the lattice generated by $v_1, \ldots, v_4.$ Let $\mathrm{H}_\mathbb{Z}$ denote the discrete cocompact subgroup of $\R \times \mathrm{SL}(2, \R)$ generated by $A$ and $\mathrm{SL}(1, \mathbb{H}_{\mathbb{Z}}^{3,2}),$ then the subgroup $\Gamma = \mathrm{H}_{\mathbb{Z}} \ltimes \operatorname{span}_{\mathbb{Z}}(v_1, v_2, v_3, v_4)$ of $\left( \R \times \mathrm{SL}(2, \R) \right) \ltimes \R^4$ is a cocompact lattice. The left action of $\Gamma$ on $M = \left( \left(\R \times \mathrm{SL}(2,\R) \right) \ltimes \R^4 \right) / \mathrm{S}^1$ is free and properly discontinuous, and preserves the homogeneous ERP structure $\varphi$ on $M.$ Hence the compact quotient $\bar{M} = \Gamma \backslash M$ admits an ERP $\G_2$-structure that pulls back to $M$ to equal $\varphi.$

\section{Quadratic closed $\G_2$-structures with special torsion of negative type}\label{sect:U2mi}

This section is devoted to the study of $\lambda$-quadratic closed $\G_2$-structures with special torsion of negative type.

\subsection{The induced $\U(2)^-$-structure}\label{ssect:U2miInduced}

Let $(M,\varphi)$ be a $\lambda$-quadratic closed $\G_2$-structure such that $\tau$ has stabiliser everywhere conjugate to $\U(2)^-.$ Since $\tau=0$ identically if and only if $\varphi$ is torsion-free, we will assume that $\tau$ is non-vanishing somewhere and restrict to the open dense subset where $\tau \neq 0.$ We may then adapt frames and define a $\U(2)^-$-structure $\mathcal{Q} \subset \mathcal{B}$ by
\begin{equation}
\mathcal{Q} = \left\lbrace u : T_x M \to V \mid u \in \mathcal{B}_x, \right.
 \left.  u^* \left( -2 e_2 \wedge e_3 + e_4 \wedge e_5 + e_6 \wedge e_7 \right) \: \text{is a multiple of} \: {\tau} \right\rbrace.
\end{equation}

Thus, on $\mathcal{Q},$ the torsion 2-form $\tau$ is given by
\begin{equation}
\tau = 4 f \, \left(-2 \, \omega_2 \wedge \omega_3 + \omega_4 \wedge \omega_5 + \omega_6 \wedge \omega_7 \right),
\end{equation}
for some non-vanishing function $f.$

The action of $\U(2)^{-}$ on the standard representation $V \cong \R^7$ of $\G_2$ is reducible: we have $V \cong \langle e_1 \rangle \oplus \langle e_2, e_3 \rangle \oplus \langle e_4, e_5, e_6, e_7 \rangle \cong \R \oplus \C \oplus \C^2.$ We split the tautological $V$-valued 1-form $\omega$ accordingly as $\omega = \alpha + \nu + \eta.$ Explicitly, define complex-valued 1-forms $\eta_a$ and $\nu$ and a real-valued 1-form $\alpha$ on $\mathcal{Q}$ by
\begin{equation}
\begin{aligned}
\alpha & = \omega_1, & \eta_1 &= \omega_4 + i \omega_5,  \\
\nu & = \omega_2 + i \omega_3, & \eta_2 &= \omega_6 + i \omega_7,
\end{aligned}
\end{equation}
Just as in \S\ref{sect:U2pl}, $1 \leq a, b, c, \ldots \leq 2$ will be indices obeying the complex Einstein summation convention, meaning that any pair of barred and unbarred indices is implicitly summed over, and we make use of the $\eps$-symbol with two such indices, defined to be the unique skew-symmetric symbol satisfying $\eps_{12} = 1/2.$

The natural $\mathfrak{g}_2$-valued connection form $\theta$ also splits. There is a $\U(2)^-$-invariant decomposition
\begin{equation}
\mathfrak{g}_2 = \mathfrak{u}(2)^- \oplus \C \oplus \left(\C^2 \otimes \Lambda^2_\C \C^2 \right) \oplus \C^2,
\end{equation}
and $\theta$ decomposes accordingly as $\theta = \kappa + \xi + \gamma + \sigma.$ Explicitly, define complex-valued 1-forms $\kappa_{a \overline{b}} = - \overline{\kappa_{b \bar{a}}},$ $\xi,$ $\gamma_a,$ and $\sigma_{a}$ by
\begin{equation*}
\begin{aligned}
\kappa_{1\bar{1}} &= -i \theta_{45} , & \gamma_{1} & = -\tfrac{1}{2} \left(\theta_{24}+\theta_{35} \right) - \tfrac{i}{2} \left(\theta_{25} - \theta_{34} \right), \\
\kappa_{1\bar{2}} &= \tfrac{1}{2} \left(\theta_{46} + \theta_{57} \right) - \tfrac{i}{2} \left(\theta_{47} - \theta_{56} \right) , & \gamma_{2} & = -\tfrac{1}{2} \left(\theta_{26} + \theta_{37} \right) - \tfrac{i}{2} \left(\theta_{27} - \theta_{36} \right), \\
\kappa_{2\bar{2}} &= -i \theta_{67} , & \sigma_{1} & = \tfrac{1}{2} \left(\theta_{35} -\theta_{24} \right) - \tfrac{i}{2} \left(\theta_{25} + \theta_{34} \right), \\
\xi & =  \tfrac{1}{2} \left(\theta_{57} - \theta_{46} \right) - \tfrac{i}{2} \left(\theta_{47} + \theta_{56} \right) , & \sigma_{2} & = \tfrac{1}{2} \left(\theta_{37} - \theta_{26} \right) - \tfrac{i}{2} \left(\theta_{27} + \theta_{36} \right).
\end{aligned}
\end{equation*}
The $\mathfrak{u}(2)$-valued 1-form $\kappa$ is the connection form associated to the natural connection on the $\U(2)^-$-structure $\mathcal{Q},$ while the 1-forms $\xi,$ $\gamma,$ and $\sigma$ are are semibasic for the projection $\mathcal{Q} \to M.$

\subsubsection{Structure equations}

The first structure equation on $\mathcal{B}$ (\ref{eq:CartanIG2}) restricted to $\mathcal{Q}$ reads
\begin{equation}\label{eq:U2miStructPre}
\begin{aligned}
d \alpha &= -i \, \xi \wedge \nu + i \, \overline{\xi \wedge \nu} - 2 i \, \overline{\eps_{ab}} \sigma_a \wedge \eta_b + 2 i \, \eps_{ab} \overline{\sigma_a \wedge \eta_b}, \\
d \nu &= \kappa_{a\bar{a}} \wedge \nu - 2 i \, \overline{\xi} \wedge \alpha + \overline{\gamma_a} \wedge \eta_a + \sigma_a \wedge \overline{\eta_a} - 4f \, \alpha \wedge \nu + 2if \, \eps_{ab} \overline{\eta_a \wedge \eta_b}, \\
d \eta_a &= -\kappa_{a \bar{b}} \wedge \eta_b + 4 i \, \eps_{ab} \overline{\sigma_b} \wedge \alpha - \gamma_a \wedge \nu - \sigma_a \wedge \overline{\nu} + 2 \eps_{ab} \xi \wedge \overline{\eta_b} - 2 i f \, \eps_{ab} \overline{\nu \wedge \eta_b}.
\end{aligned}
\end{equation}

On $\mathcal{Q},$ the $\G_2$ 3-form $\varphi,$ 4-form ${*}_{\varphi} \varphi,$ and torsion 2-form $\tau$ are given by
\begin{equation}
\begin{aligned}
\varphi &= \tfrac{i}{2} \alpha \wedge \left(\nu \wedge \overline{\nu} + \eta_a \wedge \overline{\eta_a} \right) + \tfrac{1}{2} \overline{\eps_{ab}} \nu \wedge \eta_a \wedge \eta_b + \tfrac{1}{2} \eps_{ab} \overline{\nu \wedge \eta_a \wedge \eta_b}, \\
{*}_{\varphi} \varphi &= \tfrac{i}{2} \alpha \wedge \left( \overline{\eps_{ab}} \nu \wedge \eta_a \wedge \eta_b - \eps_{ab} \overline{\nu \wedge \eta_a \wedge \eta_b} \right) - \tfrac{1}{4} \eta_a \wedge \overline{\eta_a} \wedge \left(\nu \wedge \overline{\nu} + \tfrac{1}{2} \eta_b \wedge \overline{\eta_b} \right), \\
\tau &= 2if \, \left(-2 \, \nu \wedge \overline{\nu} + \eta_a \wedge \overline{\eta_a} \right).
\end{aligned}
\end{equation}

The $\lambda$-quadratic condition (\ref{eq:LamQuad}) reads
\begin{equation}\label{eq:U2miQuad}
\begin{aligned}
\tfrac{7}{16} d \left( 2if \, \left(-2 \, \nu \wedge \overline{\nu} + \eta_a \wedge \overline{\eta_a} \right) \right) =& i f^2 \, \alpha \wedge \left(\left(10 \lambda +3 \right) \nu \wedge \overline{\nu} - \left(11 \lambda - 3 \right) \eta_a \wedge \overline{\eta_a} \right), \\
& + 3 f^2 \, \left( \lambda +1 \right) \left( \overline{\eps_{ab}} \nu \wedge \eta_a \wedge \eta_b - \eps_{ab} \overline{\nu \wedge \eta_a \wedge \eta_b} \right),
\end{aligned}
\end{equation}
while equation (\ref{eq:Bry469}) gives
\begin{equation}\label{eq:U2miBryA}
d f = - 28 \frac{\lambda \left(2 \lambda -1 \right)}{3 \lambda - 4} f^2 \, \alpha,
\end{equation}
and equation (\ref{eq:Bry466}) gives
\begin{equation}\label{eq:U2miBryB}
d \left(f^3 \nu \wedge \overline{\nu} \wedge \eta_a \wedge \overline{\eta_a} \wedge \eta_b \wedge \overline{\eta_b} \right) = - \tfrac{36}{7} \left( 6 \lambda -1 \right) \alpha \wedge \nu \wedge \overline{\nu} \wedge \eta_a \wedge \overline{\eta_a} \wedge \eta_b \wedge \overline{\eta_b}.
\end{equation}

Equations (\ref{eq:U2miQuad} - \ref{eq:U2miBryB}), together with their exterior derivatives, imply that there exist complex-valued functions $B_a,$ $E_a,$ and $F_{ab} = F_{ba}$ on $\mathcal{Q}$ such that
\begin{equation}\label{eq:U2mintors}
\begin{aligned}
\xi &= \tfrac{4i}{7} \tfrac{17 \lambda^2 + 6 \lambda -11}{3 \lambda -4} f \overline{\nu} + i B_a \overline{\eta_a},\\
\gamma_a &= -\tfrac{2}{3} B_a \alpha + E_a \nu + F_{ab} \eta_b, \\
\sigma_a &= 2 i \eps_{ab} \overline{B_b \nu} + \tfrac{2i}{7} \tfrac{32 \lambda^2 + 57 \lambda - 24}{3 \lambda -4} f \, \eps_{ab} \overline{\eta_b}.
\end{aligned}
\end{equation}
Substituting equation (\ref{eq:U2mintors}) into (\ref{eq:U2miStructPre}) and differentiating, the identity $d^2 \eta = 0$ implies
\begin{equation}
B_a \overline{B_b} - \tfrac{1}{2} \delta_{a\bar{b}} B_c \overline{B_c}, \:\:\:\:\: a, b = 1, 2
\end{equation}
and it follows that $B = 0$ on $\mathcal{Q}.$

The structure equations (\ref{eq:U2miStructPre}) may now be rewritten as
\begin{equation}\label{eq:U2miStruct}
\begin{aligned}
d \alpha &= 0, \\
d \nu &= \kappa_{a\bar{a}} \wedge \nu + \overline{E_a \nu} \wedge \eta_a + \tfrac{4}{7} \tfrac{34 \lambda^2 - 9 \lambda + 6}{3 \lambda - 4} f \, \alpha \wedge \nu - \tfrac{8i}{7} \tfrac{\left(\lambda+1\right) \left(8\lambda+1\right)}{3\lambda-4} f \, \eps_{ab} \overline{\eta_a \wedge \eta_b}, \\
d \eta_a &= - \kappa_{a \bar{b}} \wedge \eta_b + F_{ab} \nu \wedge \overline{\eta_b} + \tfrac{2}{7} \tfrac{32 \lambda^2 + 57 \lambda - 24}{3 \lambda - 4} f \, \alpha \wedge \eta_a + \tfrac{40 i}{7} \tfrac{\left(5 \lambda -2 \right) \left( \lambda+1\right)}{3 \lambda -4} f \, \eps_{ab} \overline{\nu \wedge \eta_b}.
\end{aligned}
\end{equation}

\begin{thm}\label{thm:U2miLamVals}
	Let $\left(M, \varphi \right)$ be a $\lambda$-quadratic closed $\G_2$-structure with special torsion of negative type. Then
	\begin{align}
	\lambda \in \left\lbrace -1, -\tfrac{1}{8}, \tfrac{2}{5}, \tfrac{3}{4} \right\rbrace.
	\end{align}
\end{thm}
\begin{proof}
	Differentiating the structure equations (\ref{eq:U2miStruct}), the identity $d^2 \nu = 0$ implies
	\begin{align}
	\left( 8 \lambda +1 \right)& \left( 5 \lambda -2 \right) \left(4 \lambda -3 \right) \left( \lambda +1 \right) f^2 = 0,
	\end{align}
	and we have assumed that $f$ does not vanish identically.
\end{proof}

Now, since $d \alpha = 0,$ there exists a locally defined function $r$ on $M,$ unique up to the addition of a constant, such that $\alpha_1 = dr.$ Equation (\ref{eq:U2miBryA}) then implies that there exists a constant $k$ with
\begin{equation}
f = \frac{3 \lambda -4}{28 \lambda \left(2 \lambda -1 \right)}\frac{1}{\left( r-k \right)}.
\end{equation}
After altering $r$ by a constant we assume $k=0.$

\begin{thm}\label{thm:U2MiIncomp}
	Let $\left(M, \varphi \right)$ be a $\lambda$-quadratic closed $\G_2$-structure with special torsion of negative type. Then the induced metric $g_{\varphi}$ is incomplete.
\end{thm}

\begin{proof}
	The scalar curvature of $g_{\varphi}$ is given by (\ref{eq:G2ScalRic}), so
	\begin{equation}
	\mathrm{Scal}(g_{\varphi}) = -\tfrac{1}{2} | \tau|^2 = -\frac{3 \, \left(3 \lambda -4 \right)^2}{49 \, \lambda^2 \left(2 \lambda - 1\right)^2} \frac{1}{r^2},
	\end{equation}
	and this blows up at $r = 0,$ which is at a finite distance.
\end{proof}

Define complex-valued 1-forms $\varpi$ and $\theta_a$ and complex-valued functions $A_a$ and $H_{ab} = H_{ba}$ on $\mathcal{Q}$ by
\begin{equation}
\begin{aligned}
\varpi & = r^{-\frac{34 \lambda^2 - 9 \lambda +6}{49 \lambda \left( 2 \lambda -1 \right)}} \nu, & A_a & = r^{\frac{32 \lambda^2 + 57 \lambda -24}{98 \lambda \left( 2 \lambda -1 \right)}} E_a, \\
\theta_a & = r^{-\frac{32 \lambda^2 + 57 \lambda -24}{98 \lambda \left( 2 \lambda -1 \right)}} \eta_a, & H_{ab} & = r^{\frac{34 \lambda^2 - 9 \lambda +6}{49 \lambda \left( 2 \lambda -1 \right)}} F_{ab}.
\end{aligned}
\end{equation}
In these new variables the structure equations (\ref{eq:U2miStruct}) become
\begin{equation}\label{eq:U2miStructSix}
\begin{aligned}
d \varpi &= \kappa_{a\bar{a}} \wedge \varpi + \overline{A_a \varpi} \wedge \theta_a - \tfrac{2i}{49} \tfrac{\left(\lambda+1\right) \left(8\lambda+1\right)}{\lambda \left(2 \lambda -1 \right)} r^{-\frac{5 \left( 5 \lambda -2 \right) \left( 4 \lambda -3 \right)}{49 \lambda \left( 2 \lambda -1 \right)}} \, \eps_{ab} \overline{\theta_a \wedge \theta_b}, \\
d \theta_a &= - \kappa_{a \bar{b}} \wedge \theta_b + H_{ab} \varpi \wedge \overline{\theta_b} + \tfrac{5 i}{49} \tfrac{\left(5 \lambda -2 \right) \left( \lambda+1\right)}{\lambda \left(2 \lambda -1 \right)} r^{-\frac{2 \left( 8 \lambda +1 \right) \left( 4 \lambda -3 \right)}{49 \lambda \left( 2 \lambda -1 \right)}} \, \eps_{ab} \overline{\varpi \wedge \theta_b}.
\end{aligned}
\end{equation}

When $\lambda \in \left\lbrace -1, -\frac{1}{8}, \frac{2}{5}, \frac{3}{4} \right\rbrace,$ as is necessary from Theorem \ref{thm:U2miLamVals}, the function $r$ does not appear in the derivatives of $\varpi$ or $\theta_a,$ and the problem of understanding the geometry of $M$ is reduced to understanding the geometry of the $6$-manifold defined by $r = 1.$ Call this $6$-manifold $N,$ and let $\mathcal{P}$ denote the pullback of the $\U(2)^-$-structure $\mathcal{Q}$ to $N$. The bundle $\mathcal{P}$ is a $\U(2)$-subbundle of the principal coframe bundle of $N$, and so defines a $\U(2)$-structure on $N.$ The tautological form on $\mathcal{P}$ is $\left( \varpi, \theta_a \right),$ while the forms $\kappa_{a \bar{b}}$ are the components of the natural $\mathfrak{u}(2)$-valued connection form. This discussion is summarised in the follow result.

\begin{prop}\label{prop:U2miRedToSix}
	Let $\left( M, \varphi \right)$ be a $\lambda$-quadratic closed $\G_2$-structure with special torsion of negative type, where $\lambda$ is necessarily one of $-1,$ $-{1}/{8},$ ${2}/{5},$ or ${3}/{4}.$ Then $M$ is locally equivalent to $ \R^{+} \times N,$ where $N$ is a $6$-manifold endowed with a $\U(2)$-structure $\mathcal{P}$ satisfying structure equations (\ref{eq:U2miStructSix}). The pullback of the $\G_2$ 3-form $\varphi$ to to $\mathcal{Q}$ is
	\begin{equation}
	\begin{aligned} \label{eq:U2miphi}
	\varphi =& \tfrac{i}{2} dr \wedge \left( r^{{\frac {68{\lambda}^{2}-18\lambda+12}{49\lambda \left( 2\lambda-1 \right) }}} \varpi \wedge \overline{\varpi} + r^{{\frac {32 {\lambda}^{2}+57 \lambda-24}{49 \lambda \left( 2 \lambda-1 \right) }}} \theta_a \wedge \overline{\theta_a} \right) \\
	& + \tfrac{1}{2} r^{\frac { 6 \left( \lambda+1 \right)  \left( 11\lambda-3 \right) }{49\lambda \left( 2\lambda-1 \right) }} \left(  \overline{\eps_{ab}} \varpi \wedge \theta_a \wedge \theta_b + \eps_{ab} \overline{\varpi \wedge \theta_a \wedge \theta_b} \right),
	\end{aligned}
	\end{equation}
	where $r$ is the coordinate in the $\R^+$ direction.
	
	Conversely, let $\mathcal{P} \to N$ be a $\U(2)$-structure over a 6-manifold $N$ satisfying structure equations (\ref{eq:U2miStructSix}). Let $M = \R^{+} \times N.$ Then, if $\lambda$ is one of $-1,$ $-{1}/{8},$ ${2}/{5},$ or ${3}/{4},$ the 3-form $\varphi$ on $\R^+ \times \mathcal{P}$ defined by equation (\ref{eq:U2miphi}) is invariant under the $\U(2)$-action on this space, and descends to $M$ to give a $\lambda$-quadratic closed $\G_2$-structure with special torsion of negative type.
\end{prop}

We now study the four cases corresponding to the possible values of $\lambda$ in turn. Similarly to \S\ref{sect:U2pl}, we mainly restrict to the cases where at most one of the tensors $A_a,$ $H_{ab}$ are non-zero.

\subsection{The case $\lambda = -1$}\label{ssect:LamMin1}

When $\lambda=-1,$ the structure equations (\ref{eq:U2miStructSix}) read
\begin{equation}\label{eq:StructLamMin1}
\begin{aligned}
d \varpi &= \kappa_{a\bar{a}} \wedge \varpi + \overline{A_a \varpi} \wedge \theta_a, \\
d \theta_a &= - \kappa_{a \bar{b}} \wedge \theta_b + H_{ab} \varpi \wedge \overline{\theta_b}.
\end{aligned}
\end{equation}

Now, since $\U(2)$ is a subgroup of $\SU(3),$ the $\U(2)$-structure $\mathcal{P}$ on $N$ induces an $\SU(3)$-structure on $N.$ The differential forms $\Omega \in \Omega^{1,1} (N)$ and $\Upsilon \in \Omega^{3,0}(N)$  associated to this $\SU(3)$-structure are
\begin{equation}\label{eq:SU3Struct}
\begin{aligned}
\Omega &= \tfrac{i}{2} \left( \varpi \wedge \overline{\varpi} + \theta_a \wedge \overline{\theta_a} \right), \\
\Upsilon &= \eps_{ab} \varpi \wedge \theta_a \wedge \theta_b,
\end{aligned}
\end{equation}
and from equations (\ref{eq:StructLamMin1}), the exterior derivatives of these forms satisfy
\begin{equation}
d \Omega = 0, \:\:\:\:\: d \Upsilon = 0,
\end{equation}
so that the induced $\SU(3)$-structure is Calabi-Yau.

\subsubsection{$A=H=0$}\label{sssect:U2mimin11flat}

When the torsion functions $A$ and $H$ vanish we have the following result.

\begin{prop}\label{prop:lammin1AHzero}
	Let $\mathcal{P} \to N$ be a $\U(2)$-structure satisfying equations (\ref{eq:StructLamMin1}) and the additional conditions $A_a = 0$ and $H_{ab} = 0.$ Then $N$ is locally equivalent to the product $\R^2 \times X$ of Euclidean 2-space with a Ricci-flat K\"ahler 4-manifold $X.$
	
	Conversely, the product $\Sigma \times X$ of a flat 2-manifold $\Sigma$ with a Ricci-flat K\"ahler 4-manifold $X$ carries a $\U(2)$-structure satisfying equations (\ref{eq:StructLamMin1}) and the additional conditions $A_a = 0$ and $H_{ab} = 0.$
\end{prop}

\begin{proof}
	If $A_a = 0$ and $H_{ab} = 0,$ the $\U(2)$-structure $\mathcal{P}$ is torsion-free, and it follows that the reduced holonomy group of the metric $g_{\varphi}|_N$ is contained in $\U(2).$ The only possibilities are the trivial group or $\SU(2),$ and in both these cases $N$ is locally equivalent to the metric product of $\R^2$ with a Ricci-flat K\"ahler metric. The converse statement is a straightforward computation, which we omit.
\end{proof}

\begin{example}
	By Proposition \ref{prop:lammin1AHzero}, products $\mathrm{T}^2 \times \mathrm{T}^4$ and $\mathrm{T}^2 \times \mathrm{K3}$ provide compact examples of 6-manifolds with $\U(2)$-structures satisfying equations (\ref{eq:StructLamMin1}).
\end{example}

\subsubsection{Type $A$}\label{sssect:U2mimin1TypeA}

Now suppose $\mathcal{P} \to N$ is a $\U(2)$-structure satisfying equations (\ref{eq:StructLamMin1}) and the additional condition $H_{ab} = 0.$ We assume that $A_a$ does not vanish identically on $\mathcal{P}$, as this case was considered in \S\ref{sssect:U2mimin11flat}.

In this case the structure equations (\ref{eq:StructLamMin1}) become
\begin{equation}\label{eq:StructLamMin1A}
\begin{aligned}
d \varpi &= \kappa_{a\bar{a}} \wedge \varpi + \overline{A_a \varpi} \wedge \theta_a, \\
d \theta_a &= - \kappa_{a \bar{b}} \wedge \theta_b,
\end{aligned}
\end{equation}
and the identities $d^2 \varpi = d^2 \theta_a = 0$ imply
\begin{equation}\label{eq:StructTwoLamMin1A}
\begin{aligned}
d \kappa_{a \bar{b}} &= - \kappa_{a \bar{c}} \wedge \kappa_{c \bar{b}} + R_{a\bar{b}c\bar{d}} \overline{\theta_c} \wedge \theta_d, \\
d A_a &= -A_b \kappa_{a \bar{b}} - 2 A_a \kappa_{b \bar{b}} + C_{ab} \overline{\theta_b},
\end{aligned}
\end{equation}
where $C_{ab}$ and $R_{a\bar{b}c\bar{d}}$ are complex-valued functions on $\mathcal{Q}$ satisfying $C_{ab} = C_{ba},$ $R_{a\bar{b}c\bar{d}} = R_{c\bar{b}a\bar{d}} = R_{a\bar{d}c\bar{b}},$ $R_{a\bar{b}c\bar{d}} = \overline{R_{b\bar{a}d\bar{c}}},$ and $R_{a\bar{c}c\bar{b}} = -A_a \overline{A_b}.$

\begin{prop}\label{prop:Lammin1TypeAExist}
	The $\U(2)^-$-structures satisfying equations (\ref{eq:StructLamMin1}) and the additional condition $H_{ab} = 0$ exist locally and depend on 3 functions of 2 variables.
\end{prop}

\begin{proof}
	This is a simple application of the methods of \cite{BryEDSNotes}. For a few more details see \cite{Ball19}.
\end{proof}

From equations (\ref{eq:StructLamMin1A}), there is a Riemannian submersion $N \to X,$ where $X$ is the leaf space of the distribution defined by $\varpi = 0.$ The metric $g_X$ on $X$ is given by
\begin{equation}
g_X = \theta_a \circ \overline{\theta_a},
\end{equation}
and it is K\"ahler, with K\"ahler form $\tfrac{i}{2} \theta_a \wedge \overline{\theta_a}.$ Equations (\ref{eq:StructTwoLamMin1A}) show that $A = \overline{A_a} \theta_a$ is a holomorphic section of $T^{1,0}_{\mathbb{C}} X \otimes \mathrm{K}_{X}^{* \otimes_2},$ where $\mathrm{K}_X$ denotes the canonical bundle of $X,$ and that the Ricci form of $g_X$ is given by
\begin{equation}
d \kappa_{a \bar{a}} = \overline{A_b} \theta_b \wedge A_a \overline{\theta_a}.
\end{equation}

Conversely, if $X, g_{X}$ is a K\"ahler 4-manifold with Ricci form $\zeta \wedge \overline{\zeta},$ where $\zeta$ is a holomorphic section of $T^{1,0}_{\mathbb{C}} X \otimes \mathrm{K}_{X}^{* \otimes_2},$ then it is possible to locally construct a $\U(2)^-$-structure satisfying equations (\ref{eq:StructLamMin1}) and the additional condition $H_{ab} = 0$ on a $\C$-bundle over $X.$ This is done by noting that the $\mathfrak{u}(1,1)$-valued 1-form
\begin{equation}
\mu = \left(\begin{array}{cc}
-\kappa_{a \bar{a}} & \zeta \\
\overline{\zeta} & \kappa_{a \bar{a}}
\end{array}\right)
\end{equation}
must satisfy the Maurer-Cartan equation $d \mu = -\mu \wedge \mu.$ Let $f$ be a map locally integrating $\mu,$ which exists by Cartan's Theorem \ref{thm:MaurerCartan}. Then define 
\begin{equation}
\left(\begin{array}{c}
\varpi \\
\overline{\varpi}
\end{array}\right) = f^{-1} \left( \begin{array}{c}
d z_1 \\
\overline{d z_1}
\end{array}\right)
\end{equation}
on the product of the $\U(2)$-coframe bundle of $X$ with $\C.$

\subsubsection{Type $H$}

Now suppose $\mathcal{P} \to N$ is a $\U(2)$-structure satisfying equations (\ref{eq:StructLamMin1}) and the additional condition $A_{a} = 0.$ We assume that $H_{ab}$ does not vanish identically on $\mathcal{P}$, as this case was considered in \S\ref{sssect:U2mimin11flat}.

Define complex-valued 1-forms $\sigma_{ab}$ on $\mathcal{P}$ by $\sigma_{ab} = - H_{ab} \varpi.$ Then we have
\begin{equation}
d \theta_a = -\kappa_{a \bar{b}} \wedge \theta_b - \sigma_{ab} \wedge \overline{\theta_b},
\end{equation}
and the exterior derivatives of the structure equations (\ref{eq:StructLamMin1}) imply that the $\mathfrak{sp}(2, \R)$-valued 1-form
\begin{equation}\label{eq:MCSp2}
\mu = \left(\begin{array}{cc}
\kappa_{a\bar{b}} & \sigma_{ab} \\
\overline{\sigma_{ab}} & -\kappa_{a\bar{b}}
\end{array}\right)
\end{equation}
satisfies the Maurer-Cartan equation $d \mu = - \mu \wedge \mu.$ Thus, by Cartan's Theorem \ref{thm:MaurerCartan}, there exists a locally defined map $f : \mathcal{P} \to \mathrm{Sp}(2, \R)$ satisfying $\mu = f^{-1} df.$ This map $f$ descends to give a map $\Sigma \to \mathrm{Sp}(2, \R) / \mathrm{U}(2),$ where $\Sigma$ is the leaf space of the (integrable) distribution $\varpi = 0$ on $N.$

\begin{prop}\label{prop:lammin1holcurve}
	The map $\Sigma \to \mathrm{Sp}(2, \R) / \mathrm{U}(2)$ constructed above is a holomorphic map, where $\mathrm{Sp}(2, \R) / \mathrm{U}(2)$ is considered with its $\mathrm{Sp}(2, \R)$-invariant K\"ahler structure.
\end{prop}

\begin{proof}
	If we write the Maurer-Cartan form of $\mathrm{Sp}(2, \R)$ in the form (\ref{eq:MCSp2}), the symmetric and antisymmetric forms
	\begin{equation}
	\sigma_{ab} \circ \overline{\sigma_{ab}} \:\:\: \text{and} \:\:\: \tfrac{i}{2} \sigma_{ab} \wedge \overline{\sigma_{ab}},
	\end{equation} 
	are invariant under the $\U(2)$-action generated by the vector fields dual to the $\kappa_{a\bar{b}},$ and descend to $\mathrm{Sp}(2, \R) / \mathrm{U}(2)$ to define the $\mathrm{Sp}(2, \R)$-invariant K\"ahler structure in question. The $(1,0)$-forms of the associated complex structure are the 1-forms whose pullback to $\mathrm{Sp}(2, \R)$ is a linear combination of $\sigma_{{11}},$ $\sigma_{{12}},$ and $\sigma_{22}.$ The result then follows immediately from the equation $\sigma_{ab} = - H_{ab} \varpi.$
\end{proof}

In light of Proposition \ref{prop:lammin1holcurve}, it is natural to try and take a holomorphic curve $\Sigma \to \mathrm{Sp}(2, \R) / \mathrm{U}(2)$ and construct a $\U(2)$-structure on the total space of the associated $\R^4$-bundle coming from the standard 4-dimensional representation of $\mathrm{Sp}(2, \R).$ We now indicate how this can be done.

Let $\Sigma \to \mathrm{Sp}(2, \R) / \mathrm{U}(2)$ be a holomorphic curve. By definition, the image of $\Sigma$ is an integral manifold of the exterior differential system $\mathcal{I} = \langle \sigma_{11} \wedge \sigma_{12} , \sigma_{12} \wedge \sigma_{22}, \sigma_{11} \wedge \sigma_{22} \rangle,$ which is well-defined on $\mathrm{Sp}(2, \R)/\U(2).$ Let $\mathcal{F}(\Sigma)$ denote the pullback of the $\U(2)$-bundle $\mathrm{Sp}(2,\R) \to \mathrm{Sp}(2, \R)/\U(2)$ to $\Sigma.$ We need to find a $(1,0)$-form on $\mathcal{F}(\Sigma)$ such that
\begin{equation}\label{eq:U2mivarpi}
d \varpi = \kappa_{a\bar{a}} \wedge \varpi.
\end{equation}
Let $z$ be a holomorphic coordinate on $\Sigma.$ Then $\varpi$ will be of the form $e^{F} d z$ for some function $F.$ Integrating the equation
\begin{align}
\frac{\partial F}{\partial \overline{z}} d \overline{z} - \frac{\partial F}{\partial z} d z = \kappa_{a\bar{a}}
\end{align}
to find $F,$ we find that $\varpi$ satisfies equation (\ref{eq:U2mivarpi}) as desired. Then on $\mathcal{F}(\Sigma)$ there will exist function $H_{ab}=H_{ba}$ with $\sigma_{ab} = -H_{ab} \varpi,$ and we obtain a $\U(2)$-structure on the total space of the associated $\R^4$-bundle satisfying equations (\ref{eq:StructLamMin1}) and $A_{a} = 0$ as required.

\subsection{The case $\lambda= -1/8$}

When $\lambda=-1/8,$ the structure equations (\ref{eq:U2miStructSix}) read
\begin{equation}\label{eq:StructLamMin18}
\begin{aligned}
d \varpi &= \kappa_{a\bar{a}} \wedge \varpi + \overline{A_a \varpi} \wedge \theta_a, \\
d \theta_a &= - \kappa_{a \bar{b}} \wedge \theta_b + H_{ab} \varpi \wedge \overline{\theta_b} - 3 i \eps_{ab} \overline{\varpi \wedge \theta_b}.
\end{aligned}
\end{equation}

As in \S\ref{ssect:LamMin1}, the $\U(2)$-structure $\mathcal{P}$ induces a $\SU(3)$-structure with associated differential forms $\Omega$ and $\Upsilon$ given by equation (\ref{eq:SU3Struct}). We have
\begin{equation}
\begin{aligned}
d \Omega &= -3 \Re \left( \eps_{ab} \varpi \wedge \theta_a \wedge \theta_b \right), \\
d \Upsilon &= \tfrac{3i}{2} \varpi \wedge \overline{\varpi} \wedge \theta_a \wedge \overline{\theta_a},
\end{aligned}
\end{equation}
so this $\SU(3)$-structure is never torsion-free.

\subsubsection{$A=H=0$}\label{sssect:U2mimin1by81flat}

When the torsion functions $A$ and $H$ vanish we have the following result.

\begin{prop}\label{prop:lammin1by8AHzero}
	Let $\mathcal{P} \to N$ be a $\U(2)$-structure satisfying equations (\ref{eq:StructLamMin18}) and the additional conditions $A_a = 0$ and $H_{ab} = 0.$ Then $N$ is locally equivalent to the twistor space $\mathcal{T}X$ of a Ricci-flat K\"ahler 4-manifold $X,$ endowed with a particular $\U(2)$-structure.
	
	Conversely, the twistor space $\mathcal{T}X$ of a Ricci-flat K\"ahler 4-manifold $X$ carries a $\U(2)$-structure satisfying equations (\ref{eq:StructLamMin18}) and the additional conditions $A_a = 0$ and $H_{ab} = 0.$
\end{prop}

\begin{proof}
	If $A_a=0$ and $H_{ab} = 0,$ the structure equations (\ref{eq:StructLamMin18}) become
	\begin{equation}\label{eq:StructLamMin18AHzero}
	\begin{aligned}
	d \varpi &= \kappa_{a\bar{a}} \wedge \varpi, \\
	d \theta_a &= - \kappa_{a \bar{b}} \wedge \theta_b - 3 i \eps_{ab} \overline{\varpi \wedge \theta_b}.
	\end{aligned}
	\end{equation}
	We see that there is a Riemannian submersion $N \to X,$ where $X$ is the leaf space of the distribution $\theta = 0$ on $N$ endowed with the metric $g_X = \theta_a \circ \overline{\theta_a}.$ The Levi-Civita form of $X$ is given by the components of $\kappa_{a\bar{b}}$ and $\varpi.$ It follows from the first equation of (\ref{eq:StructLamMin18AHzero}) that $g_X$ is Ricci-flat. From the appearance of $\varpi$ in the second equation of (\ref{eq:StructLamMin18AHzero}) we may identify $N$ with the total space of the twistor space $\mathcal{T}(X).$
	
	The converse follows by reversing the steps in the argument just given.
\end{proof}

\begin{example}
	By Proposition \ref{prop:lammin1by8AHzero}, twistor spaces $\mathcal{T} \, \mathrm{T}^4$ and $\mathcal{T} \, \mathrm{K3}$ provide compact examples of 6-manifolds with $\U(2)$-structures satisfying equations (\ref{eq:StructLamMin18}).
\end{example}

\subsubsection{Type $A$}

There are no $\U(2)$-structures satisfying equations (\ref{eq:StructLamMin18}) and the additional conditions $H_{ab} = 0, A_a \neq 0.$

\begin{prop}
	Let $\mathcal{P} \to N$ be a $\U(2)$-structure satisfying equations  (\ref{eq:StructLamMin18}). If $H_{ab}=0,$ then $A_a=0$
\end{prop}
\begin{proof}
	If $H_{ab}=0,$ a straightforward calculation yields that the equations $d^2 \varpi = d^2 \theta_a = 0$ imply $A_a=0.$
\end{proof}

\subsubsection{Type $H$}

Now suppose $\mathcal{P} \to N$ is a $\U(2)$-structure satisfying equations (\ref{eq:StructLamMin18}) and the additional condition $A_{a} = 0.$ We assume that $H_{ab}$ does not vanish identically on $\mathcal{P}$, as this case was considered in \S\ref{sssect:U2mimin1by81flat}.

It will be useful to use real notation. Let $1 \leq a, b \leq 3$ and $1 \leq i, j \leq 2$ be indices with the specified ranges, and define real-valued 1-forms $\varpi_i$ by $\varpi = \varpi_1 + i \varpi_2,$ real-valued functions $S_{aij}$ by
\begin{equation}
\begin{aligned}
S_{111} &= 2 \Im H_{12}, & S_{211} &= \Re H_{11} + \Re H_{22}, & S_{311} &= - \Im H_{11} + \Im H_{22}, \\
S_{112} &= 2 \Re H_{12}, & S_{212} &= - \Im H_{11} - \Im H_{22}, & S_{312} &= - \Re H_{11} + \Re H_{22}, \\
S_{122} &= -2 \Im H_{12}, & S_{222} &= - \Re H_{11} - \Re H_{22}, & S_{322} &= \Im H_{11} - \Im H_{22},
\end{aligned}
\end{equation}
Next, define real-valued 1-forms $\gamma_{ai}$, $\psi_{ab} = -\psi_{ba}$ and $\rho_{ij} = -\rho_{ji}$ by
\begin{equation}
\begin{aligned}
\psi_{23} &=  - i \left( \kappa_{1\bar{1}} - \kappa_{2\bar{2}} \right), & \gamma_{ai} &= S_{aij} \varpi_j, \\
\psi_{31} &= -2 \Re \kappa_{1\bar{2}}, & \rho_{12} &= -i \left( \kappa_{1\bar{1}} + \kappa_{2\bar{2}} \right), \\
\psi_{12} &= 2 \Im \kappa_{1\bar{2}}. & &
\end{aligned}
\end{equation}

Then the identities $d^2 \varpi = 0,$ $d^2 \theta = 0,$ and $d^2 \kappa = 0,$ together with the assumption that $H$ does not vanish identically, imply that the $\mathfrak{so}(3,3)$-valued 1-form
\begin{equation}
\mu = \left(\begin{array}{ccc}
0 & 0 & 3 \varpi_i \\
0 & \psi_{ab} & \gamma_{ai} \\
3 \varpi_j & \gamma_{aj} & \rho_{ij}
\end{array}\right)
\end{equation}
satisfies the Maurer-Cartan equation $d \mu = -\mu \wedge \mu.$

The proof of the following theorem follows the same argument as the proof of Theorem \ref{thm:TypeSCons}, so we omit it.

\begin{thm}
	Each $\U(2)^{+}$-structure satisfying structure equations (\ref{eq:StructLamMin18}) with $A_a=0$ and $H_{ab}$ not identically zero is locally equivalent to the $\R^4$-bundle over a timelike maximal immersion of a surface $\Sigma$ in $\SO(3,3)/\SO(3,2)$ associated to the flat $\mathfrak{sl}(4, \R)$-connection arising from the isomorphism $\mathfrak{so}(3,3) \cong \mathfrak{sl}(4, \R).$.
	
	Conversely, given a maximal timelike surface $\Sigma$ in $SO(3,3)/SO(3,2)$ one may construct a 6-manifold $N$ carrying a $\U(2)^{+}$-structure satisfying structure equations (\ref{eq:StructLamMin18}) with $A_a=0$ on the $\R^4$-bundle over $\Sigma$ associated to the flat $\mathfrak{sl}(4, \R)$-connection arising from the isomorphism $\mathfrak{so}(3,3) \cong \mathfrak{sl}(4, \R).$
\end{thm}

\subsection{The case $\lambda= 2/5$}

When $\lambda=2/5,$ the structure equations (\ref{eq:U2miStructSix}) read
\begin{equation}\label{eq:StructLam2by5}
\begin{aligned}
d \varpi &= \kappa_{a\bar{a}} \wedge \varpi + \overline{A_a \varpi} \wedge \theta_a + 3 i \eps_{ab} \overline{\theta_a \wedge \theta_b}, \\
d \theta_a &= - \kappa_{a \bar{b}} \wedge \theta_b + H_{ab} \varpi \wedge \overline{\theta_b}.
\end{aligned}
\end{equation}

As in \S\ref{ssect:LamMin1}, the $\U(2)$-structure $\mathcal{P}$ induces a $\SU(3)$-structure with associated differential forms $\Omega$ and $\Upsilon$ given by equation (\ref{eq:SU3Struct}). We have
\begin{equation}
\begin{aligned}
d \Omega &= -3 \Re \Upsilon, \\
d \Upsilon &= -\tfrac{3i}{2} \theta_a \wedge \overline{\theta_a} \wedge \theta_b \wedge \overline{\theta_b}
\end{aligned}
\end{equation}
so this $\SU(3)$-structure is never torsion-free.

\subsubsection{$A=H=0$}\label{sssect:U2mi2by51flat}

When the torsion functions $A$ and $H$ vanish we have the following result.

\begin{prop}\label{prop:lam2by5AHzero}
	Let $\mathcal{P} \to N$ be a $\U(2)$-structure satisfying equations (\ref{eq:StructLam2by5}) and the additional conditions $A_a = 0$ and $H_{ab} = 0.$ Then $N$ is locally equivalent to a particular (described in the proof) $\U(2)$-structure on product $\R^2 \times X$ of Euclidean 2-space with a Ricci-flat K\"ahler 4-manifold $X.$
	
	Conversely, any product of this form carries a $\U(2)$-structure satisfying equations (\ref{eq:StructLam2by5}) and the additional conditions $A_a = 0$ and $H_{ab} = 0.$
\end{prop}

\begin{proof}
	If $A_a = 0$ and $H_{ab} = 0,$ the structure equations (\ref{eq:StructLam2by5}) become
	\begin{equation}\label{eq:StructLam2by5AHzero}
	\begin{aligned}
	d \varpi &= \kappa_{a\bar{a}} \wedge \varpi + 3 i \eps_{ab} \overline{\theta_a \wedge \theta_b}, \\
	d \theta_a &= - \kappa_{a \bar{b}} \wedge \theta_b.
	\end{aligned}
	\end{equation}
	
	The identities $d^2 \varpi = d^2 \theta_a = 0$ imply
	\begin{equation}\label{eq:StructLam2by5AHzeroCurv}
	d \kappa_{a \bar{b}} = - \kappa_{a \bar{c}} \wedge \kappa_{c \bar{b}} + R_{a\bar{b}c\bar{d}} \overline{\theta_c} \wedge \theta_d,
	\end{equation}
	where $R_{a\bar{b}c\bar{d}}$ are complex-valued functions on $\mathcal{Q}$ satisfying $R_{a\bar{b}c\bar{d}} = R_{c\bar{b}a\bar{d}} = R_{a\bar{d}c\bar{b}},$ $R_{a\bar{b}c\bar{d}} = \overline{R_{b\bar{a}d\bar{c}}},$ and $R_{a\bar{c}c\bar{b}} = 0.$ In particular, $d \kappa_{a \bar{a}} = 0,$ so locally $\kappa_{a\bar{a}} = i \, d s$ for some real valued function $s.$ By multiplying $\varpi$ and $\theta_a$ by suitable powers of $e^{is}$ we may assume that $\kappa_{2\bar{2}} = -\kappa_{1\bar{1}}$ in the structure equations (\ref{eq:StructLam2by5AHzero}).
	
	There is a Riemannian submersion $N \to X,$ where $X$ is the leaf space of $\theta = 0,$ equipped with the K\"ahler metric $g_X = \theta_a \circ \overline{\theta_a}.$ This metric is Ricci-flat by equation (\ref{eq:StructLam2by5AHzeroCurv}). Finally, $d \left(\eps_{ab} \overline{\theta_a \wedge \theta_b} \right) = 0$ on $X,$ so there exists a locally defined 1-form on $X$ $\zeta$ with $d \zeta = \eps_{ab} \overline{\theta_a \wedge \theta_b}.$ Thus, on $N$ we may introduce a complex coordinate $z$ such that $\varpi = d z + 3 i \zeta.$ This proves the first part of the theorem. The converse follows by reversing the construction just described.
\end{proof}

\begin{example}
	If $X$ is a hyperk\"ahler 4-manifold with $\left[ \tfrac{1}{2\pi} \omega_J \right]$ and $\left[ \tfrac{1}{2\pi} \omega_K \right]$ integral classes, then the proof of Proposition \ref{prop:lam2by5AHzero} implies that associated $\mathrm{T}^2$-bundle over $X$ is an example of a compact 6-manifold $N$ with a $\U(2)$-structure satisfying equations (\ref{eq:StructLam2by5}) and $A_a = 0$ and $H_{ab} = 0.$ One just needs to take the real and imaginary parts of $\varpi$ to be connection forms with curvature the appropriate multiples of $\omega_J$ and $\omega_K.$
\end{example}

\subsubsection{Type A}

Now suppose $\mathcal{P} \to N$ is a $\U(2)$-structure satisfying equations (\ref{eq:StructLam2by5}) and the additional condition $H_{ab} = 0.$ We assume that $A_a$ does not vanish identically on $\mathcal{P}$, as this case was considered in \S\ref{sssect:U2mi2by51flat}.

In this case the structure equations (\ref{eq:StructLam2by5}) become
\begin{equation}\label{eq:StructLam2by5A}
\begin{aligned}
d \varpi &= \kappa_{a\bar{a}} \wedge \varpi + \overline{A_a \varpi} \wedge \theta_a + 3 i \eps_{ab} \overline{\theta_a \wedge \theta_b}, \\
d \theta_a &= - \kappa_{a \bar{b}} \wedge \theta_b,
\end{aligned}
\end{equation}
and the identities $d^2 \varpi = d^2 \theta_a = 0$ imply
\begin{equation}\label{eq:StructTwoLam2by5A}
\begin{aligned}
d \kappa_{a \bar{b}} &= - \kappa_{a \bar{c}} \wedge \kappa_{c \bar{b}} + R_{a\bar{b}c\bar{d}} \overline{\theta_c} \wedge \theta_d, \\
d A_a &= -A_b \kappa_{a \bar{b}} - 2 A_a \kappa_{b \bar{b}} + C_{ab} \overline{\theta_b},
\end{aligned}
\end{equation}
where $C_{ab}$ and $R_{a\bar{b}c\bar{d}}$ are complex-valued functions on $\mathcal{Q}$ satisfying $C_{ab} = C_{ba},$ $R_{a\bar{b}c\bar{d}} = R_{c\bar{b}a\bar{d}} = R_{a\bar{d}c\bar{b}},$ $R_{a\bar{b}c\bar{d}} = \overline{R_{b\bar{a}d\bar{c}}},$ and $R_{a\bar{c}c\bar{b}} = -A_a \overline{A_b}.$

These structure equations are almost identical to those of \S\ref{sssect:U2mimin1TypeA}, and because of this an almost identical argument to that of Proposition \ref{prop:Lammin1TypeAExist} yields the following result.

\begin{prop}\label{prop:Lam2by5TypeAExist}
	The $\U(2)^-$-structures satisfying equations (\ref{eq:StructLam2by5}) and the additional condition $H_{ab} = 0$ exist locally and depend on 3 functions of 2 variables.
\end{prop}

We will not say more about the structures of this type, except to note that the calculations after Proposition \ref{prop:Lammin1TypeAExist} may be repeated in this case too.

\subsection{The case $\lambda= 3/4$}\label{ssect:Lam3by4}

When $\lambda=3/4,$ the structure equations (\ref{eq:U2miStructSix}) read
\begin{equation}\label{eq:StructLam3by4}
\begin{aligned}
d \varpi &= \kappa_{a\bar{a}} \wedge \varpi + \overline{A_a \varpi} \wedge \theta_a - \tfrac{4i}{3} \eps_{ab} \overline{\theta_a \wedge \theta_b}, \\
d \theta_a &= - \kappa_{a \bar{b}} \wedge \theta_b + H_{ab} \varpi \wedge \overline{\theta_b} + \tfrac{5i}{6} \eps_{ab} \overline{\varpi \wedge \theta_b}.
\end{aligned}
\end{equation}

As in \S\ref{ssect:LamMin1}, the $\U(2)$-structure $\mathcal{P}$ induces a $\SU(3)$-structure with associated differential forms $\Omega$ and $\Upsilon$ given by equation (\ref{eq:SU3Struct}). We have
\begin{equation}
\begin{aligned}
d \Omega &= -3 \Re \Upsilon, \\
d \Upsilon &= -\tfrac{3i}{2} \theta_a \wedge \overline{\theta_a} \wedge \theta_b \wedge \overline{\theta_b}
\end{aligned}
\end{equation}
so this $\SU(3)$-structure is never torsion-free. However, the rescaled $\SU(3)$-structure defined by
\begin{equation}
\begin{aligned}
\tilde{\Omega} &= \tfrac{i}{2} \left( \tfrac{25}{36} \, \varpi \wedge \overline{\varpi} + \tfrac{10}{9} \theta_a \wedge \overline{\theta_a} \right), \\
\tilde{\Upsilon} &= \tfrac{25}{27} \, \eps_{ab} \varpi \wedge \theta_a \wedge \theta_b,
\end{aligned}
\end{equation}
is \emph{nearly K\"ahler}, meaning that
\begin{equation}
d \tilde{\Omega} = 3 \Re \tilde{\Upsilon}, \:\:\:\:\: d \Im \tilde{\Upsilon} = -2 \tilde{\Omega} \wedge \tilde{\Omega}.
\end{equation}

\subsubsection{$A = H = 0$}

If the torsion functions $A$ and $H$ vanish identically, we have the following result.

\begin{prop}
	Let $\mathcal{P} \to N$ be a $\U(2)$-structure on a $6$-manifold $N$ satisfying structure equations (\ref{eq:StructLam3by4}) and the additional conditions $A_a = 0$ and  $H_{ab} = 0.$ Then $N$ is locally equivalent to the twistor space $\mathbb{T} X$ of an anti-self-dual Einstein 4-manifold $X$ with scalar curvature ${80}/{9}$ endowed with a particular $\U(2)$-structure, such that the fibres have constant curvature $25/9.$
	
	Conversely, the twistor space $\mathcal{T} (X)$ of an anti-self-dual Einstein 4-manifold $X$ with scalar curvature ${80}/{9}$ carries a $\U(2)$-structure satisfying structure equations (\ref{eq:StructLam3by4}) and the additional conditions $A_a = 0$ and  $H_{ab} = 0.$
\end{prop}

\begin{proof}
	If $A_a=0$ and $H_{ab} = 0,$ the structure equations (\ref{eq:StructLam3by4}) become
	\begin{equation}\label{eq:StructLam3by4AHzero}
	\begin{aligned}
	d \varpi &= \kappa_{a\bar{a}} \wedge \varpi - \tfrac{4i}{3} \eps_{ab} \overline{\theta_a \wedge \theta_b}, \\
	d \theta_a &= - \kappa_{a \bar{b}} \wedge \theta_b + \tfrac{5i}{6} \eps_{ab} \overline{\varpi \wedge \theta_b}.
	\end{aligned}
	\end{equation}
	We see that there is a Riemannian submersion $N \to X,$ where $X$ is the leaf space of the distribution $\theta = 0$ on $N$ endowed with the metric $g_X = \theta_a \circ \overline{\theta_a}.$ The Levi-Civita form of $X$ is given by the components of $\kappa_{a\bar{b}}$ and $\varpi.$ It follows from the first equation of (\ref{eq:StructLamMin18AHzero}) that $g_X$ is anti-self-dual Einstein with scalar curvature $80/9$. From the appearance of $\varpi$ in the second equation of (\ref{eq:StructLam3by4AHzero}) we may identify $N$ with the total space of the twistor space $\mathcal{T}(X).$ The metric $\varpi \circ \overline{\varpi}$ on the fibres has constant curvature $25/9.$ This proves the first statement. The converse follows by reversing the steps in the argument just given.
\end{proof}

\subsubsection{Types $A$ and $H$}

There are no $\U(2)$-structures satisfying equations (\ref{eq:StructLam3by4}) and the additional conditions $H_{ab} = 0, A_a \neq 0$ or $A_a=0, H_{ab} \neq 0$.

\begin{prop}
	Let $\mathcal{P} \to N$ be a $\U(2)$-structure satisfying equations (\ref{eq:StructLam3by4}).
	\begin{enumerate}
		\item If $H_{ab}=0$ then $A_a=0.$
		\item If $A_a=0$ then $H_{ab}=0.$
	\end{enumerate}
\end{prop}
\begin{proof}
	If $H_{ab}=0,$ then the equations $d^2 \varpi = d^2 \theta = 0$ imply $A_a=0,$ while if $A_a=0,$ then the equations $d^2 \varpi = d^2 \theta = 0$ imply $H=0.$
\end{proof}

\section{Laplace solitons}\label{sect:LapSols}

In this section we present two explicit examples of complete, steady, gradient solitons for the Laplacian flow (\ref{eq:LapFlow}) with special torsion of negative type. These solitons arise via an ansatz motivated by the results of \S\ref{sect:U2mi}. The $\G_2$-structures considered in this section are of the form
\begin{equation}
\varphi = \tfrac{i}{2} dr \wedge \left( F(r)^2 \varpi \wedge \overline{\varpi} + G(r)^2 \theta_a \wedge \overline{\theta_a} \right) + F(r) G(r)^2 \Re \left( \eps_{ab} \varpi \wedge \theta_a \wedge \theta_b \right),
\end{equation}
where $\varpi$ and $\theta_a$ are the components of the tautological 1-form for a $\U(2)$-structure $\mathcal{P} \to N$ over a 6-manifold satisfying structure equations (\ref{eq:StructLamMin18}) or (\ref{eq:StructLam2by5}).

The gradient Laplace soliton system is not involutive, and determining the local generality of gradient Laplace solitons is an open problem. The examples presented below both have a local generality of 2 functions of 3 variables, which hence gives a lower bound for the generality of the gradient Laplace soliton system.

\subsection{Example 1}\label{ssect:LapSolEg1}

Let $(X, g_X, \Omega)$ be a Ricci-flat anti-self-dual 4-manifold. Let $\mathcal{R} \to X$ be the induced $\SO(4)$-structure bundle of $X$ consisting of coframes $u : T_p X \to \R^4$ such that
\begin{equation}
\begin{aligned}
g_X &= u^* \left((e^1)^2 + (e^2)^2 + (e^3)^2 + (e^4)^2\right)
\end{aligned}
\end{equation}
and let $\eta = \left(\eta_1, \eta_2, \eta_3, \eta_4 \right)$ denote the $\R^4$-valued tautological 1-form on $\mathcal{R}.$

On $\mathcal{R},$ we have Cartan's first structure equation,
\begin{equation}\label{eq:ASDRFStruct1}
d \left( \begin{array}{c}
\eta_1 \\
\eta_2 \\
\eta_3 \\
\eta_4
\end{array} \right) = {\left(\begin {array}{cccc}
	 0 & \zeta_{{1}} +\xi_{{1}} & \zeta_{{2}}-\xi_{{2}} & -\zeta_{{3}}-\xi_{{3}} \\
	 -\zeta_{{1}}-\xi_{{1}} & 0 & -\zeta_{{3}}+\xi_{{3}} & -\zeta_{{2}}-\xi_{{2}} \\ 
	 -\zeta_{{2}}+\xi_{{2}} & \zeta_{{3}}-\xi_{{3}} & 0 & \zeta_{{1}}-\xi_{{1}}	\\
	 \zeta_{{3}} + \xi_{{3}} & \zeta_{{2}}+\xi_{{2}} & -\zeta_{{1}}+\xi_{{1}} &0 
	 \end {array} \right)} \wedge \left( \begin{array}{c}
\eta_1 \\
\eta_2 \\
\eta_3 \\
\eta_4
\end{array} \right).
\end{equation}
where $\xi_1,$ $\xi_2,$ $\xi_3,$ $\zeta_1,$ $\zeta_2,$ and $\zeta_3$ are the components of the $\mathfrak{so}(4)$-valued Levi-Civita form on $\mathcal{R},$ written in a way to reflect the isomorphism $\mathfrak{so}(4) \cong \mathfrak{su}(2)^+ \oplus \mathfrak{su}(2)^-.$ As a consequence of the assumption that $g_X$ is anti-self-dual Ricci-flat, Cartan's second structure equation reads
\begin{equation}\label{eq:ASDRFStruct2}
\begin{aligned}
d \left(\begin{array}{c}
\zeta_1 \\
\zeta_2 \\
\zeta_3
\end{array}\right) &= -2 \, \left(\begin{array}{c}
\zeta_2 \wedge \zeta_3 \\
\zeta_3 \wedge \zeta_1 \\
\zeta_1 \wedge \zeta_2
\end{array}\right) \\
d \left(\begin{array}{c}
\xi_1 \\
\xi_2 \\
\xi_3
\end{array}\right) &= -2 \, \left(\begin{array}{c}
\xi_2 \wedge \xi_3 \\
\xi_3 \wedge \xi_1 \\
\xi_1 \wedge \xi_2
\end{array}\right) + \left(\begin{array}{ccc}
W_{11} & W_{12} & W_{13} \\
W_{12} & W_{22} & W_{23} \\
W_{13} & W_{23} & W_{33}
\end{array}\right) \left(\begin{array}{c}
\eta_1 \wedge \eta_2 - \eta_3 \wedge \eta_4 \\
\eta_3 \wedge \eta_1 + \eta_4 \wedge \eta_2 \\
\eta_4 \wedge \eta_1 + \eta_2 \wedge \eta_3
\end{array}\right),
\end{aligned}
\end{equation}
for functions $W_{ij}$ on $\mathcal{R}$ satisfying $W_{11} + W_{22} + W_{33} = 0,$ representing the Weyl tensor of $g_X.$

Now, let $N$ denote the twistor space on $X,$ the unit sphere bundle of $\Lambda^2_+ T^* X.$ This is the quotient of $\mathcal{R}$ by $\U(2)^- \subset \SO(4),$ and thus $\mathcal{R}$ is naturally a $\U(2)$-structure on $N$. The components of the tautological form of this $\U(2)$-structure are $\zeta_2,$ $\zeta_3,$ $\eta_1,$ $\eta_2,$ $\eta_3,$ and $\eta_4.$ The forms $\xi_1, \zeta_1, \zeta_2,$ and $\zeta_3$ are the natural connection forms. It is easy to see that the forms
\begin{equation}
\begin{aligned}
\mathrm{vol}_{F} &= \zeta_2 \wedge \zeta_3, \\
\Omega_{X} &= \eta_1 \wedge \eta_2 + \eta_3 \wedge \eta_4, \\
\Gamma &= \zeta_3 \wedge \left(\eta_3 \wedge \eta_1 + \eta_2 \wedge \eta_4 \right) + \zeta_2 \wedge \left( \eta_4 \wedge \eta_1 + \eta_3 \wedge \eta_2 \right),
\end{aligned}
\end{equation}
descend to $N.$

Now, on the product $\R \times N,$ with coordinate $r$ in the $\R$-direction consider a $\G_2$-structure of the form
\begin{equation}\label{eq:ASDRFG2}
\varphi = d r \wedge \left(f(r)^2 \mathrm{vol}_F + g(r)^2 \Omega_{X} \right) + f(r) g(r)^2 \, \Gamma,
\end{equation}
where $f(r)$ and $g(r)$ are non-vanishing functions. A calculation using the structure equations (\ref{eq:ASDRFStruct1}) and (\ref{eq:ASDRFStruct2}) yields that $\varphi$ is closed if and only if
\begin{equation}\label{eq:ASDRFClosedCond}
g'(r) = - \frac{g(r)\left(f'(r) + 2 \right)}{2 \, f(r)},
\end{equation}
which we assume from now on. In this case, the torsion 2-form is given by
\begin{equation}\label{eq:ASDRFtau}
\tau = \left(f'(r) + 2 \right) \left(2 \, f(r) \, \mathrm{vol}_F - \frac{g(r)^2}{f(r)} \, \Omega_{X} \right).
\end{equation}

\begin{remark}\label{rmk:ASDRFTorsFree}
	We see from equation (\ref{eq:ASDRFtau}) that the condition for $\varphi$ to define a torsion-free $\G_2$-structure is that $f'(r) = -2,$ from which it follows that $g'(r) = 0.$ This solution corresponds to the torsion-free $\G_2$-structure on the metric product $X \times \R^3,$ which has reduced holonomy a subgroup of $\SU(2).$
\end{remark}

Let $V = v(r) \, \partial_r$ be a vector field in the $\R$-direction. Another calculation using the structure equations (\ref{eq:ASDRFStruct1}) and (\ref{eq:ASDRFStruct2}) shows that the condition that $\left( \varphi, V \right)$ define a Laplace soliton, namely
\begin{equation}
d \tau = c \, \varphi + \mathcal{L}_V \, \varphi,
\end{equation}
is equivalent to the ODE system
\begin{subequations}\label{eq:ASDRFODE}
	\begin{align}
	f'(r) &= v(r) \, \left( f(r) + 2 \right) + \tfrac{1}{2} c \, f(r)^2, \label{eq:ASDRFODEa} \\
	v'(r) &= \frac{v(r) \, \left( f(r) \, v(r) + 2 \right) + c \left( 3 c \, f(r)^2 - 12 \, f(r) v(r) - 14 \right)}{6 \, f(r)}.
	\end{align}
\end{subequations}

We assume from now on that $c = 0,$ i.e.\ that $\varphi$ is a steady Laplace soliton. In this case it is possible to solve the system (\ref{eq:ASDRFODE}) explicitly. We have that
\begin{equation}
\frac{d}{dr} \left( f(r)^2 \, v(r) \right) = 0,
\end{equation}
so there exists a constant $k_1$ (which we shall assume non-zero) such that
\begin{equation}\label{eq:ASDRFv}
v(r) = - \frac{2\, k_1}{f(r)^2}.
\end{equation}
Then, after a simple integration (assuming $f(r) > k_1,$ the other case being similar), equation (\ref{eq:ASDRFODEa}) implies
\begin{equation}
-\frac{f(r)}{2} - \frac{k_1}{2} \log(f(r) - k_1 ) = r + k - \tfrac{1}{2} k_1,
\end{equation}
for some constant $k.$ This equation can be inverted by means of the Lambert W function, and we have (after shifting $r$ so we may assume $k_2 = 0$)
\begin{equation}\label{eq:ASDRFf}
f (r) = k_1 \left( \LamW \left( -\frac{1}{k_1}  \exp {\frac{2\,r}{k_1}} \right) + 1 \right).
\end{equation}
Next, equation (\ref{eq:ASDRFClosedCond}) together with (\ref{eq:ASDRFODEa}) implies that
\begin{equation}
\frac{d}{dr} \left(\frac{f(r)-k_1}{k_1 \, f(r) g(r)^2} \right) = 0,
\end{equation}
so
\begin{equation}\label{eq:ASDRFg}
g(r)^2 = \frac{f(r) - k_1}{k_1 k_2 f(r)},
\end{equation}
for some constant $k_2$ (with the same parity as $k_1$).

\begin{thm}
The 2-parameter family of closed $\G_2$-structures $\varphi$ defined by equation (\ref{eq:ASDRFG2}) and the equations (\ref{eq:ASDRFf}) and (\ref{eq:ASDRFg}) define steady gradient Laplace solitons on the manifold $M = \R \times N,$ where $N$ is the twistor space of an anti-self-dual Ricci-flat 4-manifold $X$. The metric $g_{\varphi}$ on $M$ is complete if and only if the metric on $X$ is. 
\end{thm}

\begin{proof}
	The fact that $\varphi$ defines a steady Laplace soliton follows from the calculations above. It is easy to see that the vector field $V = \operatorname{grad} h$ for
	\begin{equation}
	h(r) = - 2 k_1 \int \frac{1}{f(r)^2} d r,
	\end{equation}
	so $\varphi$ is a gradient soliton.
	
	Finally, completeness of the metric $g_{\varphi}$ when $X$ is complete follows from the fact that $\left| f(r) \right| > \left| k_1 \right| > 0,$ so neither $f(r)$ nor $g(r)$ vanish on $\R,$ and their domain is all of $\R.$
\end{proof}

Now assume that $X$ is a compact anti-self-dual Ricci-flat 4-manifold, for instance a K3 surface or a 4-torus. The volume form $\mathrm{vol}_M$ of the steady Laplace soliton $\varphi$ on $M = \R \times N$ just constructed is
\begin{equation}
\begin{aligned}
\mathrm{vol}_M &= f(r)^2 g(r)^4 \, dr \wedge \mathrm{vol}_F \wedge \eta_1 \wedge \eta_2 \wedge \eta_3 \wedge \eta_4, \\
& = \frac{1}{k_1^2 k_2} \left(f(r) - k_1 \right)^2 \, dr \wedge \mathrm{vol}_F \wedge \eta_1 \wedge \eta_2 \wedge \eta_3 \wedge \eta_4.
\end{aligned}
\end{equation}

Now, as $r \to + \infty,$ $f(r)$ approaches $k_1$ exponentially quickly, so the volume of the $r$-positive end of $M$ is finite. In contrast, as $r \to - \infty,$ $f(r)$ is asymptotic to a linear function with gradient $-2$, so the volume growth at the $r$-negative end is cubic in $r$ (which is the arc length parameter along $\R$). In fact, from the asymptotics of $f(r)$ we see that as $r \to - \infty$ the closed $\G_2$-structure $\varphi$ is asymptotic to the torsion-free $\G_2$-structure on $X \times \R^3$ described in Remark \ref{rmk:ASDRFTorsFree}.

\subsection{Example 2}\label{ssect:LapSolEg2}

Let $X$ be a hyperk\"ahler 4-manifold with triple of 2-forms $\left(\Omega_I, \Omega_J, \Omega_K \right)$ such that $\left[ \tfrac{1}{2\pi} \Omega_J \right]$ and $\left[ \tfrac{1}{2\pi} \Omega_K \right]$ are integral classes. Let $\mathcal{R} \to X$ be the torsion-free $\SU(2)$-structure induced by the hyperk\"ahler structure on $X,$ consisting of coframes for which
\begin{equation}
\begin{aligned}
\Omega_I &= u^* \left(e^1 \wedge e^2 + e^3 \wedge e^4 \right), \\
\Omega_J &= u^* \left(e^1 \wedge e^3 + e^4 \wedge e^2 \right), \\
\Omega_K &= u^* \left(e^4 \wedge e^1 + e^3 \wedge e^2 \right).
\end{aligned}
\end{equation}

Let $N$ be the total space of the $\mathrm{T^2}$ bundle classified by $\left[ \tfrac{1}{2\pi} \Omega_J \right]$ and $\left[ \tfrac{1}{2\pi} \Omega_K \right]$, and let $\pi_1$ and $\pi_2$ be connection forms with curvature $\Omega_J$ and $\Omega_K.$ Then, the pullback of $\mathcal{R}$ to $N$ has a coframe $\left(\pi_1, \pi_2, \eta_1, \ldots \eta_4, \xi_1, \xi_2, \xi_3 \right)$ satisfying structure equations of the form
\begin{equation}\label{eq:HKa}
\begin{aligned}
d \left( \begin{array}{c}
\eta_1 \\
\eta_2 \\
\eta_3 \\
\eta_4
\end{array} \right) &= {\left(\begin {array}{cccc}
	0 & \xi_{{1}} & -\xi_{{2}} & -\xi_{{3}} \\
	-\xi_{{1}} & 0 & \xi_{{3}} & -\xi_{{2}} \\ 
	\xi_{{2}} & -\xi_{{3}} & 0 & -\xi_{{1}}	\\
	\xi_{{3}} & \xi_{{2}} & \xi_{{1}} &0 
	\end {array} \right)} \wedge \left( \begin{array}{c}
\eta_1 \\
\eta_2 \\
\eta_3 \\
\eta_4
\end{array} \right), \\ 
d \pi_1 &= \eta_1 \wedge \eta_3 + \eta_4 \wedge \eta_2, \\
d \pi_2 &= \eta_4 \wedge \eta_1 + \eta_3 \wedge \eta_2,
\end{aligned}
\end{equation}
and
\begin{equation}\label{eq:HKb}
d \left(\begin{array}{c}
\xi_1 \\
\xi_2 \\
\xi_3
\end{array}\right) = -2 \, \left(\begin{array}{c}
\xi_2 \wedge \xi_3 \\
\xi_3 \wedge \xi_1 \\
\xi_1 \wedge \xi_2
\end{array}\right) + \left(\begin{array}{ccc}
W_{11} & W_{12} & W_{13} \\
W_{12} & W_{22} & W_{23} \\
W_{13} & W_{23} & W_{33}
\end{array}\right) \left(\begin{array}{c}
\eta_1 \wedge \eta_2 - \eta_3 \wedge \eta_4 \\
\eta_3 \wedge \eta_1 + \eta_4 \wedge \eta_2 \\
\eta_4 \wedge \eta_1 + \eta_2 \wedge \eta_3
\end{array}\right).
\end{equation}

Now, on the product $\R \times N,$ with coordinate $r$ in the $\R$-direction, consider the 3-form
\begin{equation}\label{eq:HKG2}
\varphi = d r \wedge \left(f(r)^2 \pi_1 \wedge \pi_2 + g(r)^2 \Omega_I \right) - f(r)g(r)^2 \pi_2 \wedge \Omega_J + f(r) g(r)^2 \pi_1 \wedge \Omega_K,
\end{equation} 
which defines a $\G_2$-structure on $\R \times N$ away from the zeroes of $f(r)$ and $g(r).$ A calculation using the structure equations (\ref{eq:HKa}) and (\ref{eq:HKb}) yields that $\varphi$ is closed if and only if
\begin{equation}\label{eq:HKclosed}
g'(r) = - \frac{f'(r) g(r)^2 + f(r)^2}{2 \, f(r) g(r)},
\end{equation}
which we assume from now one. In this case, the torsion 2-form is given by
\begin{equation}\label{eg:HKtau}
\tau = \left( f'(r) g(r)^2 - f(r)^2 \right) \left( \frac{2}{g(r)^2} \pi_1 \wedge \pi_2 - \frac{1}{f(r)} \Omega_I \right).
\end{equation}

\begin{remark}
	From equation (\ref{eg:HKtau}), the condition that $\varphi$ define a torsion-free $\G_2$-structure is the equation $f'(r) = f(r)^2/g(r)^2.$ The solutions to the system consisting of this equation and equation (\ref{eq:HKclosed}) are $f(r)= k^{\frac{2}{3}} / r^{\frac{1}{3}},$ $g(r) = \sqrt{3} k^{\frac{1}{3}} r^{\frac{1}{3}},$ for a constant $k.$ This defines a $\G_2$-structure on $N \times (0, + \infty)$ which is incomplete as $r \to 0,$ but is forward complete.
\end{remark}

As in the previous section, let $V = v(r) \, \partial_r$ be a vector field in the $\R$-direction. A calculation using the structure equations (\ref{eq:HKa}) and (\ref{eq:HKb}) shows that the condition that $( \varphi, V )$ defines a steady Laplace soliton, namely
\begin{equation}
d \tau = \mathcal{L}_V \, \varphi,
\end{equation}
is equivalent to the ODE system
\begin{subequations}
	\begin{align}
	f' (r) &= \frac{f(r)\left(g(r)^2 \, v(r) + 2 f(r) \right)}{2 g(r)^2}, \\
	g'(r) &= - \frac{g(r)^2 \, v(r) + 4 f(r)}{4 g(r)} , \\
	v' (r) &=  \frac{v(r) \left(g(r)^2 \, v(r) + 4 \, f(r)\right)}{2 \, g(r)^2}.
	\end{align}
\end{subequations}

This system can be solved explicitly. We have that
\begin{equation}
\frac{d}{dr} \left( g(r)^2 \, v(r) \right) = 0,
\end{equation}
so there exists a constant $k_1$ (which we shall assume non-zero) such that
\begin{equation}
v(r) = - \frac{2 \, k_1}{g(r)^2}.
\end{equation}
We then have
\begin{equation}
\frac{d}{dr} \left(g(r)^2 f(r) \left(f(r) - k_1 \right) \right) = 0.
\end{equation}
We assume that $0 < f(r) < k_1,$ so there exists a constant $k_2$ with
\begin{equation}\label{eq:HKg}
g(r) = \frac{k_2}{\sqrt{f(r)\left(k_1 - f(r) \right)}}.
\end{equation}
Then a simple integration yields that (after possibly shifting $r$ by a constant)
\begin{equation}\label{eq:HKf}
\frac{2 k_2^2}{k_1^3} \left(\log f(r) - \log (k_1 - f(r)) \right) + \frac{k_2^2}{k_1^2} \left(\frac{1}{k_1 - f(r)} - \frac{1}{f(r)} \right) = r.
\end{equation}
This equation defines $f(r)$ as a function of $r,$ since the left hand side,considered as a function of $f,$ is invertible on $\left(0, k_1 \right).$

\begin{thm}
	The 2-parameter family of closed $\G_2$-structures $\varphi$ defined by equation (\ref{eq:HKG2}) and the equations (\ref{eq:HKf}) and (\ref{eq:HKg}) define complete steady gradient Laplace solitons on the manifold $M = \R \times N.$
\end{thm}

\begin{proof}
	The fact that $\varphi$ defines a steady Laplace soliton follows from the calculations above. It is easy to see that the vector field $V$ is the gradient of the function $h$ defined by
	\begin{equation}
	h(r) = - \frac{2 k_1}{k_2^2} \int f(r) \left(k_1 - f(r) \right) dr. 
	\end{equation}
	
	Finally, completeness of the metric $g_{\varphi}$ follows from the fact that the function $f(r)$ defined by equation (\ref{eq:HKf}) is defined on all of $\R$ and satisfies $0 < f(r) < k_1.$
\end{proof}

As $r \to + \infty,$ $f(r)$ is asymptotic to $k_1 - k_2^2/(k_1^2 r),$ and as $r \to - \infty,$ $f(r)$ is asymptotic to $-k_2^2/(k_1^2 r).$ The volume form induced by $\varphi$ on $M$ is
\begin{equation}
\mathrm{vol}_M = \frac{k_2^2}{\left(k_1 - f(r)\right)^2} d r \wedge \pi_1 \wedge \pi_2 \wedge \mathrm{vol}_X, 
\end{equation}
where $\mathrm{vol}_X$ is the volume form associated to the hyperk\"ahler structure on $X.$ Thus, we see that the volume growth is cubic at the $r$-positive end and linear at the $r$-negative end of $\R \times N.$

\bibliography{QuadClosedU2PlusRefs}

\newcommand{\noop}[1]{}
\begin{bibdiv}
\begin{biblist}

\bib{Ball19}{thesis}{
      author={Ball, Gavin},
       title={{S}even {D}imensional {G}eometries with {S}pecial {T}orsion},
        type={Ph.D. Thesis},
        date={2019},
         url={https://hdl.handle.net/10161/18734},
}

\bib{BallConfFlat20}{article}{
      author={{Ball}, Gavin},
       title={{Closed G2-structures with conformally flat metric}},
        date={2020-02},
     journal={arXiv e-prints},
       pages={arXiv:2002.01634},
      eprint={2002.01634},
}

\bib{BaragSemi}{article}{
      author={Baraglia, D.},
       title={Moduli of coassociative submanifolds and semi-flat
  {$G_2$}-manifolds},
        date={2010},
        ISSN={0393-0440},
     journal={J. Geom. Phys.},
      volume={60},
      number={12},
       pages={1903\ndash 1918},
         url={https://doi.org/10.1016/j.geomphys.2010.07.006},
      review={\MR{2735277}},
}

\bib{BryEDSNotes}{article}{
      author={{Bryant}, R.~L.},
       title={{Notes on exterior differential systems}},
        date={2014-05},
     journal={ArXiv e-prints},
      eprint={1405.3116},
}

\bib{EDSBook}{book}{
      author={Bryant, R.~L.},
      author={Chern, S.~S.},
      author={Gardner, R.~B.},
      author={Goldschmidt, H.~L.},
      author={Griffiths, P.~A.},
       title={Exterior differential systems},
      series={Mathematical Sciences Research Institute Publications},
   publisher={Springer-Verlag, New York},
        date={1991},
      volume={18},
        ISBN={0-387-97411-3},
         url={https://doi.org/10.1007/978-1-4613-9714-4},
      review={\MR{1083148}},
}

\bib{BryXuLapFlow11}{article}{
      author={{Bryant}, Robert},
      author={{Xu}, Feng},
       title={{Laplacian Flow for Closed $G_2$-Structures: Short Time
  Behavior}},
        date={2011-01},
     journal={arXiv e-prints},
       pages={arXiv:1101.2004},
      eprint={1101.2004},
}

\bib{BryExcept}{article}{
      author={Bryant, Robert~L.},
       title={Metrics with exceptional holonomy},
        date={1987},
        ISSN={0003486X},
     journal={Annals of Mathematics},
      volume={126},
      number={3},
       pages={525\ndash 576},
         url={http://www.jstor.org/stable/1971360},
}

\bib{Bry05}{article}{
      author={Bryant, Robert~L.},
       title={Some remarks on {$G_2$}-structures},
        date={2006},
     journal={Proceedings of G{\"o}kova Geometry-Topology Conference 2005},
}

\bib{CleyIv08}{article}{
      author={Cleyton, Richard},
      author={Ivanov, Stefan},
       title={Curvature decomposition of $g_2$-manifolds},
        date={2008},
        ISSN={0393-0440},
     journal={Journal of Geometry and Physics},
      volume={58},
      number={10},
       pages={1429 \ndash  1449},
  url={http://www.sciencedirect.com/science/article/pii/S0393044008001022},
}

\bib{DonaldsonAdiabat17}{incollection}{
      author={Donaldson, Simon},
       title={Adiabatic limits of co-associative {K}ovalev-{L}efschetz
  fibrations},
        date={2017},
   booktitle={Algebra, geometry, and physics in the 21st century},
      series={Progr. Math.},
      volume={324},
   publisher={Birkh\"{a}user/Springer, Cham},
       pages={1\ndash 29},
      review={\MR{3702382}},
}

\bib{FFMNilSols16}{article}{
      author={Fern\'{a}ndez, Marisa},
      author={Fino, Anna},
      author={Manero, V\'{\i}ctor},
       title={Laplacian flow of closed {$G_2$}-structures inducing
  nilsolitons},
        date={2016},
        ISSN={1050-6926},
     journal={J. Geom. Anal.},
      volume={26},
      number={3},
       pages={1808\ndash 1837},
  url={https://doi-org.proxy3.library.mcgill.ca/10.1007/s12220-015-9609-3},
      review={\MR{3511459}},
}

\bib{FiYaHype18}{article}{
      author={Fine, Joel},
      author={Yao, Chengjian},
       title={Hypersymplectic 4-manifolds, the {$G_2$}-{L}aplacian flow, and
  extension assuming bounded scalar curvature},
        date={2018},
        ISSN={0012-7094},
     journal={Duke Math. J.},
      volume={167},
      number={18},
       pages={3533\ndash 3589},
  url={https://doi-org.proxy3.library.mcgill.ca/10.1215/00127094-2018-0040},
      review={\MR{3881202}},
}

\bib{FiRa217}{article}{
      author={{Fino}, Anna},
      author={{Raffero}, Alberto},
       title={{Closed G$_2$-structures on non-solvable Lie groups}},
        date={2017-12},
     journal={arXiv e-prints},
       pages={arXiv:1712.09664},
      eprint={1712.09664},
}

\bib{FiRa119}{article}{
      author={Fino, Anna},
      author={Raffero, Alberto},
       title={{Closed warped G2-structures evolving under the Laplacian flow}},
        date={2020},
     journal={Ann. Sc. Norm. Super. Pisa Cl. Sci. (5)},
      volume={20},
      number={1},
       pages={315\ndash 348},
}

\bib{FiRafLap18}{article}{
      author={{Fino}, Anna},
      author={{Raffero}, Alberto},
       title={{A class of eternal solutions to the {G$_2$}-Laplacian flow}},
        date={\noop{3001}in press},
     journal={J Geom Anal},
}

\bib{HuWa18}{article}{
      author={Huang, Hongnian},
      author={Wang, Yuanqi},
      author={Yao, Chengjian},
       title={Cohomogeneity-one {$G_2$}-{L}aplacian flow on the 7-torus},
        date={2018},
        ISSN={0024-6107},
     journal={J. Lond. Math. Soc. (2)},
      volume={98},
      number={2},
       pages={349\ndash 368},
         url={https://doi-org.proxy3.library.mcgill.ca/10.1112/jlms.12137},
      review={\MR{3873112}},
}

\bib{IvLaSecond}{book}{
      author={Ivey, Thomas~A.},
      author={Landsberg, Joseph~M.},
       title={Cartan for beginners},
      series={Graduate Studies in Mathematics},
   publisher={American Mathematical Society, Providence, RI},
        date={2016},
      volume={175},
        ISBN={978-1-4704-0986-9},
         url={https://doi-org.proxy3.library.mcgill.ca/10.1090/gsm/175},
        note={Differential geometry via moving frames and exterior differential
  systems, Second edition [of MR2003610]},
      review={\MR{3586335}},
}

\bib{Joyce96}{article}{
      author={Joyce, Dominic~D.},
       title={Compact {R}iemannian {$7$}-manifolds with holonomy {$G_2$}. {I},
  {II}},
        date={1996},
        ISSN={0022-040X},
     journal={J. Differential Geom.},
      volume={43},
      number={2},
       pages={291\ndash 328, 329\ndash 375},
         url={http://projecteuclid.org/euclid.jdg/1214458109},
      review={\MR{1424428}},
}

\bib{KaLaERP20}{article}{
      author={{Kath}, Ines},
      author={{Lauret}, Jorge},
       title={{A new example of a compact ERP G2-structure}},
        date={2020-05},
     journal={arXiv e-prints},
       pages={arXiv:2005.02462},
      eprint={2005.02462},
}

\bib{LotLamLap19}{article}{
      author={{Lambert}, B.},
      author={{Lotay}, J.D.},
       title={{Spacelike Mean Curvature Flow}},
        date={2019},
     journal={J Geom Anal},
}

\bib{LauFirSol17}{article}{
      author={Lauret, Jorge},
       title={Laplacian flow of homogeneous {$G_2$}-structures and its
  solitons},
        date={2017},
        ISSN={0024-6115},
     journal={Proc. Lond. Math. Soc. (3)},
      volume={114},
      number={3},
       pages={527\ndash 560},
         url={https://doi-org.proxy3.library.mcgill.ca/10.1112/plms.12014},
      review={\MR{3653239}},
}

\bib{LauretLap}{article}{
      author={Lauret, Jorge},
       title={Laplacian solitons: questions and homogeneous examples},
        date={2017},
        ISSN={0926-2245},
     journal={Differential Geom. Appl.},
      volume={54},
      number={part B},
       pages={345\ndash 360},
         url={https://doi.org/10.1016/j.difgeo.2017.06.002},
      review={\MR{3693936}},
}

\bib{LauNicERPLICAG20}{article}{
      author={Lauret, Jorge},
      author={Nicolini, Marina},
       title={{Extremally Ricci pinched G2-structures on Lie groups}},
        date={\noop{3001}in press},
     journal={Comm. Anal. Geom.},
}

\bib{LauNicERPLI20}{article}{
      author={Lauret, Jorge},
      author={Nicolini, Marina},
       title={{The classification of ERP G2-structures on Lie groups}},
        date={\noop{3001}in press},
     journal={Ann. Mat. Pura Appl. (4)},
}

\bib{LotWeiLapShi}{article}{
      author={Lotay, Jason~D.},
      author={Wei, Yong},
       title={Laplacian flow for closed {${\rm G}_2$} structures: {S}hi-type
  estimates, uniqueness and compactness},
        date={2017},
        ISSN={1016-443X},
     journal={Geom. Funct. Anal.},
      volume={27},
      number={1},
       pages={165\ndash 233},
         url={https://doi.org/10.1007/s00039-017-0395-x},
      review={\MR{3613456}},
}

\bib{LotWeiAna}{article}{
      author={Lotay, Jason~D.},
      author={Wei, Yong},
       title={Laplacian flow for closed {$\rm G_2$} structures: real
  analyticity},
        date={2019},
        ISSN={1019-8385},
     journal={Comm. Anal. Geom.},
      volume={27},
      number={1},
       pages={73\ndash 109},
  url={https://doi-org.proxy3.library.mcgill.ca/10.4310/CAG.2019.v27.n1.a3},
      review={\MR{3951021}},
}

\bib{LotWeiStab}{article}{
      author={Lotay, Jason~D.},
      author={Wei, Yong},
       title={Stability of torsion-free {$\rm G_2$} structures along the
  {L}aplacian flow},
        date={2019},
        ISSN={0022-040X},
     journal={J. Differential Geom.},
      volume={111},
      number={3},
       pages={495\ndash 526},
         url={https://doi-org.proxy3.library.mcgill.ca/10.4310/jdg/1552442608},
      review={\MR{3934598}},
}

\bib{MorrisArith}{book}{
      author={Morris, Dave~Witte},
       title={Introduction to arithmetic groups},
   publisher={Deductive Press},
        date={2015},
        ISBN={978-0-9865716-0-2; 978-0-9865716-1-9},
      review={\MR{3307755}},
}

\bib{NicShrinkSol20}{article}{
      author={{Nicolini}, Marina},
       title={{New examples of shrinking Laplacian solitons}},
        date={2020-06},
     journal={arXiv e-prints},
       pages={arXiv:2006.13074},
      eprint={2006.13074},
}

\end{biblist}
\end{bibdiv}

\end{document}